\documentclass{article}

\usepackage{microtype}
\usepackage{graphicx}
\usepackage{booktabs} % for professional tables

\usepackage{hyperref}
\usepackage{xcolor}  
\usepackage{bm}
\usepackage{comment}
\usepackage[thinc]{esdiff}
\usepackage{float}
\usepackage{flafter}
\usepackage{subcaption}
\usepackage{bbm}
\usepackage{multirow}
\usepackage[accepted]{icml2024}

\usepackage{amsmath}
\usepackage{amssymb}
\usepackage{mathtools}
\usepackage{amsthm}
\usepackage{enumitem}
\usepackage{mathtools}
\usepackage{graphicx}
\usepackage[capitalize,noabbrev]{cleveref}

\theoremstyle{plain}
\newtheorem{theorem}{Theorem}[section]
\newtheorem{thm}[theorem]{Theorem}

\newtheorem{lemma}[theorem]{Lemma}
\newtheorem{corollary}[theorem]{Corollary}
\theoremstyle{definition}

\newtheorem{assumption}[theorem]{Assumption}
\theoremstyle{remark}

\usepackage[textsize=tiny]{todonotes}

\DeclareMathOperator*{\argmin}{argmin}

\newcommand{\tp}{\intercal}

\newcommand{\bigO}{\ensuremath{\mathop{}\mathopen{}\mathcal{O}\mathopen{}}}
\newcommand{\smallO}{ \scalebox{0.7}{$\mathcal{O}$}}
\newcommand{\bigOp}{\bigO_\mathrm{p}}

\newcommand{\given}{{\mid}}

\newcommand{\EE}{{\mathbb{E}}}

\newcommand{\calA}{{\mathcal{A}}}

\newcommand{\calD}{{\mathcal{D}}}
\newcommand{\calE}{{\mathcal{E}}}

\newcommand{\calL}{{\mathcal{L}}}
\newcommand{\calM}{{\mathcal{M}}}
\newcommand{\calN}{{\mathcal{N}}}

\newcommand{\calP}{{\mathcal{P}}}
\newcommand{\calQ}{{\mathcal{Q}}}
\newcommand{\calR}{{\mathcal{R}}}
\newcommand{\calS}{{\mathcal{S}}}

\newcommand{\quest}[1]{\textbf{Q#1}}
\newcommand{\dist}{{\rho}}

\def\Var{\mathrm{Var}} % the symbol Var for covariance used the sans serif letter

\icmltitlerunning{A Fine-grained Analysis of Fitted Q-evaluation}

\begin{document}

\twocolumn[
%\icmltitle{Submission and Formatting Instructions for \\
%          International Conference on Machine Learning (ICML 2024)}
\icmltitle{A Fine-grained Analysis of Fitted Q-evaluation: Beyond Parametric Models} %placeholder
%\icmltitle{On the Horizon Dependence of Off-policy Fitted Q-evaluation}
%\icmltitle{Nonparametric Off-policy Fitted Q-evaluation: Sample Size and Horizon Dependence}

% It is OKAY to include author information, even for blind
% submissions: the style file will automatically remove it for you
% unless you've provided the [accepted] option to the icml2024
% package.

% List of affiliations: The first argument should be a (short)
% identifier you will use later to specify author affiliations
% Academic affiliations should list Department, University, City, Region, Country
% Industry affiliations should list Company, City, Region, Country

% You can specify symbols, otherwise they are numbered in order.
% Ideally, you should not use this facility. Affiliations will be numbered
% in order of appearance and this is the preferred way.
% \icmlsetsymbol{equal}{*}

\begin{icmlauthorlist}
\icmlauthor{Jiayi Wang}{yyy1}
\icmlauthor{Zhengling Qi}{yyy3}
\icmlauthor{Raymond K. W. Wong}{yyy2}
%\icmlauthor{}{sch}
%\icmlauthor{}{sch}
\end{icmlauthorlist}

\icmlaffiliation{yyy1}{Department of Mathematical Sciences, University of Texas at Dallas, Richardson, USA}
\icmlaffiliation{yyy2}{Department of Statistics, Texas A\&M University, College Station, USA}
\icmlaffiliation{yyy3}{School of Business, The George Washington University, Washington, D.C., USA}

\icmlcorrespondingauthor{Zhengling Qi}{qizhengling@email.gwu.edu}
% \icmlcorrespondingauthor{Firstname2 Lastname2}{first2.last2@www.uk}

% You may provide any keywords that you
% find helpful for describing your paper; these are used to populate
% the "keywords" metadata in the PDF but will not be shown in the document
\icmlkeywords{Machine Learning, ICML}

\vskip 0.3in
]

% this must go after the closing bracket ] following \twocolumn[ ...

% This command actually creates the footnote in the first column
% listing the affiliations and the copyright notice.
% The command takes one argument, which is text to display at the start of the footnote.
% The \icmlEqualContribution command is standard text for equal contribution.
% Remove it (just {}) if you do not need this facility.

\printAffiliationsAndNotice{}  % leave blank if no need to mention equal contribution
% \printAffiliationsAndNotice{\icmlEqualContribution} % otherwise use the standard text.

\begin{abstract}
In this paper, we delve into the statistical analysis of the fitted Q-evaluation (FQE) method, which focuses on estimating the value of a target policy using offline data generated by some behavior policy.
We provide a comprehensive theoretical understanding of FQE estimators under both parameteric and nonparametric models on the $Q$-function. Specifically, we address three key questions related to FQE that remain largely unexplored in the current literature: (1) Is the optimal convergence rate for estimating the policy value regarding the sample size $n$ ($n^{-1/2}$) achievable for FQE under a non-parametric model with a fixed horizon ($T$)? (2) How does the error bound depend on the horizon $T$? (3) What is the role of the probability ratio function in improving the convergence of FQE estimators?
Specifically, we show that under the completeness assumption of $Q$-functions, which is mild in the non-parametric setting, the estimation errors for policy value using both parametric and non-parametric FQE estimators can achieve an optimal rate in terms of $n$. The corresponding error bounds in terms of both $n$ and $T$ are also established. With an additional realizability assumption on ratio functions, the rate of estimation errors can be improved from $T^{1.5}/\sqrt{n}$ to $T/\sqrt{n}$, which matches the sharpest known bound in the current literature under the tabular setting.
\end{abstract}

\section{Introduction}
\label{sec:intro}

In reinforcement learning (RL), off-policy evaluation (OPE) is an important topic that focuses on estimating the expected total reward (e.g., the value defined in \eqref{def: integrated value fun})  of a target policy based on data collected from a potentially different and unknown policy \citep{sutton2018reinforcement}.
OPE is particularly useful in high-stakes domains where the implementation of a new policy can incur significant costs or risks, and has been extensively studied in RL \citep[e.g.][]{xie2019towards, duan2020minimax,yin2020asymptotically,chen2022well,ji2022sample,wang2023projected}. See \citet{uehara2022review} for an overview of the OPE methods.
Among various algorithms for OPE, fitted Q-evaluation (FQE) is arguably one of the most popular algorithms. In FQE, $Q$-functions (defined in \eqref{def: Q fun}) are estimated in a backward manner using supervised learning methods, and their estimates are then used to construct an estimated policy value (see \eqref{eqn:FQE-value}).  FQE has demonstrated significant empirical success in many applications \citep{fu2021benchmarks, voloshin2019empirical}. Its popularity and success have also led to significant theoretical interest in FQE.
Several recent studies aim to provide theoretical justifications for its effectiveness (as discussed in detail in Section \ref{sec:literature}).
In this work, we delve deeply into the analysis of FQE estimators within the framework of a finite-horizon, time-inhomogeneous Markov Decision Process (MDP).
Compared to existing analysis in FQE, our objective is to provide a more comprehensive understanding of the convergence rate on the error bound of estimating the value of a target policy under both parametric and nonparametric models
in terms of both the number of episodes $n$ and the horizon $T$. Specifically, we seek to address the following three fundamental questions related to FQE:
\vspace*{-2mm}
\begin{itemize}[leftmargin=*]
  \item \quest{1}: For the fixed horizon $T$, how does the convergence rate depend on the number of episodes $n$ given the completeness assumption for $Q$-functions? Is the optimal convergence rate ($n^{-1/2}$) still achievable under nonparametric models of $Q$-functions?
    \item \quest{2}: 
    How does the convergence rate depend on the growing horizon $T$?
    %What is the dependence of the convergence rate with respect to a growing horizon $T$?
    \item \quest{3}: What is the role of the probability ratio functions $w_t^\pi$ (defined in \eqref{eqn:def_w}) in improving the  convergence rate for FQE estimators?
\end{itemize}
\vspace*{-2mm}
We will comment on these questions and the existing progress towards addressing them in the following subsection.

\subsection{Related Literature}
\label{sec:literature}
In recent years, many works have studied FQE from a theoretical perspective with the goal to address \quest{1} and \quest{2}, with varying degrees of success under different modeling assumptions. Below, we survey these efforts and indicate that some important understanding is still lacking.

The first line of research focuses on studying FQE under some parametric model on $Q$-functions.  For example, \citet{duan2020minimax} assumed linear MDP and established a $T^2/\sqrt{n}$ order for the estimation error of the policy value. See Theorem 2 in \citet{duan2020minimax}.  
Furthermore, \citet{zhang2022off} studied beyond linear MDP and their analysis allows that $Q$-functions lie in an almost arbitrarily parametrized function class with some differentiability condition.  They obtained the same order for the estimation error as \citet{duan2020minimax}. In those aforementioned works, they showed that the optimal convergence rate with respect to $n$ can be achieved by using FQE under a parametric model of $Q$-functions. Regarding \quest{2}, they obtain a quadratic dependence with respect to $T$. However, as shown in \citet{yin2020asymptotically}, a linear and thus better dependence of $T$ can be achieved under the tabular setting. 
%It is unclear in the current literature if this gap in horizon dependence can be closed.
\textit{This leads to an interesting question that whether linear horizon dependence can be obtained beyond tabular setting.}

The second line of research considers FQE under some nonparametric models of $Q$-function. Specifically,
\citet{nguyen2021sample} provided an error bound for the estimation error of nonparametric FQE using feed-forward ReLU network. They showed that the estimation error is of an order $T^{2-\alpha/(2\alpha+2D)}n^{-\alpha/(2\alpha+2D)}$, where $\alpha$ is a smoothness parameter and $D$ is the dimension of state and action. 
Compared to \citet{nguyen2021sample}, \citet{ji2022sample} further assumed a low-dimensional manifold structure in convolutional neural networks and improved the error bound to an order of $T^2n^{-\alpha/(2\alpha+d)}$, where $d$ is the intrinsic dimension of the state-action space.
Regarding \quest{1}, these two works are only able to show a slower convergence rate of estimating the policy value than the optimal $n^{-1/2}$ with fixed  $T$. \textit{Hence, it is unclear if $n^{-1/2}$ rate is achievable for nonparametric FQE based on the current literature.}
We will provide a positive result to this question in this paper.
For \quest{2}, \citet{nguyen2021sample} showed the horizon dependence is of an order that is larger than $T^{1.5}$  and \citet{ji2022sample}  showed a quadratic dependence. \textit{Similarly to the parametric methods, it is unclear if the linear dependence of the horizon can be obtained under non-parametric models.}

%They also show that the error bound has a polynomial dependence with respect to $T$.

%\ray{It might be good to mentiont there is a equivalence that like in Section 3.2 of \cite{duan2020minimax}. Indeed, we explain this equivalence around/after Section 2.2. This might allow us to transit smoothly to MIS estimator and the discussion of ratio functions and \quest{3}; we need to have some discussion about \quest{3}, and get the readers excited about this question.}

% In Section 3.3 in \cite{duan2020minimax}, they establish a connection between FQE estimators and marginalized importance sampling (MIS) estimators (refer to the form in \eqref{eqn:MIS}).  In the literature, 
% there are several works focuses on studying 
% marginal importance sampling (or probability ratio functions) to tackle distribution mismatch arising from differences between the target policy and the behavior policy. 
Besides FQE, another line of research uses marginal importance sampling (MIS) or probability ratio
functions to address the distributional mismatch due to the difference between the target policy and
the behavior policy and develop an MIS  estimator for OPE.
In particular, for time-inhomogeneous settings and for tabular MDP, \citet{xie2019towards}  and \citet{yin2020asymptotically} showed that the estimation error of the MIS estimator has an order of $T/\sqrt{n}$ under some proper assumptions. To the best of our knowledge, this is the sharpest bound with respect to both $n$ and $T$ in the existing literature.  
Building on the insights discussed in Section 3.3 of \citet{duan2020minimax}, a connection between FQE estimators and marginalized importance sampling (MIS) estimators (refer to the forms in \eqref{eqn:MIS} and \eqref{eqn:weights}) can be established.
Given this connection, one naturally wonders
if the successful analysis of MIS estimator (under tabular setting) can be leveraged to understand the horizon dependence for the convergence rate of FQE \textit{beyond tabular setting}.
Specifically, we ask the following question: 
is the linear dependence of $T$ in the convergence rate of FQE estimators achievable due to the connection with MIS estimator?
Moreover, in the context of a continuous state space with some nonparametric model, we further ask: whether any conditions for the marginalized sampling weights (or the probability ratio function) need to be imposed in order to achieve such sharp dependence of $T$ and how such conditions contribute to addressing the convergence of FQE estimators (\quest{3}).
We will address all these three intriguing but rarely studied questions, contributing to improving the understanding of FQE.
%This work is dedicated to providing answers to these intriguing questions.

%{Some double robust estimations have been developed to combine the estimations of both $Q$-functions and probability ratio functions \citep{jiang2016doubly,kallus2020double} and achieve better convergence results. These estimators usually concerns with the infinite-horizon and time-homogeneous setting where the probability ratio function is easier to estimate.}
%\ZL{we probably need to explain why we do not analyze this estimator if we talk about the doubly robust estimator}

% \onecolumn
\begin{table*}[t]
\caption{Comparison on the error bound for the first-order term in existing works. $\kappa$ is defined in  \eqref{eqn:kappa}.
$\tilde{\kappa}$ is the upper bound for the probability ratio functions; $D$ is the dimension of space and action. $d$ is the intrinsic dimension of the state-action space. Some logarithmic orders are omitted in the error bounds.}
\label{sample-table}
\vskip 0.15in
\begin{center}
\begin{small}
\begin{sc}
\begin{tabular}{l cccr}
\toprule
Work & Parametric? & Regularity on $Q$  & Error Bound (W.H.P) \\
\midrule
\citet{yin2020asymptotically}    & $\surd$ & Tabular & $\bigO( T \tilde{\kappa} \sqrt{1/n}  )$  \\
\hline \\
\citet{duan2020minimax} & $\surd$ & Linear & $\bigO( T^2\sqrt{ \kappa /n} )$ \\
\hline \\
\citet{zhang2022off} & $\surd$ &  Differentiable & $\bigO( T^2\sqrt{\kappa /n} )$\\
\hline \\
\citet{nguyen2021sample} & $\times$ & Besov  &  $\bigO(T^{2-\alpha/(2\alpha+2D)} \tilde{\kappa}n^{-\alpha/(2\alpha+2D)})$       \\
\hline \\
\citet{ji2022sample}     & $\times$ & Besov & $\bigO(T^2 \kappa n^{-\alpha/(2\alpha+d)})$\\
\hline \\
\multirow{2}{*}{\textbf{Our Work} in Section \ref{sec:para} }  &  \multirow{2}{*}{$\surd$}   &  \multirow{2}{*}{Linear} & $\bigO(T^{1.5} \sqrt{\kappa/n})$\\
 &    &    & $\bigO(T \tilde{\kappa} \sqrt{1/n})$ when $w_t^\pi$ are linear\\
 \hline\\
 \multirow{2}{*}{\textbf{Our Work} in Section \ref{sec:non_para}}   & \multirow{2}{*}{$\times$} & \multirow{2}{*}{H{\"o}lder} &  $\bigO(T^{1.5} \sqrt{\kappa/n})$ when $Q_t^\pi$ are smooth enough \\
 &  & & $\bigO(T \tilde{\kappa} \sqrt{1/n})$ when $w_t^\pi$ are H{\"o}lder\\
\bottomrule
\end{tabular}
\end{sc}
\end{small}
\end{center}
\vskip -0.1in
\end{table*}
% \twocolumn

\subsection{Our Contributions} 

% \ZL{Is our work considered as the first to study when $T$ grows with n?}

% \ZL{I am confused by the $O_p$ notation.}

% \jiayi{Maybe move to the place after the questions?}
In this work, we focus on addressing the three aforementioned questions in the context of FQE where $Q$-functions are estimated under either a parametric linear model, or a non-parametric model via linear sieves \citep{ai2003efficient}. In particular, we successfully establish the following results:
\vspace*{-2mm}
\begin{enumerate}[leftmargin=*]
    \item For fixed $T$,  FQE estimators with $Q$-functions modeled parametrically is shown to achieve the optimal convergence rate with respect to $n$ ($n^{-1/2}$).   When $Q$-functions are modeled nonparametrically, $n^{-1/2}$ rate can be still obtained when $Q$-functions satisfy certain smoothness conditions.  In contrast, under these conditions, the estimation for $Q$-functions only has a convergence rate slower than $n^{-1/2}$.
    \item For asymptotically growing $T$, the first-order term of the error bound is shown to have a dependence of $T^{3/2}$ while the higher order term has a dependence of $T^3$ for both parametric and nonparametric FQE. Our bound has a milder dependence with respect to $T$ for the first-order term compared to current literature, where their bounds have a quadratic dependence of $T$.
    \item 
    % When the space of probability ratio functions can be approximated well by the modeling space for $Q$-functions,  
    When probability ratio functions lie in a space of smooth functions,
    %\textit{without incorporating any additional estimating procedures for these probability ratio functions like those double robust estimators},
    \textit{without additionally estimating these probability ratio functions (like those double robust estimators)},
    the first-order term of the error bound for the vanilla FQE estimators is shown to converge with an order of $T/\sqrt{n}$ in both parametric and nonparametric settings. This bound matches with the sharpest rate of convergence in the tabular setting.
\end{enumerate}
\vspace*{-2mm}

%what if nonlinear modeling? 

\section{Set Up}
\label{sec:setup}
To set the stage for our theoretical discussion, we review the framework of discrete-time inhomogeneous Markov Decision Processes (MDP) and the fitted Q-evaluation (FQE) for estimating policy value in this section.

\subsection{Preliminary}
\label{sec: preliminary}
Denote by $\calM = ( T,  \tilde{\mathcal{S}},  \tilde \calA, 
\tilde{\mathrm{Pr}}, \tilde{\calR})$ a finite-horizon episodic  Markov Decision Process (MDP),  
%an episodic stochastic process, 
where the integer $T$ %defined as 
is the length of horizon, $\tilde {\mathcal{S}} =\{\mathcal{S}_t\}_{t=1}^T$ and $\tilde {\calA }= \{\calA_t\}_{t=1}^T$ are the state spaces and the action spaces across $T$ decision points respectively, $\tilde{\mathrm{Pr}} = \{\mathrm{Pr}_t\}_{t=1}^T$ with $\mathrm{Pr}_t(\bullet \mid s,a)$ representing the transition kernel (probability) %from time $t$ to the next time point $t+1$, 
given the state $s \in \calS_t$ and the action $a \in \calA_t$,
and $\calR = \{R_t\}_{t=1}^T$ are immediate rewards such that $R_t \given S_t = s, A_t = a \sim R_t(s,a)$. %with conditional distribution given $S_t =$
%with the random scalar reward $r \sim R_t(s,a)$ at every step $t$. %We assume that there exists reward functions $\{r_t\}_{t=1}^T$ where $r_t \in \mathbb{R}^{\mathcal{S}_t\times \mathcal{A}_t}$, such that\\
% $
% \EE_t\left[R_t \given A_t = a, S_t = s, \left\{S_j, A_j, R_j\right\}_{0 \leq j < t}\right] = r_t(s, a),
% $
We take $r_t(s,a)$ as the conditional mean of $R_t(s,a)$.
A trajectory generated from $\calM$ can be written as $\{S_t, A_t, R_t\}_{t=1}^T$, where $S_t\in\calS_t$, %corresponds to the state variables observed at time $t$,
$A_t\in\calA_t$ and $R_t\in \mathbb{R}$ denote the state, the action and the reward at time $t$ respectively. 
%  Without loss of generality, we assume $\calS_1=\cdots = \calS_T = \calS$ and $\calA_1=\cdots = \calA_T = \calA$.
Without loss of generality, we assume $|R_t| \leq 1$ for $t=1,\dots, T$.
In the following discussion, for the sake of simplicity, we assume the same state spaces and action spaces across all decision points, denoted by $\mathcal{S} \subset \mathbb{R}^d$ and $\mathcal{A}$, respectively. Additionally, we assume the action space $\mathcal{A}$ is finite.

A policy is defined as a way of choosing actions at each decision time point $t$. More specifically, denote a target policy as $\pi = \{\pi_t\}_{t=1}^T$, where $\pi_t$ is a function mapping from the state space $\calS$ to a probability mass function over the action space $\calA$.  
%In the offline RL setting, one of the fundamental tasks is to estimate a target policy's expected value function based on the pre-collected (batch) data. Given a policy $\pi$, the value function is defined as
Then OPE aims to estimate the value of $\pi$ defined as
% \begin{align}\label{def: value fun}
% 	V^\pi = \EE^\pi\left[\sum_{t = 1}^{T} R_t \given S_1 = s\right],
% \end{align}
\begin{align}\label{def: integrated value fun}
	\nu(\pi) =  \EE^\pi\left[\sum_{t = 1}^{T} R_t\right],
\end{align}
% , i.e., the expectation of the value function, defined as
using the pre-collected data generated by a fixed stationary policy $\pi^b = \{\pi^b_t\}_{t=1}^T$, which is called behavior policy. Here $\EE^\pi$ denotes the expectation with respect to the distribution whose actions are generated by the target policy $\pi$.
% where $\mathbb G$ denotes a reference distribution over $\calS$. In RL literature, $\mathbb G$ is typically assumed \textit{known} or is commonly adopted as the distribution of $S_1$. 
We further assume that the pre-collected training data consist of $n$ independent and identically distributed  
trajectories as $\calD_n = \left\{ \left\{\left(S_{i, t}, A_{i, t}, R_{i, t}\right)\right\}_{1 \leq t < {T}} \right\}_{1 \leq i \leq n}$. %, denoted by
% $$
% \calD_n = \left\{ \left\{\left(S_{i, t}, A_{i, t}, R_{i, t}\right)\right\}_{1 \leq t < {T}} \right\}_{1 \leq i \leq n},
% $$
%We assume $\calD_n$ is generated by a fixed stationary policy $\pi^b = \{\pi^b_t\}_{t=1}^T$. 
% We  make
% the following assumption on the data generating mechanism.
% \begin{assumption}\label{ass: DGP}
% 	The training data $\calD_n$ is generated by a fixed stationary policy $\pi^b = \{\pi^b_t\}_{t=1}^T$.
% \end{assumption}
% Under Assumptions \ref{ass: Markovian} and \ref{ass: DGP}, $\left\{S_t, A_t\right\}_{t\geq 0}$ forms a discrete time time-inhomogeneous Markov chain. 
%In the literature, $\pi^b$ is usually called the behavior policy, which may not be known.
For convenience, we omit $\pi^b$ in the superscript in the notation of the expectation and probability under the distribution induced by $\pi^b$. 

Next we define notation for several important probability distributions.
 For any $t=1,\dots, T$, we take $\dist_t^\pi(s, a)$ and $\dist_t^{b}(s,a)$ as the marginal density of $(S_t, A_t)$ at $(s, a) \in \calS \times \calA$ under the target policy $\pi$ and behavior policy $\pi^b$ respectively. And we define the probability ratio function $w^\pi_t$ as 
\begin{align}
	\label{eqn:def_w}
	w^\pi_t(s,a) = {\dist_t^\pi(s,a)}/{\dist_t^b(s,a)}, 	
\end{align}
for $ t=1,\dots, T$. % and denote $\mathcal{E}_t^\pi f = \int_{(s,a) \sim d_t^\pi} f(s,a) d(s,a)$.
 \subsection{Fitted Q-evaluation}
 \label{sec:FQE}
% A lot of existing methods
% directly estimate 
%  \citep[e.g.,][]{luckett2019estimating,shi2020statistical}
One can find $\nu(\pi)$ by computing the state-action value functions (also known as the $Q$-functions)
 defined as
 \begin{align}\label{def: Q fun}
  Q^\pi_t(s, a) = \EE^\pi\left[\sum_{t' = t}^{T}R_t \given S_t = s, A_t = a\right],
 \end{align}
for $t= 1,\dots, T$. Then $\nu(\pi) = \int_{(s,a) \in \calS \times \calA} Q_1^\pi(s,a) \dist^\pi_1(s,a) d(s,a)$. For simplicity, we assume the initial state distribution $\dist^\pi_1$ is known.
%  As we can see from \eqref{def: value fun} and \eqref{def: integrated value fun},
%  \begin{equation}\label{eq: value function}
%  \calV(\pi)= (1- \gamma)\int_{s\in \mathcal{S}}\sum_{a \in \calA} \pi(a \mid s) Q^\pi(s, a)\mathbb G(ds).
%  \end{equation}
 To compute the $Q$-function, it is well-known that $\{Q^\pi_t\}_{t=1}^T$ satisfy the following Bellman equation
 \begin{align}\label{eq: Bellman equation for Q}
     Q^\pi_t(s, a) = \EE\left[R_t +  V^\pi_{t+1}(S_{t+1}) \given S_t = s, A_t = a\right], 
 \end{align}
 %for $ t=1,\dots, T$, 
 where $V^\pi_t(s) =  \sum_{a \in \calA} \pi_t(a' \given s)Q^\pi_t(s, a')$.

Motivated by \eqref{eq: Bellman equation for Q}, in fitted Q-evaluation (FQE), one can utilize the offline data $\calD_n$ and recursively apply a regression technique to learn ${Q}^\pi_T, {Q}^\pi_{T-1}, \dots, {Q}^\pi_1$ in a sequential and backward order. More specifically, let $\hat{Q}^\pi_{T+1} = 0$, and for $t = T, T-1,\dots, 1$, one can compute
\begin{multline}
    \label{eqn:Qest}
    \hat{Q}^\pi_{t} = \argmin_{Q \in \calQ^{(t)}}  \frac{1}{n}\sum_{i=1}^n \Biggl\{ Q(S_{i,t}, A_{i,t}) - \\
	 \biggl[ R_{i,t}+\sum_{a'\in \calA_{t+1}} \pi_t(a\mid S_{i,t+1})\hat{Q}^\pi_{t+1} (S_{i,t+1}, a')\biggr] \Biggr\}^2
\end{multline}
for $Q_t^\pi$,
where $\calQ^{(t)}$ is a hypothesis class for $Q^\pi_t$.
Then the policy value is estimated via a plug-in estimator:
\begin{align}
\label{eqn:FQE-value}
    \hat{\nu}(\pi) = \int_{(s,a) \in \calS \times \calA} \hat{Q}_1^\pi(s,a) \dist^\pi_1(s,a) d(s,a).
\end{align}
% \ZL{I am confused by the above notation}

When considering parametric models, one can use a linear model to approximate $Q_t^\pi$ for every $t$ \citep[e.g.][]{duan2020minimax, min2021variance}.
For example, take $\Psi_{K}(s)=[ \psi_{1}(s),\cdots,\psi_{K}(s)]^\tp$ as a vector consisting of $K$ features for $s \in \calS $.
Let $\phi_{K}(s,a) = [\mathbbm{1}(a=1)\Psi_{K}(s)^\tp ,  \dots, \mathbbm{1}(a=|\mathcal{A}|)\Psi_{K}(s)^\tp ]^\tp\in\mathbb{R}^{K|A|}$. Then 
%$\calQ^{(t)}$ is a linear space such that
one can let $\calQ^{(t)} = \{\phi_K(\cdot)^\tp \beta : \beta \in \mathbb{R}^{K |\calA|}\}$. 
% In this paper, we consider the nonparametric modeling for Q-functions $Q^\pi_t$ and propose to estimate $Q^\pi_t$ based on linear sieves.  
%which takes the form
% $ \hat{Q}^\pi_t(s,a) = \Phi^{(t)}_K(s)^\tp \beta_{a}$ for some $\beta \in \mathbb{R}^K$.
% Take $\Psi^{(t)}_{K_t}(\cdot)=[ \psi^{(t)}_{1}(\cdot),\cdots,\psi^{(t)}_{K_t}(\cdot)]^\tp$ as a vector consisting of $K$ sieve basis functions at time point $t$, where $K_t$ is allowed to depend on $n$ and $T$.  One can take 
%  splines or wavelet bases (see for example, \cite{chen2018optimal}) for choices of basis functions.  Define $\phi^{(t)}_{K_t}(s,a) = [\Psi^{(t)}_{K_t}(s)^\tp \mathbbm{1}(a=1),  \dots, \Psi^{(t)}_{K_t}(s)^\tp \mathbbm{1}(a=|\mathcal{A}_t|)]^\tp$.
%Under well-specification assumptions (i.e., there exists $\beta_t$ such that $Q_t^\pi(\cdot) = \phi_k(\cdot)^\tp \beta_t$ for all $t$), \eqref{eqn:Qest} is equivalent to estimating $\beta_{t} \in \mathbb{R}^{K|\calA|}$ for  $t=1,\dots, T$, where 

% Clearly, with the linear approximation, we can estimate $Q^\pi_t$ via the estimation of the coefficients:
% \begin{multline}
%     \label{eqn:Qest-sieve}
% \hat{\beta}_{t} = \argmin_{\beta \in \mathbb{R}^{K |\calA|}}  \frac{1}{n}\sum_{i=1}^n \Biggl\{ [\phi_{K}(S_{i,t}, A_{i,t}) ]^\tp \beta \\
%   \qquad- \Biggl[ R_{i,t}+\sum_{a'\in \calA} \pi_t(a\mid S_{i,t+1})\hat{Q}^\pi_{t+1} (S_{i,t+1}, a')\Biggr] \Biggr\}^2,  \\
% \text{with} \quad \hat{Q}_t^\pi(s,a) = [\phi_{K}(s,a) ]^\tp \hat{\beta}_{t}.
% \end{multline}

Many existing work study the theoretical property of FQE under the linear model assumption such as $Q_t^\pi \in \mathcal{Q}^{(t)}$ (realizability), and $K$ is fixed \textit{independent of $n$ and $T$}. 
We call this setup the \textit{parametric} setting,
as each $Q^\pi_t$ is modeled via a finite number of parameters.
%We take such linear modeling as an example of the parametric case to be studied.
The realizability of such linear modeling hedges strongly on the careful selection of the features, which is often non-trivial.
To relax this assumption, an infinite-dimensional modeling for $Q^\pi_t$ can be used, which we refer to as a \textit{nonparametric} setting. 
In this work, we assume that for every $a \in \calA$, $Q^\pi_t(\cdot, a)$ lies in a H{\"o}lder space which, roughly speaking, consists of functions $g: \calS \subset \mathbb{R}^d \rightarrow \mathbb{R}$ with H{\"o}lder continuous derivatives of certain order.
Specifically, a H{\"o}lder space is defined as
\begin{multline}
    \Lambda_{\infty}(p,L) = \left\{ g : \sup_{0\leq \|\alpha\|_1 \leq \lfloor p \rfloor} \left\| \partial^\alpha g \right\|_{\infty}  \leq L, \right. \\ \left. \sup_{\alpha:\|\alpha\|_1 \leq \lfloor p \rfloor} \sup_{x,y\in \calS, x\neq y } \frac{|\partial^\alpha g(x) - \partial^\alpha g(y)|}{\|x-y\|_2^{p
    -\lfloor p \rfloor}} \leq L \right\}, \label{eqn:def_holder}   
\end{multline}
where $\|\cdot\|_1$ and $\|\cdot\|_2$ are the $\ell_1$ and $\ell_2$ norm of a vector respectively, $\lfloor p \rfloor$ denotes the integer less or equal to $p$ for $p>0$,  $\alpha = (\alpha_1,\cdots, \alpha_d)$ is a non-negative vector, 
\begin{align*}
    {\partial}^\alpha g (x) = \frac{{\partial}^{\alpha}  g(x) }{\partial^{\alpha_1}\cdots \partial^{\alpha_d}}. 
\end{align*}
%As for estimation, we utilize linear sieves to approximate H{\"o}lder spaces for finite sample.
A standard approach for nonparametric regression is
sieve estimation. 
We study a linear-sieve estimator in our theoretical study of nonparametric FQE, where an increasing number of basis functions are allowed to approximate $Q_t^\pi$ in the H{\"o}lder space, as sample size increases.
%There are many different approaches to formulate valid estimation under nonparametric settings.
% In this work, we focus on the linear sieve approach
% %Built upon this, we use linear sieves as a setup
% for our theoretical study of nonparametric FQE.
%understanding nonparametric modeling.
To be specific, with slight abuse of notation,  take $\Psi_{K}(\cdot)=[ \psi_{1}(\cdot),\cdots,\psi_{K}(\cdot)]^\tp$ as a vector consisting of $K$ sieve basis functions at time point $t$, where the number of basis functions \textit{$K$ is allowed to depend on $n$ and $T$}.  One can take 
 splines or wavelet bases (see for example, \citet{huang1998projection} and \citet{chen2018optimal}) for choices of basis functions.  Take $\phi_{K}(s,a) = [\Psi_{K}(s)^\tp \mathbbm{1}(a=1),  \dots, \Psi_{K}(s)^\tp \mathbbm{1}(a=|\mathcal{A}|)]^\tp$. 
 %Then one can estimate the $Q$ functions via $\hat{Q}^\pi_t = \phi_K(s,a)^\tp \hat{\beta}_t$ with some estimated coefficients $\hat{\beta}_t \in \mathbb{R}^{K |\calA|}$. %Then \eqref{eqn:Qest-sieve} can be still applied to obtain the estimator for $Q_t^\pi$. 
 We approximate $\calQ^{(t)}$ with space $\tilde{\calQ}^{(t)}:= \mathrm{span}\{(\phi_K(\cdot, \cdot)\}$ and solve for \eqref{eqn:Qest} with ${\calQ}^{(t)}$ replaced by  $\tilde{\calQ}^{(t)}$.
 In practice, one can choose different types and numbers of basis functions at different time points. To simplify the notation, we use a universal set of basis functions and a universal number of basis functions.

 Sieve estimations have been extensively studied in statistics and econometric communities with their appealing empirical performance and ease of computation \citep[e.g.][]{geman1982nonparametric, huang1998projection,ai2003efficient,chen2018optimal}. Some recent works in OPE also utilize the linear sieves as a tool to study the nonparametric estimation of $Q$-functions \citep[e.g.][]{shi2022off,chen2022well,wang2023projected}. 
 Many sieve bases can effectively approximate infinite-dimensional spaces that contain a wide range of smooth functions.  For example, (tensor-product) B-spline basis and wavelet basis can approximate H{\"o}lder space well, with the approximation error decreasing as the number of basis functions increases.  See Section 2.2 of \citet{huang1998projection} for detailed discussions on the approximation power of these bases.

 %\ZL{Based on you two discussion, I suggest we have a paragraph to explain how general and flexible linear sieve is as for educational purpose}
%\ray{Note that we might encounter reviewers from a very ML background (especially younger-generation) thinking that linear sieves are not interesting. I suggest that we describe the linear function approximation (parametric) first, which sounds more natural to RL people in ML community despite its strong assumption; it would then sound more acceptable to them when we talk about linear sieves as a setup for understanding non-parametric modeling.}

\subsection{Connection with Marginal Importance Sampling Estimator}
%\ray{Shall we establish the linkage between FQE and MIS here like Section 3.2 of \cite{duan2020minimax}? I think this would be useful for people to see the dual role of ratio and makes our discussion of the well-specification of ratios (\quest{3}) more interesting.}

% \jiayi{Add notations here.}
In this section, we connect FQE with a marginal importance sampling (MIS) estimator, which will help us understand the role of ratio functions in our theoretical analysis in the later sections. To proceed, we introduce some notation. Unless specified otherwise, the following notations apply to both parametric and nonparametric cases. 
Let  
$\Sigma_t = \EE [\phi_{K}(S_t, A_t) \phi_{K}(S_t, A_t)^\tp]\in \mathbb{R}^{K|\calA| \times K|\calA|}$, 
$\hat{\Sigma}_t = \frac{1}{n}\sum_{i=1}^n [\phi_{K}(S_{i,t}, A_{i,t}) \phi_{K}(S_{i,t}, A_{i,t})^\tp]\in \mathbb{R}^{K|\calA| \times K|\calA|}$ and $\Sigma_{t,a} = \EE [\psi_{K}(S_t) \psi_{K}(S_t)^\tp\mid A_t = a] \in \mathbb{R}^{K\times K}$.  %\ray{$\Psi_K$? $\mathbb{R}^{K\times K}$?}.

Define $\mathcal{P}^\pi_t$ and  $\hat {\mathcal{P}}_t^\pi $ as the population and estimated conditional expectation operators respectively, such that 
% $$\hat{\Sigma}_{t,a} = \frac{1}{\sum_{i=1}^n \mathbbm{1}(A_{i,t}=a)}\sum_{i=1}^n [\psi_K(S_{i,t}) \psi^\tp_K(S_{i,t})\mathbbm{1}(A_{i,t}=a) ]$$
% And we take
\begin{align*}
&
\scalebox{0.89}{$
   (\mathcal{P}_t^\pi  f)(s,a) 
	= \EE\left\{ \sum_{a'} \pi_t(a'\mid S_{t+1}) f(S_{t+1}, a') \mid S_t = s, A_t = a\right\}, $} 
    %= \hat{\EE}\left\{ \sum_{a'} \pi(a'\mid S_{h+1}) f(S_{h+1}, a') \mid S_h = s, A_h = a\right\}\\
% \end{align*}
% \begin{align*}
\\
& (\hat {\mathcal{P}}_t^\pi  f)(s,a) 
    = \phi_{K}(s,a)^\tp (\hat{\Sigma}_t)^{-1} \\ & 
    \scalebox{0.9}{ \qquad   $\left( \frac{1}{n}\sum_{i=1}^n \phi_{K}(S_{i,t}, A_{i,t}) \left[ \sum_{a'} \pi_t(a'\mid S_{i,h+1}) f(S_{i,t+1}, a') \right]\right)$},
\end{align*}
% \ray{Part of these equation is displayed in the margin}
for $f \in \calQ^{(t+1)}$, $t=1,\dots, T-1$.  

% In addition, we define ${\Pi}_t$ and  $\hat {{\Pi}}_t $ as the population and estimated projection operators respectively, such that 
% \begin{align*}
%    & \Pi_t g (s,a) \\
%    & \quad =   \phi_{K}(s,a)^\tp ({\Sigma}_t)^{-1} \EE \left[ \phi_{K}(S_{t}, A_{t}) g(S_{t}, A_{t})\right], \\
%    & (\hat \Pi_t  g)(s,a) %= \hat{\EE}\left\{  f(S_{h}, A_h) \mid S_h = s, A_h = a\right\}\\
%     = \phi_{K}(s,a)^\tp (\hat{\Sigma}_t)^{-1} \\ & \qquad \qquad \left( \frac{1}{n}\sum_{i=1}^n \phi_{K}(S_{i,t}, A_{i,t}) g(S_{i,t}, A_{i,t})\right), %\\
%     %  (\tilde \Pi_t  g)(s,a) 
%     % = \phi_{K}(s,a)^\tp ({\Sigma}_t)^{-1} \left( \frac{1}{n}\sum_{i=1}^n \phi_K(S_{i,h}, A_{i,h}) f(S_{i,h}, A_{i,h})\right)
% \end{align*}
% for $g \in \calQ^{(t)}$, $t=1,\dots, T$. 
 Define $\mathcal{E}_tf = \EE f(S_t, A_t) = \int_{(s,a)} f(s,a) \dist_t^b(s,a) d(s,a)$, $\mathcal{E}^\pi_tf = \EE^\pi f(S_t, A_t) = \int_{(s,a)} f(s,a) \dist_t^\pi(s,a) d(s,a)$ for any function $f$.
One can verify that the value estimator \eqref{eqn:FQE-value} based FQE can be represented in the following form:
\begin{align}
	\label{eqn:MIS}
\hat{\nu}(\pi) 
= \sum_{t=1}^T \frac{1}{n} \sum_{i=1}^n \hat{w}_{i,t} R_{i,t} 
\end{align}
where
\begin{equation}
	\label{eqn:weights}
	\scalebox{0.9}{$
 \hat{w}_{i,t} = \mathcal{E}^\pi_1\left\{ \left( \prod_{t'=0}^{t-1} \hat{\calP}^\pi_{t'} \right) [\phi_K^\tp (\cdot)]^\tp  (\hat{\Sigma_t})^{-1} \right\} \phi_K(S_{i,t}, A_{i,t})$. }
\end{equation}
Note that \eqref{eqn:MIS} can be considered as a MIS estimator where $\hat{w}_{i,t}$ is used to estimate the ratio function $w^\pi_t(S_{i,t},A_{i,t})$  defined in \eqref{eqn:def_w}). As discussed in Section 3.3 in \citet{duan2020minimax}, under the tabular case, \eqref{eqn:MIS} matches the estimator proposed in \citet{yin2020asymptotically}.

%And we denote $\mathcal{D}_t$ as the collection of historical data up to time step $t$, i.e., $\mathcal{D}_t = \{S_1, A_1, R_1, \dots, S_{t-1}, A_{t-1}, R_{t-1}, S_t, A_t\}$. Write  $\langle \pi, Q \rangle  (\cdot)  = \sum_{a\in \calA}  \pi(a\mid \cdot) Q(\cdot, a)$.

\textbf{Additional notations.} We provide some additional notation that will be used later in the paper.  Denote by $\|\cdot \|_{\calL_2}$ and $\|\cdot\|_{\infty}$ the $\calL_2$-norm and the infinity norm respectively. More specifically,  $\|f\|_{\calL_2} = \sqrt{\EE f^2(S_t, A_t)}$ and $\|f\|_{\infty} = \sup_{(s,a) \in \calS \times \calA}|f(s,a)|$ for some function $f$ defined on $\calS \times \calA$.  %Denote by $\|\cdot\|_2$ the Euclidean norm of a vector. 
Define the projection $\Pi_t$  such that   $\Pi_t g (s,a)
=   \phi_{K}(s,a)^\tp ({\Sigma}_t)^{-1} \EE \left[ \phi_{K}(S_{t}, A_{t}) g(S_{t}, A_{t})\right] $  for $g \in \calQ^{(t)}$. We use the notation $\lesssim$ ($\gtrsim$) to denote less (greater) than up to an absolute constant.
%, which is irrelevant of any parameter.
We write $a \asymp b$ if $a \lesssim b$ and $b \gtrsim a$. 
Lastly, if a non-negative random variable $X$ satisfies $P(X\leq c\varrho(n,T) )\rightarrow 0$ as $c\rightarrow \infty$ for any $n,T$, we write $X = \bigO(\varrho(n,T))$ with high probability (W.H.P).  
% for a random variable $X$, if it satisfies $X  \leq C \varrho(n,T) $ with probability approaching $1$ \ray{as what changes?}, where $C$ is a constant  only potentially depends on the probability (i.e., independent of $n$ and $T$), we write $X = \bigO(\varrho(n,T))$ ``with high probability'' (W.H.P).
%\ray{do you have something like page 7 in \url{https://terrytao.files.wordpress.com/2011/02/matrix-book.pdf}?}
%\ray{do you mean: ``If a non-negative random variable $X$ satisfies $P(X\leq c\varrho(n,T) )\rightarrow 0$ as $c\rightarrow \infty$ for any $n,T$, we write $X = \bigO(\varrho(n,T))$ with high probability (W.H.P).''}
\section{Error Bounds}
\label{sec:upperbound}

% \ZL{we need a formal definition of Holder Class somewhere} \jiayi{Maybe in Appendix? Cool}

%\ZL{It is also better if we could group assumptions into several categories such as modeling assumption}

%\ZL{We need to explain why we do not focus on deriving the finite sample error bound but the convergence rate. CS people obsess with finite sample error bound and always criticize the convergence rate}
In this section, we analyze the finite-sample upper bound for $|\nu(\pi) - \hat{\nu}(\pi)|$.
%   Without loss of generality, we assume $\calA_1 = \cdots = \calA_T = \calA$, $\calS_1 = \cdots = \calS_T = \calS$, $K_1 = \cdots = K_T = K$, $\Psi^{(1)}_{K_1} = \cdots = \Psi^{(T)}_{K_T} = \Psi_K$  and $\phi^{(1)}_{K_1} = \cdots = \phi^{(T)}_{K_T} = \phi_K$.
% \jiayi{Define $\bigOp$ here? And also other notations. }
% \ray{we might simply make the assumption of the same state and action spaces in Section 2.1}
Note that the transitions are not required to be homogeneous, i.e., $\Pr_t$ can vary across $t$.
% Define $\mathcal{E}_tf = \EE f(S_t, A_t) = \int_{(s,a)} f(s,a) \dist_t^b(s,a) d(s,a)$, $\mathcal{E}^\pi_tf = \EE^\pi f(S_t, A_t) = \int_{(s,a)} f(s,a) \dist_t^\pi(s,a) d(s,a)$ for any function $f$.
% We decompose 
Also for the sake of clarity, we present key results in the main text and leave the most general error bounds in Section \ref{sec:thm_more} of the appendix.
%the error bounds under certain simplifying conditions.
%where the higher-order terms exhibit a polynomial order of $T$ instead of a possibly exponential order. 
%\ray{do we need to say this? it sounds alarming to readers that you are presenting the nicer order, especially the one that is not presented is exponential. I think we can just say that "For sake of clarity, we present the error bounds under certain simplifying conditions."} 
%Theoretical results under general conditions can be found .

In the following, we impose the following assumptions for $Q^\pi_t$, $t=1,\dots, T$ and the basis functions.

% \begin{assumption}
%     \label{ass:reward}
%     $\mathcal{S} \subset \mathbb{R}^d$ is compact.  %$|R_t|\leq 1$ for $t =1 ,\dots, T$. Reward $R_t$ is a function of $S_t, A_t$, i.e., $R_t = r_t(S_t, A_t)$ for some function $r_t$.
% \end{assumption} 

\begin{assumption}
    \label{ass:realizability}
   $r_t \in \mathcal{Q}^{(t)}$, for $t=1,\dots, T$. For every $q \in \mathcal{Q}^{(t+1)}$, we have $\mathcal{P}^\pi_t q \in \mathcal{Q}^{(t)}$.
%    \ZL{there seems some redundacy on the assumption}
%    \ray{$\mathcal{P}^{\pi}_t$ is not defined.}
\end{assumption}

\begin{assumption}
\label{ass:basis}
(i) $\mathcal{S} \subset \mathbb{R}^d$ is compact. The densities $\dist_t^b, t=1,\dots, T$ are uniformly bounded away from $0$ and $\infty$ on $\mathcal{S} \times \mathcal{A}$,
i.e., there exist constants $\underline{M}, \overline{M}>0$ (independent of $T$) such that $\underline{M}\le \dist_t^b(s,a)\le \overline{M}$ for all $s,a, t$.
%for every $t = 1,\dots, T$.
(ii)
% \ray{$\Sigma_t$ defined?}
The minimal eigenvalue of $\Sigma_{t,a}$ is uniformly lower bounded for all $t$ and $a$. In addition,
$\zeta_{K,t} := \sup_{s,a} \| \Sigma_t^{-1/2} \phi_K(s,a)\|_2 =\bigO(K)$ hold uniformly for all $t$.
%\jiayi{How about this: We assume that the minimal eigenvalue of $\Sigma_{t,a}$ is uniformly lower bounded for all $t$ and $a$. In addition, we assume that $\zeta_{K,t} = \bigO(K)$ hold uniformly for all $t$.}
%\ray{does this quantity depend on $T$? it is not clear from the domain of $t$ in the sup.}
%For  every $t=1,\dots, T, a\in \mathcal{A}$ and $K$, $\Sigma_{t,a}$ %\ray{defined?}
% is positive definite, i.e.,  $\lambda_{\min}(\Sigma_{t,a})>0$. \ray{just to confirm: you don't need \textit{uniform} boundedness from zero?}
%$\zeta_{K} = \bigO(\sqrt{K})$ \ray{does $\zeta_K$ depend on $T$?}.
(iii) The sieve basis (features) in $\Psi_K$ is H{\"o}lder continuous, i.e., there exist finite constants $\omega\ge 0$ and $\omega'>0$ (independent of $K$ and $T$) %\ray{(independent of $K$ and $T$)} 
such that the following inequality holds for every $t$, $a$ and $s_t, s_t'$.
\begin{equation}
    \left\| \Sigma_{t,a}^{-1/2} \left\{ \Psi_K(s_t)- \Psi_K(s_t') \right\}\right\|_2 \lesssim K^{\omega} \| s_t - s_t'\|_2^{\omega'}.
\end{equation}
% where $\|\cdot\|_2$ denotes the  norm of a vector.\ZL{move the definition of $\ell_2$ to notation part}
\end{assumption}
Assumption \ref{ass:realizability} states the realizability for function spaces $\calQ_t$, $t=1,\dots, T$ and the completeness for the Bellman operator, which are widely adopted in RL literature \citep[e.g.][]{ji2022sample, duan2020minimax, chen2019information}. \textit{When $\calQ_t$ is an infinite-dimensional space such as H{\"o}lder class, Assumption \ref{ass:realizability} is mild.} %In \jiayi{which?}, they show that when the reward and the transition density function follows \jiayi{fill in later}, the completeness automatically holds.   
Assumption \ref{ass:basis} states the conditions for the basis functions. Assumption \ref{ass:basis}(i) is standard. The requirement for bounded support can be relaxed to unbounded support with a slight modification, see, e.g., \citet{blundell2007semi, chen2018optimal, chen2012estimation}.  Assumption \ref{ass:basis}(ii) and (iii) are very mild and are satisfied by many commonly used sieve bases, such as splines and wavelets \citep{chen2018optimal} under Assumption \ref{ass:basis}(i).

% \jiayi{Introduce $\kappa$ here?}
Define 
\begin{align}
    \label{eqn:kappa}
    \kappa :=  \frac{1}{T}\sum_{t=1}^T \sup_{f \in \calQ^{(t)}} \frac{[\mathcal{E}^\pi_t f]^2}{\|f\|^2_{\calL_2}},
 \end{align}

% \ZL{Move notation} where $\|f\|_{\calL_2} = \sqrt{\EE f^2(S_t, A_t)}$ for some function $f$ defined on $\calS \times \calA$. 
The constant $\kappa$ quantifies the distribution shift between data induced by the behavior policy and the target policy. It shows up in the later error bounds in our analysis 
 %the behavior data distribution and the distribution under the target policy,
and also appears in many prior theoretical studies of OPE \citep[e.g.][]{duan2020minimax,ji2022sample}. 
Under Assumption \ref{ass:basis}(i), $\kappa$ is always upper bounded.

% \begin{assumption}
% \label{ass:holder}
 
% \end{assumption}

\subsection{Parametric Setting: A Preliminary Result}
\label{sec:para}

% \ZL{Did we define $\Pi_t$ before?} \jiayi{Yes, in the additional notations.}

In this section, we provide error bounds under the settings where 
% \jiayi{We already mention the linear mode of $\calQ^{(t)}$ here.}
%$\calQ^{(t)}$ is a linear space given by
$\calQ^{(t)} = \left\{ \phi_K(\cdot, \cdot)^\tp \beta: \beta \in \mathbb{R}^{K|\calA|} \right\}$
for  $t = 1,\dots, T$  and $K$ is a fixed constant.  In this case,  we have $Q^\pi_t = \Pi_t Q^\pi_t$ for $t = 1,\dots, T$. %, where $\Pi_t$ is defined by $\Pi_t g (s,a)
%  =   \phi_{K}(s,a)^\tp ({\Sigma}_t)^{-1} \EE \left[ \phi_{K}(S_{t}, A_{t}) g(S_{t}, A_{t})\right] $  for $g \in \calQ^{(t)}$. %\ray{did we define the projection $\Pi_t$?} 
which is the same as the in-homogeneous setting with linear function approximation discussed in \citet{duan2020minimax}.  

% \begin{thm}
%     \label{thm:para_bound}
%    Under Assumptions %\ref{ass: Markovian}-\ref{ass: DGP}, 
%    \ref{ass:realizability}-\ref{ass:basis}, %\ray{need to fix for the whole paper, as now the assumptions are also indexed by sections}, 
%     we have
%    \begin{multline}
%     |\hat{\nu}(\pi) - \nu(\pi)| 
%     = \bigO \left\{ \sqrt{\frac{ T^3 }{n} \kappa } + \right. \\
%      \left.  T \sum_{t=1}^T \left[ \left(  1 + \sqrt{\frac{{K \log n \log T} }{n}} \right)^t -1 \right]  \sqrt{\frac{{\log n \log T} }{n}}   \right\}, \text{ W.H.P.} \label{eqn: para_polynomial}
%    \end{multline}
% %    \ray{use the notation $\kappa$ to simplify this term}
% %    where $\|f\|_{\calL_2} = \sqrt{\EE f^2(S_t, A_t)}$ for some function $f$ defined on $\calS \times \calA$. 

%    If we further assume that 
%    \begin{align}
%     \label{eqn:para_T_condition}
%     T = \smallO([n/(\log n \log T)]^{1/2}),  
%    \end{align}
%    we have 
%    \begin{align}
%    |\hat{\nu}(\pi) - \nu(\pi)| 
%     = &\bigO \left( \sqrt{\frac{ T^3 }{n}\kappa}  
%  + T^{3} \frac{\log n \log T}{n} \right), \text{ W.H.P} \label{eqn: para_comp}.
%    \end{align}
% %    where \begin{align}
% %     \label{eqn:kappa}
% %     \kappa :=  \frac{1}{T}\sum_{t=1}^T \sup_{f \in \calQ^{(t)}} \frac{[\mathcal{E}^\pi_t f]^2}{\|f\|^2_{\calL_2}}.
% %  \end{align}
% \end{thm}
\begin{thm}
    \label{thm:para_bound}
   Under Assumptions %\ref{ass: Markovian}-\ref{ass: DGP}, 
   \ref{ass:realizability}-\ref{ass:basis},  %\ray{need to fix for the whole paper, as now the assumptions are also indexed by sections}, 
%     we have
%    \begin{multline}
%     |\hat{\nu}(\pi) - \nu(\pi)| 
%     = \bigO \left\{ \sqrt{\frac{ T^3 }{n} \kappa } + \right. \\
%      \left.  T \sum_{t=1}^T \left[ \left(  1 + \sqrt{\frac{{K \log n \log T} }{n}} \right)^t -1 \right]  \sqrt{\frac{{\log n \log T} }{n}}   \right\}, \text{ W.H.P.} \label{eqn: para_polynomial}
%    \end{multline}
%    \ray{use the notation $\kappa$ to simplify this term}
%    where $\|f\|_{\calL_2} = \sqrt{\EE f^2(S_t, A_t)}$ for some function $f$ defined on $\calS \times \calA$. 
   and further assume that  
   %$\calQ^{(t)} = \{\phi_K(\cdot)^\tp \beta : \beta \in \mathbb{R}^{K |\calA|}\}$ for $t=1,\dots, T$,
%    \begin{align}
%     \label{eqn:para_T_condition}
   \( T = \smallO([n/(\log n \log T)]^{1/2}) \),  
%    \end{align}
   we have W.H.P,
   \begin{equation}
    % \scalebox{0.5}{
   |\hat{\nu}(\pi) - \nu(\pi)| 
    = \bigO \left( \sqrt{\frac{ T^3\kappa }{n}}  
 + T^{3} \frac {\log n \log T}{n} \right).
    \label{eqn: para_comp}
   \end{equation}
%    where \begin{align}
%     \label{eqn:kappa}
%     \kappa :=  \frac{1}{T}\sum_{t=1}^T \sup_{f \in \calQ^{(t)}} \frac{[\mathcal{E}^\pi_t f]^2}{\|f\|^2_{\calL_2}}.
%  \end{align}
\end{thm}

% The term $\kappa$ \eqref{eqn:kappa} quantifies the distribution shift between data induced by the behavior policy and the target policy.
% %the behavior data distribution and the distribution under the target policy,
% This quantity also appears in many prior theoretical studies of OPE \citep[e.g.][]{duan2020minimax,ji2022sample}. 
% If $\kappa$ is bounded, then 
The first term in our bound \eqref{eqn: para_comp} exhibits an order of $T^{1.5}/\sqrt{n}$.  
If we focus on the convergent cases, the first term of \eqref{eqn: para_comp} dominates (up to logarithmic orders of $n$ and $T$), and hence is usually of interest. In \eqref{eqn: para_comp}, we refer to the first term as the first-order term and the second term as the higher-order term. In the later discussion, following this convention, we refer the term that has a slowest dependence on $n$ as the first-order term and the remaining terms as the higher-order terms.
%\ray{provide some explanation about higher-order term and why the first term is usually of focus; essentially, when we focus on \textit{convergent} settings, the first term dominates (up to logarithmic orders of $n$ and $T$) and so we call the second term the higher order term.}

%\ZL{we need to emphasize we are using the exactly same conditions as Duan to obtain a sharper first-order term.}

Extending the theoretical results from \citet{duan2020minimax} in the time-homogeneous setting  to the in-homogeneous one, 
%In Theorem 2 of \citet{duan2020minimax} \ray{best existing result?}, 
their bound will also comprise two major terms. The first-order term will have an order of $T^2/\sqrt{n}$. Compared with their bound, our first order term has an order of $T^{1.5}/\sqrt{n}$. We achieve a sharper horizon dependence by exploiting the fact that the variance of the first order term can be decomposed as a sum of $T$ individual expectations of the conditional variance. See Lemma \ref{lem:decompoeision} for more details. It is important to note that we use the exact same conditions as \citet{duan2020minimax} to achieve this rate of convergence. As for the higher-order term, our bound exhibits an order of $T^3/n$ and their bound is $T^{3.5}/n$. We have a stronger requirement for basis functions (Assumption \ref{ass:basis}(iii)) and derive the uniform convergence to achieve this. If we drop this assumption, by adopting their proof, we can show the same bound as theirs for the higher-order term. 

Next, Theorem \ref{thm:para_fast} shows that with an additional realizability assumption (Assumption \ref{ass:w}) on the probability ratio functions, the convergence rate of the error will depend \textit{linearly} with respect to the horizon $T$ in the first-order term.
This is a significant improvement in horizon dependence over the existing literature on the setting of using linear function approximation.
%\ray{if we have established the relationship between FQE and MIS in previous section. We can comment on that for tabular case (a special case of our parametric setting), \cite{yin2020asymptotically} shows the same first-order dependence for MIS estimator. Then we can comment that our results are directly comparable due to the relationship we show earlier in previous section. In tabular settings, we have well-specification of both Q and ratio (as the basis functions are all possible indicators). We should then emphasize that our is more general and not restricted to tabular settings.}
Our result also
aligns with the sharpest known dependence of the horizon in the first-order term under the tabular setting, as established in 
\citet{yin2020asymptotically} for their MIS estimator. 
% Due to the equivalence between FQE and MIS estimators (see the discussion in Section \ref{sec:FQE}). 
% As we have shown in Section \ref{sec:FQE} that the equivalence between FQE and MIS estimators, 
Note that the MIS estimator in the tabular setting can be considered as a special case of \eqref{eqn:MIS} by taking basis functions as indicator functions,
%and thus, it is a special case of our FQE estimators
due to the equivalence discussed in Section \ref{sec:FQE}.  In this case, it is also remarked that Assumption \ref{ass:w} holds automatically. Therefore, Theorem \ref{thm:para_fast} bridges the gap in the current literature
%and generalizes the result in \citet{yin2020asymptotically}
by showing linear horizon dependence for more general linear modeling (with potentially continuous state space).

\begin{assumption}
    \label{ass:w}
    %(a) $w_t^\pi$ is uniformly bounded away from $0$ and $\infty$ on $\mathcal{S} \times \mathcal{A}$  for every $t = 1,\dots, T$. \ZL{Does Assumption 3.2 already imply this?} \jiayi{We do not specify the upper boundedness of $\dist_t^\pi$?}
    % (b)
      $w_t^\pi \in \left\{ \phi_K(\cdot, \cdot)^\tp \beta: \beta \in \mathbb{R}^{K|\calA|} \right\}$, $t=1,\dots, T$.
\end{assumption}

% \jiayi{Assumption \ref{ass:w} is the realizability assumption for $w_t^\pi$. Under the tabular setting, Assumption \ref{ass:w} automatically holds.}

\begin{thm}
    \label{thm:para_fast}
       Under Assumptions %\ref{ass: Markovian}-\ref{ass: DGP}, 
       \ref{ass:realizability}-\ref{ass:basis}, and \ref{ass:w} with the condition that \( T = \smallO([n/(\log n \log T)]^{1/2}) \), %\eqref{eqn:para_T_condition}, 
       we have 
       W.H.P,
       %\begin{multline}
       % |\hat{\nu}(\pi) - \nu(\pi)| 
       % = \bigOp \left\{ T\sqrt{\frac{ 1}{n}\kappa }  \right.\\
       % \left. + T \sum_{t=1}^T \left[ \left(  1 + \sqrt{\frac{{K \log n \log T} }{n}} \right)^t -1 \right]  \sqrt{\frac{{\log n \log T} }{n}}   \right\},  \nonumber %\label{eqn: para_polynomial}
       %\end{multline}
       %If we further assume  \eqref{eqn:para_T_condition},\ZL{it's unusual to use equation in the later section} \jiayi{haha, wrong reference}
    %    \begin{align}
    %     % \label{eqn:para_T_condition}
    %     T = \smallO([n/(\log n \log T)]^{1/2}),   \nonumber
    %    \end{align}
       %we have 
       \begin{align}
        \scriptsize
       |\hat{\nu}(\pi) - \nu(\pi)| 
        = &\bigO \left( T\sqrt{\frac{ \kappa }{n}}  
     + T^{3} \frac{\log n \log T}{n} \right).  \label{eqn: para_comp2}
       \end{align}
\end{thm}

\subsection{Nonparametric Setting: Key Results}
\label{sec:non_para}
In this subsection, we generalize the parametric setting to the case where $Q$ functions are modeled nonparametrically. As opposed to the previous subsection, one fundamental difference is that linear function approximation with finite $K$ could incur non-negligible approximation errors. In the non-parametric setting, the number basis functions $K$ will grow with the sample size %and the horizon \ray{it usually decrease with the horizon, right? we can simply remove ``and the horizon''} 
so that the approximation error diminishes asymptotically. However, one needs to control the estimation error due to the increasing model complexity. In addition to the assumptions listed in Section \ref{sec:para}, we will need the following assumptions for our theoretical study.

% \jiayi{Remove this assumption since we assume Holder?}
\begin{assumption}
\label{ass:covering}
    % There exist constants $0 <\alpha < 2$ and $B>0$ (independent of $T$) such that $\calN(\mathcal{Q}^{(t)}(1), \|\cdot\|_{\infty}, \epsilon) \leq \exp(B \epsilon^{-\alpha})$ for $t=1,\dots, T$.
    % Here
    % %$\calN(\cdot, \cdot, \cdot)$ denotes the covering number
    % % and $\mathcal{Q}^{(t)}(1) = \{q \in \mathcal{Q}^{(t)}, \|q\|_{\infty} \leq 1\}$.
    % $\calN(\mathcal{Q}^{(t)}(1), \|\cdot\|_{\infty}: \epsilon)$ denotes the $\varepsilon$-covering number of $\mathcal{Q}^{(t)}(1) = \{q \in \mathcal{Q}^{(t)}, \|q\|_{\infty} \leq 1\}$ with respect to the norm $\|\cdot\|_{\infty}$. % where $\|f\|_{\infty} = \sup_{(s,a) \in \calS \times \calA}|f(s,a)|$ for $f \in \mathbb{R}^{\calS \times \calA}$.
    % % \ZL{move the definition of $\|\cdot\|_{\infty}$ to notation part}
For every $a \in \calA$ and $t=1,\dots, T$, $\{q(\cdot, a): q \in \calQ^{(t)}, \|q\|_{\infty}\leq 1\}$ is a subset of $\Lambda_{\infty}(p,L)$ with constants $p>d/2$ and $L>0$.
%\jiayi{Modify assumption}
\end{assumption}

\begin{assumption}
    \label{ass:bias}
    There exists a constant $\beta_Q>1/2$ (independent of $T$)
    such that $\sup_{q \in \mathcal{Q}^{(t)}(1)}\|q - \Pi_t q\|_{\infty} \lesssim K^{-\beta_Q}$ for $t=1,\dots, T$.
    \end{assumption}

%     Assumption \ref{ass:covering} 
% specifies the complexity of the hypothesis spaces. These covering number assumptions are satisfied by many commonly used functional classes. For example, if $ \{q(\cdot, a): q \in \calQ^{(t)}(1)\}$ is a subset of reproducing kernel Hilbert space (RKHS) or Sobolev space \citep{hearst1998support,geer2000empirical,gyorfi2006distribution} for every $a \in \calA$,   Assumption \ref{ass:covering}  holds.
% Take the Sobolev spaces as an example, \(\alpha = \alpha_t = d/(2q)\), where $d$ is the dimension of the input of functions and \(q\) is the number of continuous derivatives possessed by the functions in the corresponding space.  
Again, we emphasize that Assumption \ref{ass:covering} together with Assumption \ref{ass:realizability} are very mild as the H{\"o}lder space is very broad. 
{Indeed, we can show that Assumption \ref{ass:realizability} is not hard to satisfy under Assumption \ref{ass:covering}. By using similiar proof arguments of Lemma 1 in \citet{shi2022off}, one can show that if the transition kernel $p_t(s'|\cdot, a) \in \Lambda(p,L)$, $t=1,\dots, T$, for any $a \in \mathcal{A}$ and $s'\in \mathcal{S}$, then we have $\sup_{q \in  \mathcal{Q}^{(t+1)} : \|q\|_\infty\le 1 }||( \mathcal{P}^\pi_t) q||_{\Lambda(p)} \leq L$ for any policy $\pi$, where $\|\cdot\|_{\Lambda(p)}$ is the H{\"o}lder norm defined in the space $\Lambda(p,L)$ for $t=1,\dots, T$.  %For the Bellman operator $\mathcal{T}_t$, one can show that $\sup_{q \in B(1)}||(\Pi_{t=1}^T\mathcal{T}_t) q||_{\Lambda(p)} \leq LT$.
Therefore, the completeness assumption (Assumption \ref{ass:realizability}) is satisfied. }
Assumption \ref{ass:bias} specifies the uniform, projection error of $\Pi_t$ in sup-norm. %When $ \{q(\cdot, a): q \in \calQ^{(t)}(1)\}$ is a subset of H{\"o}lder space $C^q$ for every $a \in \calA$, 
which is satisfied when taking the basis functions such as B-spline basis or wavelet. 
%More specifically, let $B_{\infty}(q,L)$ denote a H{\"o}lder ball of smoothness $q>0$ and radius $L$ defined on a bounded open set $\Omega$ in $\mathbb{R}^d$:
% \begin{align*}
%  &  B_{\infty}(q,L) = \Biggl\{ f: \sum_{|\alpha|\leq q} \sup_{\Omega} |{\partial}^\alpha f | \leq L,\\
%  &   |\alpha| = \alpha_1 + \cdots + \alpha_d, \alpha_1,\dots, \alpha_d \text{ are nonnegative integers} \Biggr\}.
%   & {\partial}^\alpha f (x) = \frac{{\partial}^{|\alpha|} f }{\partial^{\alpha_1}\cdots \partial^{\alpha_d}} f(x).
% \end{align*} 
In these cases, one can take $\beta_Q = p/d$ so that Assumption \ref{ass:bias} holds.
% \jiayi{Roughly speaking, a larger $\beta_Q$ indicates smoother $Q$-functions.}
See Section 3.1 of \citet{chen2018optimal}. 
% \ZL{we need to talk about smoothness constant $\beta_Q$ to connect with later results} 
% \ray{all the above discussion of Sobolev, H{\"o}lder, splines and wavelet needs to restrict to part of the tensor product space, since $a$ is discrete.}\jiayi{Added.}

\subsubsection{A Slower Convergence Rate: No Realizability on Ratio Functions}

\label{sec:slow}

\begin{thm}
    \label{thm:slow}
   Under Assumptions %\ref{ass: Markovian}-\ref{ass: DGP}, 
   \ref{ass:realizability}-\ref{ass:basis}, \ref{ass:covering}-\ref{ass:bias}, if we further assume that $K = \smallO( \min\{ \sqrt{n/(\log n \log T)}, n/(T^2\log n \log T)\} )$, \(  T = \smallO(K^{\beta_Q})\),  then we have W.H.P, $|\hat{\nu}(\pi) - \nu(\pi)| 
    = $
%    \begin{multline}
%     |\hat{\nu}(\pi) - \nu(\pi)| 
%     = \bigOp \left\{ T^2 K^{-\beta_Q} + \sqrt{\frac{T^{3}}{n} \kappa}  \right. \\
%     \left.  + T \sum_{t=1}^T \left[ \left(  1 + K^{-\beta_Q} + \sqrt{\frac{{K \log n \log T} }{n}} \right)^t -1 \right] \right. \\
%     \left. \left( K^{-\beta_Q} + \sqrt{\frac{{K \log n \log T} }{n}} \right)   \right\}. \label{eqn: slow1}
%    \end{multline}
%    If we further assume that 
%    \begin{align}
%     \label{eqn:T_condition}
%     % T = \smallO(K^{\beta_Q} + [n/(\log n \log T)]^{1/2} K^{-1/2}), 
%     K = \smallO\{n/(T\log n \log T)\}, \  T = \smallO(K^{\beta_Q})
%    \end{align}
%    \ZL{it is a bit confusing when we have two conditions on $K$ and $T$}
%    we have 
   \begin{align}
   \label{eqn:slow}
   	% \scalebox{1}{$
   \bigO ( 
    T^2 K^{-\beta_Q} + \sqrt{\frac{T^{3}\kappa}{n}} +  \frac{T^{3}K \log n \log T}{n}). 
    % $} 
   \end{align}
%    where $\|f\|_{\calL_2} = \sqrt{\EE f^2(S_t, A_t)}$ for some function $f$ defined on $\calS \times \calA$. 
   %To ensure the existence of $K$, we require $$T \log T = \smallO\left(\left(\frac{n}{\log n}\right)^{\frac{\beta_Q}{1+2\beta_Q}}\right).$$
\end{thm}
% In Theorem \ref{thm:slow}, \eqref{eqn: slow1} provides our most general upper bound for convergence rate of $|\hat{\nu}(\pi) - \nu(\pi)|$. When $T$ is relatively small, the third component in \eqref{eqn: slow1} can be further bounded by a term that has a polynomial order dependence of $T$, as shown in \eqref{eqn:slow}. 
To ensure the existence of $K$ that satisfies the conditions listed in Theorem \ref{thm:slow}, % \eqref{eqn:T_condition}, 
we require 
\begin{align}
    \label{eqn:T_requirement}
    T \log T = \smallO\left(\left({n}/{\log n}\right)^{\frac{\beta_Q}{1+2\beta_Q}}\right).
\end{align}

The bound \eqref{eqn:slow} consists of three terms.  {The first term results from the bias of the approximation.  %When $\calQ^{(t)}$ belongs to the the H{\"o}lder space $C^q$, 
The larger $p$ in Assumption \ref{ass:covering} is (i.e., the smoother the $Q$-functions are), the smaller this bias term would be. }% the larger $\beta_Q$ is, and therefore the smaller this bias term would be.}  
 In the following corollaries, we discuss the dominant term under different scenarios. 
First, we consider the case where $T$ is bounded.

\begin{corollary}
    \label{cor:slow}
Suppose Assumptions %\ref{ass: Markovian}-\ref{ass: DGP},
 \ref{ass:realizability}-\ref{ass:basis}, \ref{ass:covering}-\ref{ass:bias} hold. We further assume that $T$ is bounded.
\begin{enumerate}
    \item[(i)] If $1/2 < \beta_Q \leq 1$, then by taking $K \asymp \sqrt{n}/(\log n \log T)$, we have 
    \begin{align}
        |\hat{\nu}(\pi) - \nu(\pi)| 
         = &\bigO \left( n^{-\beta_Q/2} \log n  \right), \text{ W.H.P}.
            \end{align}
    \item[(ii)] If $\beta_Q > 1$, then by taking $K \asymp (n/(\log n))^{1/(1 + \beta_Q)}$, we have
    \begin{align}
        |\hat{\nu}(\pi) - \nu(\pi)| 
        =  \bigO \left( n^{-1/2}  \right), \text{ W.H.P}.
       \end{align} 
    %    If we further assumes that $\beta_w = \beta_Q = \beta$, then the bound reduces to 
    %    \begin{align}
    %     |\hat{\nu}(\pi) - \nu(\pi)| 
    %     = \bigOp \left( \frac{T}{\sqrt{n}}  +  T^3  K^{-2\beta} + T^3 K^{-\beta} \sqrt{\frac{{K \log n} }{n}}    + T^{3} \frac{K \log n \log T}{n}  \right). 
    %    \end{align}
\end{enumerate}
\end{corollary}
As shown in case (ii) of Corollary \ref{cor:slow}, when $\beta_Q$ is large enough, i.e., $Q$ functions are smooth enough, we can achieve the optimal convergence rate $n^{-1/2}$. This addresses the fundamental question \quest{1}. When $1/2 < \beta_Q\leq 1$, by choosing $K$ appropriately,  our bound is faster than the optimal convergence rate $n^{-\beta_Q/(1 + 2\beta_Q)}$ for nonparametrically estimating the $Q$-functions \citep{chen2022well}. In other words, even without stricter smoothness assumption, our results indicate that the FQE estimator $\hat{\nu}(\pi)$ can still achieve a non-trivial fast rate, despite that $\hat{\nu}(\pi)$ is a simple plug-in estimator based on estimation of the nonparametric functions $\{Q_t^\pi\}$. %\ZL{we need to emphasize that this result is better than those in the table here; what is the insight? why we can achieve this?}
In comparison to the bounds derived in \citet{nguyen2021sample} and \citet{ji2022sample}, we have faster convergence rates with respect to $n$ in both cases of Corollary \ref{cor:slow}. Unlike their analysis, we decompose the error term differently (see the decomposition in Section \ref{sec:proof}), where we separate the first order term \eqref{eqn:def_E1} and the bias-induced term \eqref{eqn:def_E3}. By choosing $K$ appropriately, we can render the bias asymptotically ignorable to achieve optimal dependence on sample size $n$.  

Next, we discuss the scenario when the horizon $T$ can grow with $n$. For a neat presentation, we omit the log factors appearing in \eqref{eqn:korder_slow} and \eqref{eqn:slow2} in Corollary \ref{cor:slow2}.
\begin{corollary}
    \label{cor:slow2}
   Under Assumptions %\ref{ass: Markovian}, \ref{ass: reward}, \ref{ass: DGP}, 
   \ref{ass:realizability}, \ref{ass:basis}, \ref{ass:covering} and \ref{ass:bias}, if $T\log T = \smallO\left(\left({n}/{\log n}\right)^{{\beta}/{(1+2\beta)}}\right)$, then we can take 
\begin{align}
    % K \asymp \min \left\{ \left( \frac{n}{T \log n \log T} \right)^{\frac{1}{1+\beta_Q}}, \frac{\sqrt{n}}{\log n \log T}, \right.\\
    % \left. \frac{n}{T^2 \log n \log T} \right\}, 
    K \asymp \min \left\{ \left( {n}/{T } \right)^{\frac{1}{1+\beta_Q}}, {\sqrt{n}},  {n}/{T^2 } \right\},
    \label[type]{eqn:korder_slow} 
\end{align}
and we have W.H.P,  \quad  $ |\hat{\nu}(\pi) - \nu(\pi)| =$
   \begin{align}
   	\scalebox{0.94}{$
    \bigO \left( 
%    \left( \frac{n}{\log n \log T} \right)^{-\frac{\beta_Q}{1 + \beta_Q}} T^{\frac{2 + 3\beta_Q}{1+\beta_Q}}\right. \\
%    \left. + T^{2+ \beta_Q} \left( \frac{n}{\log n \log T} \right)^{-\beta_Q} + T^2 n^{\frac{-\beta_Q}{2}}(\log n \log T)^{\beta_Q} \right. \\
%    \left. + \sqrt{\frac{T^{3}}{n}\kappa} \right). \label{eqn:slow2}
%    |\hat{\nu}(\pi) - \nu(\pi)| 
%    = \bigOp \left( 
 {n}^{-\frac{\beta_Q}{1 + \beta_Q}} T^{\frac{2 + 3\beta_Q}{1+\beta_Q}}  + T^{2+ \beta_Q} {n} ^{-\beta_Q} + T^2 n^{\frac{-\beta_Q}{2}} + \sqrt{\frac{T^{3}\kappa}{n}} \right).$} \label{eqn:slow2}
%   \end{multline}
   \end{align}
%    where $\|f\|_{\calL_2} = \sqrt{\EE f^2(S_t, A_t)}$ for some function $f$ defined on $\calS \times \calA$. 
   %To ensure the existence of $K$, we require $$T \log T = \smallO\left(\left(\frac{n}{\log n}\right)^{\frac{\beta_Q}{1+2\beta_Q}}\right).$$
\end{corollary}

Which of the four terms in \eqref{eqn:slow2} dominates depends on the choice of 
%One of the four components in \eqref{eqn:slow2} will dominate given different choices of
$\beta_Q$, $n$ and $T$. When $\beta_Q > 1$, $\sqrt{T^3/n}$ is the first-order term and we obtain a polynomial dependence of order $1.5$ for $T$.  The remaining higher-order terms have a stronger dependence on $T$ but still converge faster due to a better dependence on $n$. When $\beta_Q\leq 1$, the slowest dependence concerning $n$ becomes $n^{-\beta_Q/2}$ and the corresponding term is quadratic in $T$ (up to logarithmic orders).
These results respond to \quest{2} and provide some key understanding of the horizon dependence.
In the next subsection, we show that FQE estimator $\hat{\nu}(\pi)$ can achieve better horizon dependence
with an additional realizability condition of the ratio function, adding another layer of understanding to \quest{2} by addressing \quest{3}.

\subsubsection{A Faster Convergence Rate: Realizability on Ratio Function}
\label{sec:fast}
%In this section, we discuss the scenario where we can obtain a faster rate than the one in Theorem \ref{thm:slow} for FQE when nonparametricly modeling the $Q$ functions. To achieve this, we need the following assumptions for the probability ratio function $w_t^\pi$, $t = 1,\dots, T$. 
In this section, we show a better convergence guarantee for FQE estimator by adopting the following realizability condition on the probability ratio function $w_t^\pi$, $t = 1,\dots, T$. 
\begin{assumption}
    \label{ass:basis_w}
%    (a) $w_t^\pi$ is uniformly bounded away from $0$ and $\infty$ on $\mathcal{S} \times \mathcal{A}$  for every $t = 1,\dots, T$. 
%    (b)  
There exists a constant $\beta_w>1/2$ such that $\sup_{t}\|w_t^\pi - \Pi_t w_t^\pi\|_{\infty} \lesssim K^{-\beta_w}$ for $t=1,\dots, T$.  
    \end{assumption}
 
% \begin{assumption}
%     \label{ass:bias_w}
   
%     \end{assumption}

% Assumption \ref{ass:basis_w}(a) is satisfied if $\dist^b_t$ is bounded away from 0 and $\dist^\pi_t$ is bounded above.  This overlap condition is commonly adopted in the RL literature \citep[e.g.][]{uehara2020minimax,yin2020asymptotically, wang2023projected}.  
Assumption \ref{ass:basis_w} imposes the smoothness conditions for the probability ratio functions and assumes that they can be approximated well by the basis functions that are used to model $Q$ functions,
as the number of basis functions $K$ increases. Similar to the discussion in Section \ref{sec:slow}, this assumption can be fulfilled when $w^\pi_t(\cdot, a)$ belongs to the  H{\"o}lder spaces for every $a \in \calA$ and the basis functions are taken as B-spline basis or wavelet. \textit{This implies that Assumption \ref{ass:basis_w} is also very mild.}

\begin{thm}
    \label{thm:fast}
    Under Assumptions %\ref{ass: Markovian}, \ref{ass: reward}, \ref{ass: DGP}, 
    \ref{ass:realizability}, \ref{ass:basis}, \ref{ass:covering}, \ref{ass:bias} and \ref{ass:basis_w},    If we further assume $K = \smallO( \min\{ \sqrt{n/(\log n \log T)}, n/(T^2\log n \log T)\} )$,  %$K = \smallO(\sqrt{n/(\log n \log T)})$,  \(K = \smallO\{n/(T^2\log n \log T)\},
     \(T = \smallO(K^{\beta_Q})\),  we have W.H.P,
%        \begin{align}
%    & |\hat{\nu}(\pi) - \nu(\pi)| 
%     = \bigOp \left\{ \frac{T}{\sqrt{n}} \right. \label{eqn:comp1} \\
%    & \left.  +  K^{-\beta_w}\left[ T^2 K^{-\beta_Q} + \sqrt{\frac{T^{3}}{n}}  \right] \right.  \label{eqn:comp2}\\
%    & \left.  + T \sum_{t=1}^T \left[ \left(  1 + K^{-\beta_Q} + \sqrt{\frac{{K \log n \log T} }{n}} \right)^t -1 \right]  \right.  \nonumber \\
%  &\qquad \qquad   \left. \left( K^{-\beta_Q} + \sqrt{\frac{{K \log n \log T} }{n}} \right)   \right\}. \label{eqn:comp3}
%    \end{align}
%    If we further assume \eqref{eqn:T_condition},
% %    \begin{align}
% %     % \label{eqn:T_condition}
% %     T = \smallO(K^{\beta_Q} + [n/(\log n \log T)]^{1/2} K^{-1/2}),  \nonumber
% %    \end{align}
%    we have 
        \begin{multline}
            |\hat{\nu}(\pi) - \nu(\pi)| 
            = \bigO \left( \frac{T}{\sqrt{n}}  +   T^2 K^{-\beta_Q-\beta_w} + T^3  K^{-2\beta_Q} + \right.\\
            \left.  T^3 K^{-\beta_Q} \sqrt{\frac{{K \log n \log T} }{n}}    + \frac{T^{3} K \log n \log T}{n}  \right).  \label{eqn:fast}
           \end{multline} 
\end{thm}

Similar to Theorem \ref{thm:slow}, %when $T$ is relatively slow (see \eqref{eqn:T_condition}), \eqref{eqn:comp3} can be upper bounded by a term in polynomial order of $T$. In this case, 
we require $T$ to satisfy \eqref{eqn:T_condition} in order to ensure the existence of $K$.   %For the following discussion, we focus on the case that $T$ fulfills \eqref{eqn:T_condition}.
The first term \eqref{eqn:fast} corresponds to the convergence error with respect to the first-order term E1 defined in \eqref{eqn:def_E1} % \ray{create equation number} 
in Section \ref{sec:proof} of the Appendix. The second term in \eqref{eqn:fast} is the error term that takes into account of the projection error of the probability ratio function $w^\pi_t$.
As illustrated in the bound \eqref{eqn:fast} in Theorem \ref{thm:fast}, even though we do not explicitly model the probability ratio function $w^\pi_t$, there exists some ``double-rate robustness'' property in FQE that helps to obtain a faster rate convergence as long as the probability ratio function can be approximated well by sieve bases that are used to model $Q$ functions. This responds to \quest{3}. Next, we take a deeper look for the bound \eqref{eqn:fast} under different scenarios. To simplify the discussion, we take $\beta_Q = \beta_w$, where we assume $w_t^\pi$ and $Q_t^\pi$ have the same degree of smoothness.

\begin{corollary}
    \label{cor:cases}
    Under Assumptions %\ref{ass: Markovian}-\ref{ass:basis}, 
    % \ZL{Assumptions number is not correct} \jiayi{Modified} 
    \ref{ass:realizability}, \ref{ass:basis}, \ref{ass:covering}, \ref{ass:bias} and \ref{ass:basis_w}, ,  and  further assume $\beta_Q = \beta_w = \beta > 1/2 $, $T\log T = \smallO\left(\left({n}/{\log n}\right)^{{\beta}/{(1+2\beta)}}\right)$, by taking the optimal order of $K$ such that 
    \begin{align}
        \label{eqn:korder}
        K \asymp \left\{ {n}/({\log n \log T}) \right\}^{\frac{1}{1+2\beta}},
    \end{align}
    we have   W.H.P,  \quad $ |\hat{\nu}(\pi) - \nu(\pi)|=$
    \begin{align}
        \label{eqn:cases}
        \begin{cases}
            \bigO\left( \frac{T}{\sqrt{n}} \right),  \mathrm{if }\  T = \smallO\left( n^{\frac{2\beta - 1}{4(1+ 2\beta)}}  (\log n)^{\frac{-\beta}{1+2\beta}}\right), \\ 
            \bigO \left( T^3 \left( \frac{n}{\log n \log T} \right)^{\frac{-2\beta}{1+2\beta}} \right),  \mathrm{otherwise}.
        \end{cases} 
    \end{align}
    % \ZL{it seems that we have consistency or not?}
    % \begin{enumerate}
    %     \item[(i)] when $T = \smallO\left( n^{\frac{2\beta - 1}{4(1+ 2\beta)}}  (\log n)^{\frac{-\beta}{1+2\beta}}\right) $, 
    %     \begin{align}
    %         |\hat{\nu}(\pi) - \nu(\pi)|  = \bigOp\left( \frac{T}{\sqrt{n}} \right)
    %     \end{align}
    %     \item[(ii)] when $T = \smallO\left( n^{\frac{2\beta - 1}{4(1+ 2\beta)}}  (\log n)^{\frac{-\beta}{1+2\beta}}\right) $, 
    %     \begin{align}
    %         |\hat{\nu}(\pi) - \nu(\pi)|  = \bigOp\left( \frac{T}{\sqrt{n}} \right)
    %     \end{align}
    % \end{enumerate}
\end{corollary}
As we can see, \eqref{eqn:korder} is equivalent to the optimal order for number of basis functions in common nonparametric regression, up to a logarithmic term \citep{chen2018optimal}.  If the number of horizon $T$ is bounded, % \ray{bounded}, 
 we can achieve the optimal convergence rate ($n^{-1/2}$) for $|\hat{\nu}(\pi) - \nu(\pi)|$ even though we estimate $Q$ functions nonparametricly. Compared to Corollary \ref{cor:slow}, we do not require $\beta_Q>1$ to achieve such optimal convergence rate, adding extra understanding to \quest{1}. In Corollary \ref{cor:slow}, the number of basis functions needs to be chosen of order larger than \eqref{eqn:korder} in order to remove the effect of approximation bias.
Next, we focus on the horizon dependence (\quest{2}). 
 When the horizon $T$ is allowed to grow with $n$, different regimes for upper bounds emerge depending on the relative order of $T$ to $n$.

In the scenario where $T$ grows relatively slowly compared to $n$ (case 1 in \eqref{eqn:cases}), the convergence exhibits a $n^{-1/2}$ dependence with respect to $n$, with a linear dependence on the horizon.  To the best of our knowledge, this convergence rate aligns with the best-known rate for FQE in tabular settings \citep{yin2020asymptotically} (necessarily parametric), despite our analysis is conducted under a much more challenging nonparametric setting.  Conversely, when $T$
grows faster,
%does not grow slowly enough compared to $n$,
the error bound exhibits a cubic horizon dependence. %with a polynomial order of $3$.
In this case, it exhibits a $n^{-2\beta/(1+2\beta)}$ dependence on $n$, which is better than $n^{-1/2}$ dependence that we found in the aforementioned case. However, it is important to note that
%, combining the roles of $T$ and $n$ in this scenario,
the overall rate is slower than the case of slowly growing $T$ due to the stronger dependence on $T$.
 In \citet{yin2020asymptotically}, their higher order term exhibits a $n^{-1}$ dependence with respect to $n$ and a $T^{3/2}$ dependence with respect to $T$. In comparison to their tabular setting, our model introduces additional bias terms due to the continuous setting and nonparametric modeling.  It remains unclear whether this order can be further improved to match theirs. We leave this as a subject for future research.
\section{Simulation Study}
\label{sec:simu}
We conduct a simulation study to illustrate the behavior of the error $|\hat{\nu}(\pi) - \nu(\pi)|$ with respect to $n$ and $T$.
The goal here is to provide empirical evidence of our theoretical results, and so we use a relatively simple simulation setup for the purpose of clear demonstration.
Specifically, the data generative model is given as follows.
The state variable is a one-dimensional continuous variable 
and the action is a binary variable, i.e., \(\calA_t =  \{0,1\}\) for all $t$.
The initial state follows the uniform distribution within $[-2,2]$.
The transition dynamics are given by \(S_{i,t+1} = (2A_{i,t}-1)f(S_{i,t})\),  where \(f(x)\) is a function 
constructed from cubic B-spline basis functions as depicted in Figure \ref{fig:transition} in Appendix. 
% \begin{figure}[t]
%     \centering
%     \includegraphics[width = 0.5\linewidth]{transition.pdf}
%     \caption{Illustration of the function $f$ in the transition mechanism.}
%     \label{fig:transition}
% \end{figure}
The behavior policy independently follows a Bernoulli distribution with mean \(1/2\). The immediate reward \(R_{i,t}\) is defined as \(R_{i,t} = 2S_{i,t+1}\). Though the data follows the homogeneous MDP, we treat it as an inhomogeneous setting and use FQE discussed in Section \ref{sec:FQE} to estimate the value for different target policies.
We evaluate the following two different target policies.
\vspace*{-2mm}
\begin{enumerate}[leftmargin=*]
	\item[(a)] \(\pi_t(a=1\mid s) = 0.5, \; \text{for} \;  s\in\mathcal{S}, t = 1,\dots, T\). 
	This policy is the same as the behavior one.
    %np.exp(spline(S[:,0]))/(1 + np.exp(spline(S[:,0])))
    \item[(b)] \(\pi_t(a=1\mid s) = \exp\{f(s)\}/\{1 + \exp(f(s))\}, \, \text{for} \;   s\in\mathcal{S}, t = 1,\dots, T.\)
	This policy is smooth with respect to the state variable.
	% \item[(c)]
	% \jiayi{Shall we add this optimal policy?} \ZL{At least no in the main text.}
    % \[
	% 	\pi_2(a=1\mid s) = 	
	% 	\begin{cases}
	% 		1 & \text{if}\  f(s) > 0\\
	% 		0 & \text{otherwise}
	% 	  \end{cases},
    %   \quad \text{for} \quad  s\in\mathcal{S}, t = 1,\dots, T.
	% \]
	% This policy is a discontinuous function with respect to the state variable. 
    \vspace*{-2mm}
\end{enumerate}
% We evaluate the following two different target policies.
For every target policy, we evaluate values with $n = 200, 400, \dots, 2000$, and $T =  20, 40, \dots,  200$. We use cubic B-spline to construct basis functions at every step $t$. The knots are placed at evenly distributed percentiles of samples.   The choice of $K$ is tricky to decide in practice. We consider two different approaches to specify $K$. First, following \eqref{eqn:korder} in  Corollary \ref{cor:cases}, we consider taking $\beta =2$, as the transition function has a continuous second derivative. And we try several different constants $c$ with $K = c n^{1/5}$. We found that $c=3$ yields an appropriate result and we fix $K = 3n^{1/5}$. For the second approach, we use leave-one-out cross-validation to decide $K$ at every step.  

Figure \ref{fig:policya} %and \ref{fig:policyb} %and \ref{fig:policyc} 
summarize the simulation results for target policies (a) over 200 simulation replicates when $K$ is selected by LOOCV.    As we can see,  we observe a roughly linear dependence between $|\hat{\nu}(\pi) - \nu(\pi)|$ and $T$, especially when $T$ is relatively small compared to $n$. This is reflected in \eqref{eqn:cases} in Corollary \ref{cor:cases}. The simulation results when $K$ is selected using the criterion $K = 3n^{1/5}$ and the results for policy (b) can be found in Section  \ref{sec:more_figures} in Appendix.  %The relationship between  between $|\hat{\nu}(\pi) - \nu(\pi)|$ and $n$ exhibits a roughly $1/\sqrt{n}$  dependence.

\begin{figure}[ht]
    \centering
    % \begin{subfigure}[b]{0.4\textwidth}
    %     \centering
    %     \includegraphics[width=\textwidth]{policya_K_n.pdf}
    %     \caption{$K = 3n^{1/5}$}
    %     % \label{fig:y equals x}
    % \end{subfigure}
    % \hfill
    % \begin{subfigure}[b]{0.4\textwidth}
    %     \centering
    %     \includegraphics[width=\textwidth]{policya_K_T.pdf}
    %     \caption{$K = 3n^{1/5}$}
    %     % \label{fig:three sin x}
    % \end{subfigure}
    % \hfill
    \begin{subfigure}[b]{0.43\textwidth}
        \centering
        \includegraphics[width=0.9\textwidth]{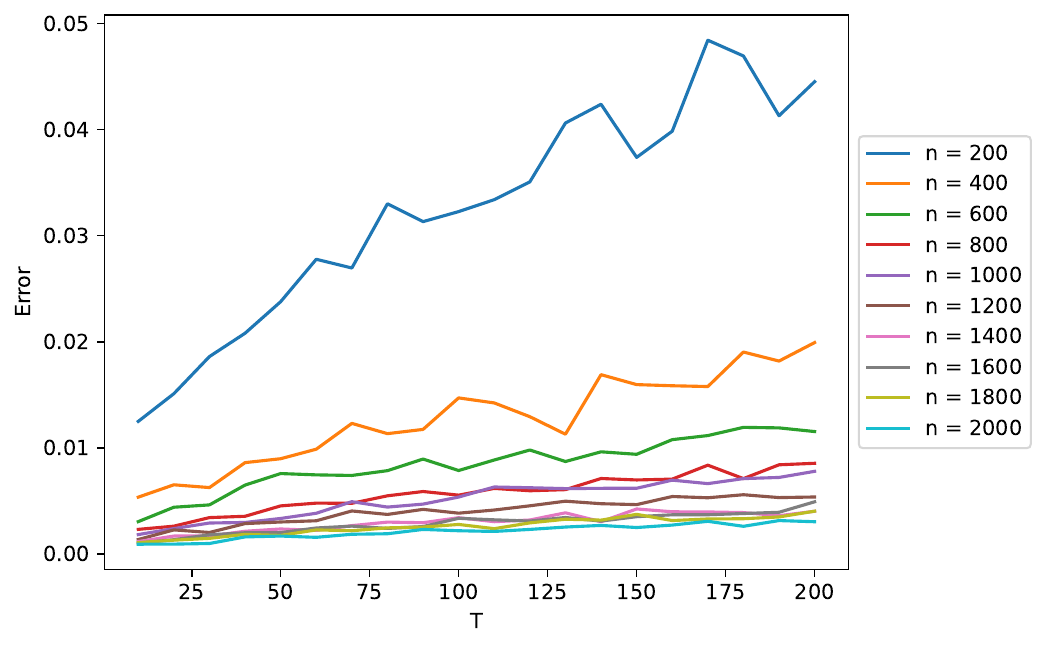}
        % \caption{$K$ is selected by LOOCV}
        % \label{fig:five over x}
    \end{subfigure}
    \hfill
    \begin{subfigure}[b]{0.43\textwidth}
        \centering
        \includegraphics[width=0.9\textwidth]{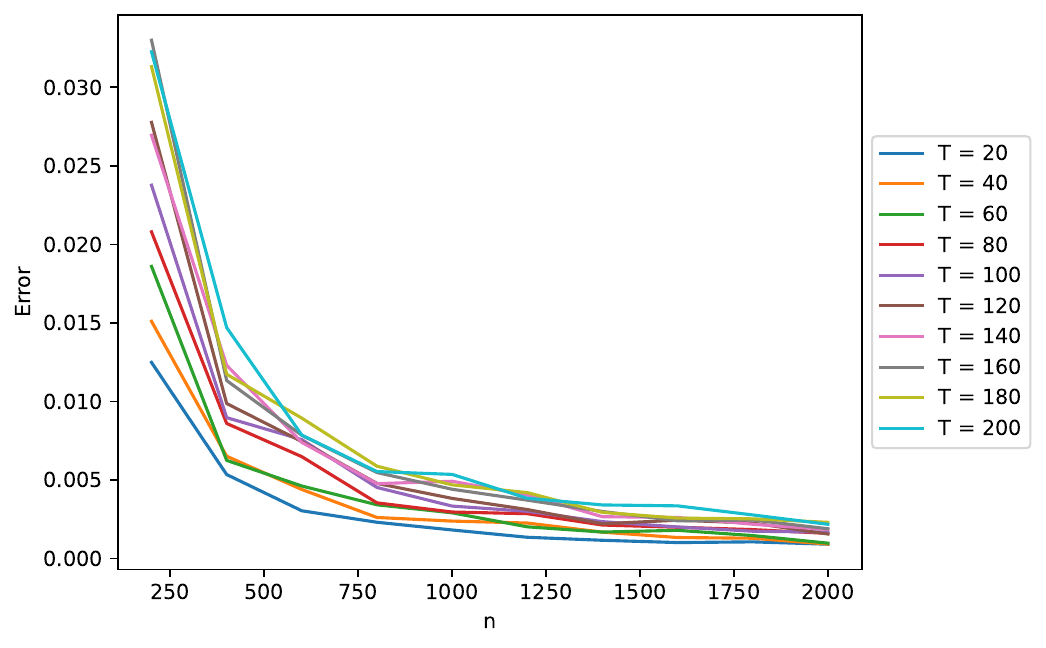}
        % \caption{$K$ is selected by LOOCV}
        % \label{fig:five over x}
    \end{subfigure}
       \caption{Simulation results for $ |\hat{\nu}(\pi) - \nu(\pi)| $ when the target policy $\pi$ is (a) and $K$ is selected by LOOCV. The upper plot demonstrates the change of error along with the change of $T$, different curves represent different $n$. The bottom plot demonstrates the change of error along with the change of $n$, different curves represent different $T$. }
       \label{fig:policya}
\end{figure}

\section*{Impact Statement}
This paper presents work whose goal is to advance the field of Machine Learning. There are many potential societal consequences of our work, none which we feel must be specifically highlighted here.

\bibliographystyle{icml2024}
\bibliography{refer.bib}

%%%%%%%%%%%%%%%%%%%%%%%%%%%%%%%%%%%%%%%%%%%%%%%%%%%%%%%%%%%%%%%%%%%%%%%%%%%%%%%
%%%%%%%%%%%%%%%%%%%%%%%%%%%%%%%%%%%%%%%%%%%%%%%%%%%%%%%%%%%%%%%%%%%%%%%%%%%%%%%
% APPENDIX
%%%%%%%%%%%%%%%%%%%%%%%%%%%%%%%%%%%%%%%%%%%%%%%%%%%%%%%%%%%%%%%%%%%%%%%%%%%%%%%
%%%%%%%%%%%%%%%%%%%%%%%%%%%%%%%%%%%%%%%%%%%%%%%%%%%%%%%%%%%%%%%%%%%%%%%%%%%%%%%
\newpage
\appendix
\onecolumn

\section{Additional Figures for the Simulation Results}
\label{sec:more_figures}

\begin{figure}[h]
    \centering
    \includegraphics[width = 0.5\linewidth]{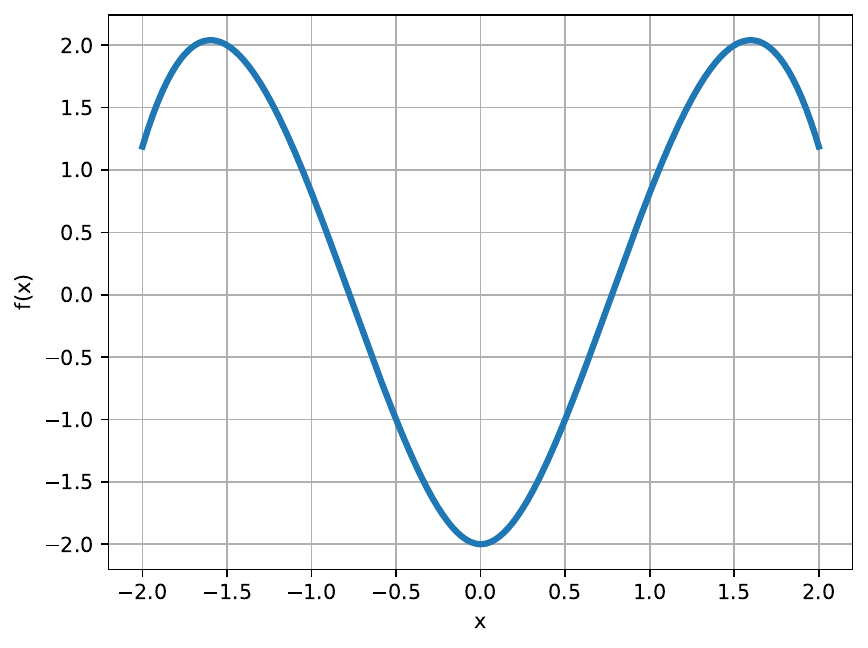}
    \caption{Illustration of the function $f$ in the transition mechanism.}
    \label{fig:transition}
\end{figure}

\begin{figure}[ht]
    \centering
    % \begin{subfigure}[b]{0.45\textwidth}
    %     \centering
    %     \includegraphics[width=\textwidth]{policyb_K_n.pdf}
    %     \caption{$K = 3n^{1/5}$}
    %     % \label{fig:y equals x}
    % \end{subfigure}
    % \hfill
    % \begin{subfigure}[b]{0.45\textwidth}
    %     \centering
    %     \includegraphics[width=\textwidth]{policyb_K_T.pdf}
    %     \caption{$K = 3n^{1/5}$}
    %     % \label{fig:three sin x}
    % \end{subfigure}
    % \hfill
    \begin{subfigure}[b]{0.43\textwidth}
        \centering
        \includegraphics[width=\textwidth]{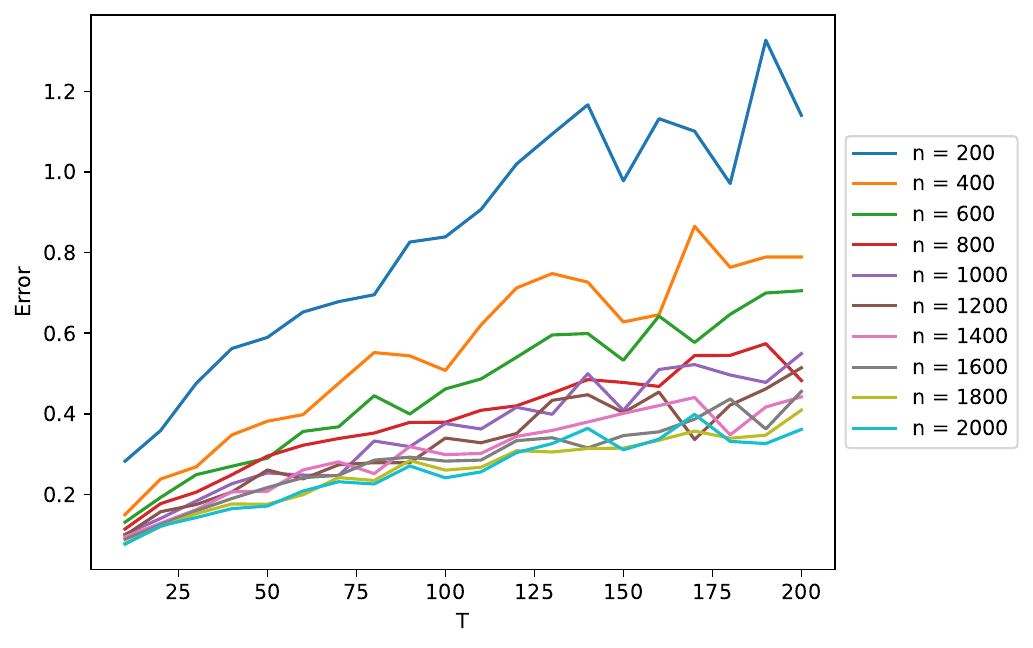}
        % \caption{$K$ is selected by LOOCV}
        % \label{fig:five over x}
    \end{subfigure}
    \hfill
    \begin{subfigure}[b]{0.43\textwidth}
        \centering
        \includegraphics[width=\textwidth]{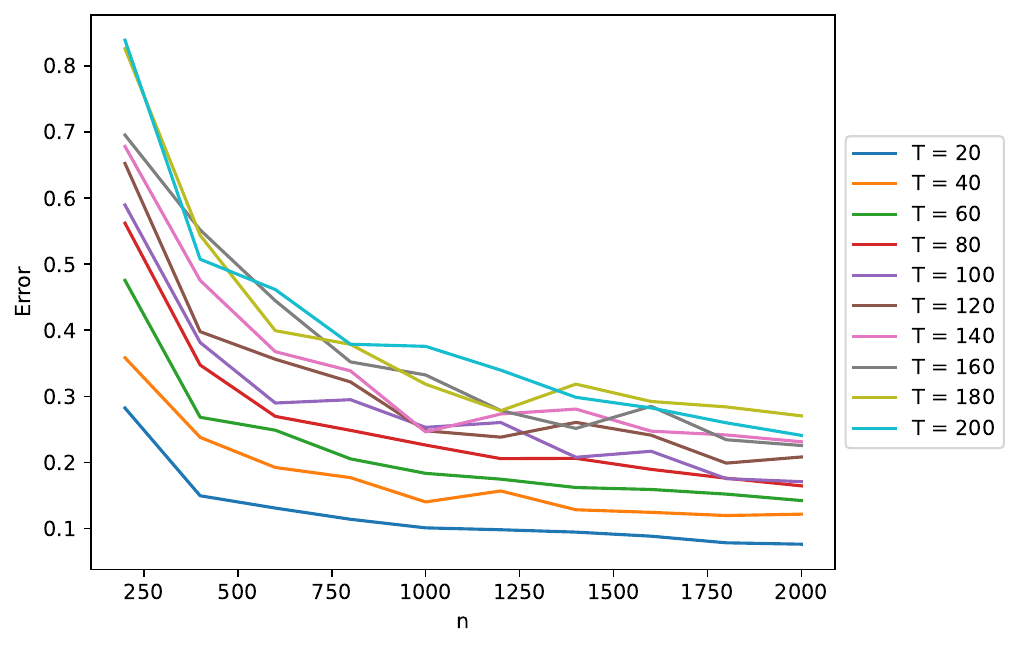}
        % \caption{$K$ is selected by LOOCV}
        % \label{fig:five over x}
    \end{subfigure}
       \caption{Simulation results for $ |\hat{\nu}(\pi) - \nu(\pi)| $ when the target policy $\pi$ is (b). See detailed description in Figure \ref{fig:policya}.}
       \label{fig:policyb}
\end{figure}

\begin{figure}[ht]
    \centering
    \begin{subfigure}[b]{0.4\textwidth}
        \centering
        \includegraphics[width=\textwidth]{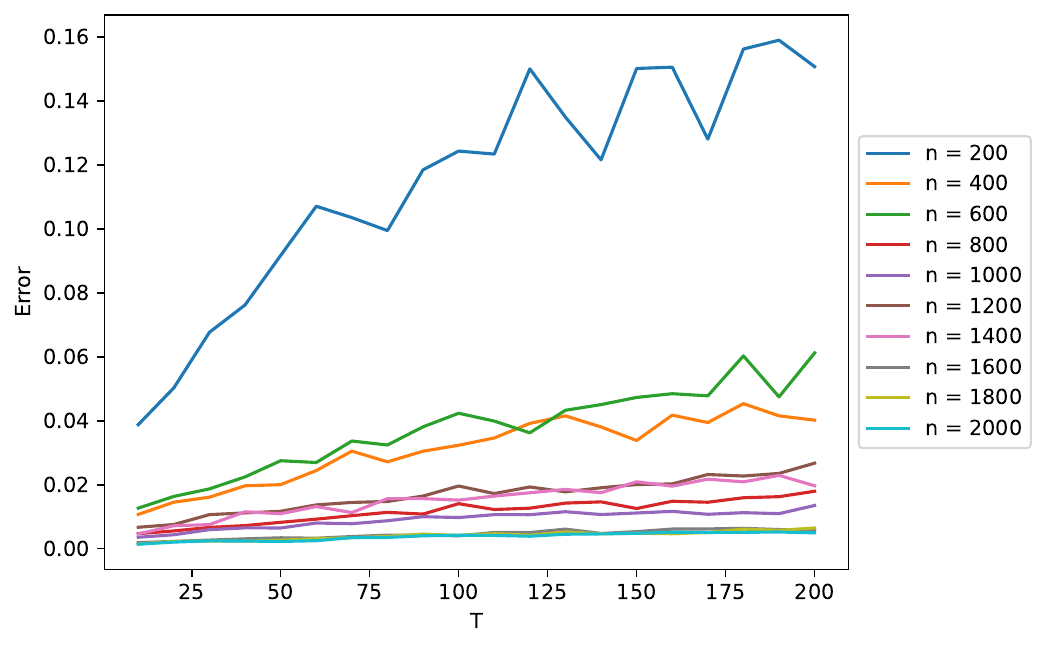}
        % \caption{$K = 3n^{1/5}$}
        % \label{fig:y equals x}
    \end{subfigure}
    \hfill
    \begin{subfigure}[b]{0.4\textwidth}
        \centering
        \includegraphics[width=\textwidth]{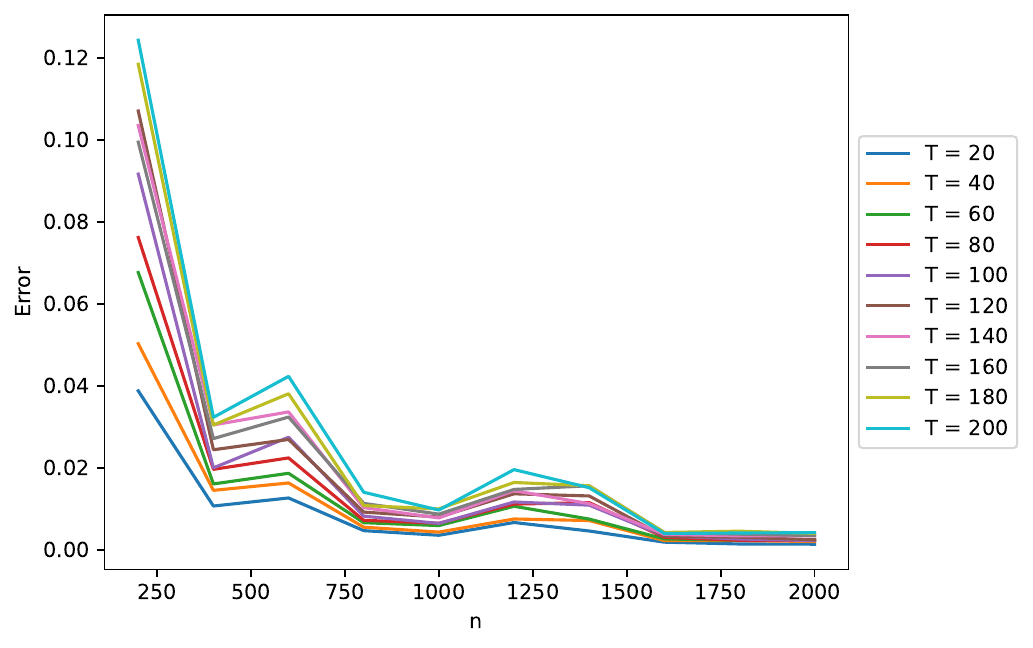}
        % \caption{$K = 3n^{1/5}$}
        % \label{fig:three sin x}
    \end{subfigure}
    % \hfill
    % \begin{subfigure}[b]{0.43\textwidth}
    %     \centering
    %     \includegraphics[width=\textwidth]{policya_CV_n.pdf}
    %     % \caption{$K$ is selected by LOOCV}
    %     % \label{fig:five over x}
    % \end{subfigure}
    % \hfill
    % \begin{subfigure}[b]{0.43\textwidth}
    %     \centering
    %     \includegraphics[width=\textwidth]{policya_CV_T.pdf}
    %     % \caption{$K$ is selected by LOOCV}
    %     % \label{fig:five over x}
    % \end{subfigure}
       \caption{Simulation results for $ |\hat{\nu}(\pi) - \nu(\pi)| $ when the target policy $\pi$ is (a) and $K$ is selected as $K = 3n^{1/5}$. The left plot demonstrate the change of error along with the change of $T$, different curves represent different $n$. The right plot demonstrates the change of error along with the change of $n$, different curves represent different $T$. }
       \label{fig:policya_more}
\end{figure}

\begin{figure}[h]
    \centering
    \begin{subfigure}[b]{0.45\textwidth}
        \centering
        \includegraphics[width=\textwidth]{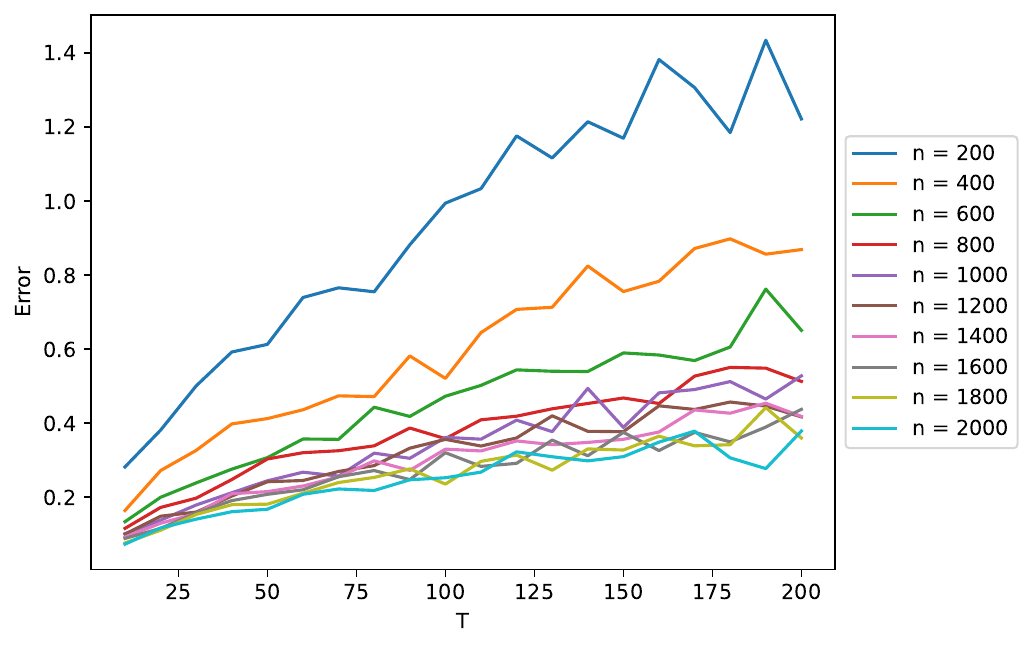}
        % \caption{$K = 3n^{1/5}$}
        % \label{fig:y equals x}
    \end{subfigure}
    \hfill
    \begin{subfigure}[b]{0.45\textwidth}
        \centering
        \includegraphics[width=\textwidth]{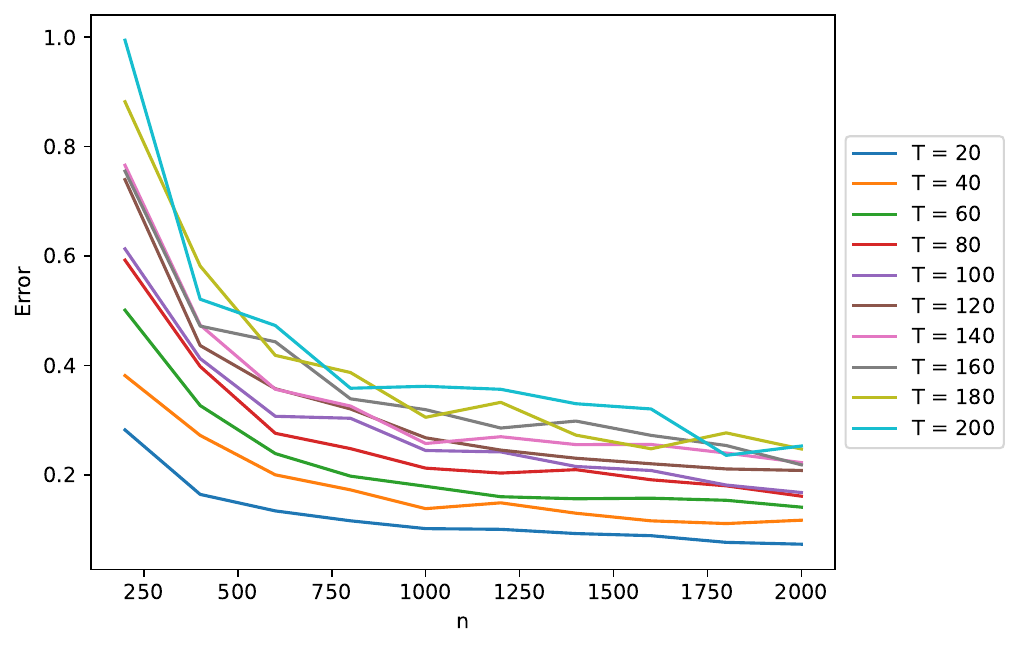}
        % \caption{$K = 3n^{1/5}$}
        % \label{fig:three sin x}
    \end{subfigure}
    % \hfill
    % \begin{subfigure}[b]{0.43\textwidth}
    %     \centering
    %     \includegraphics[width=\textwidth]{policyb_CV_n.pdf}
    %     % \caption{$K$ is selected by LOOCV}
    %     % \label{fig:five over x}
    % \end{subfigure}
    % \hfill
    % \begin{subfigure}[b]{0.43\textwidth}
    %     \centering
    %     \includegraphics[width=\textwidth]{policyb_CV_T.pdf}
    %     % \caption{$K$ is selected by LOOCV}
    %     % \label{fig:five over x}
    % \end{subfigure}
       \caption{Simulation results for $ |\hat{\nu}(\pi) - \nu(\pi)| $ when the target policy $\pi$ is (b). See detailed description in Figure \ref{fig:policya_more}.}
       \label{fig:policyb_more}
\end{figure}

In addition to two policies evaluated in Section \ref{sec:simu}, we also evaluate the following policy 
\begin{itemize}
    \item[(c)]
	% \jiayi{Shall we add this optimal policy?} \ZL{At least no in the main text.}
    \[
		\pi_2(a=1\mid s) = 	
		\begin{cases}
			1 & \text{if}\  f(s) > 0\\
			0 & \text{otherwise}
		  \end{cases},
      \quad \text{for} \quad  s\in\mathcal{S}, t = 1,\dots, T.
	\]
	This policy is a discontinuous function with respect to the state variable. 
\end{itemize}

\begin{figure}[h]
    \centering
    % \begin{subfigure}[b]{0.45\textwidth}
    %     \centering
    %     \includegraphics[width=\textwidth]{policyc_K_n.pdf}
    %     \caption{$K = 3n^{1/5}$}
    %     % \label{fig:y equals x}
    % \end{subfigure}
    % \hfill
    % \begin{subfigure}[b]{0.45\textwidth}
    %     \centering
    %     \includegraphics[width=\textwidth]{policyc_K_T.pdf}
    %     \caption{$K = 3n^{1/5}$}
    %     % \label{fig:three sin x}
    % \end{subfigure}
    % \hfill
    \begin{subfigure}[b]{0.43\textwidth}
        \centering
        \includegraphics[width=\textwidth]{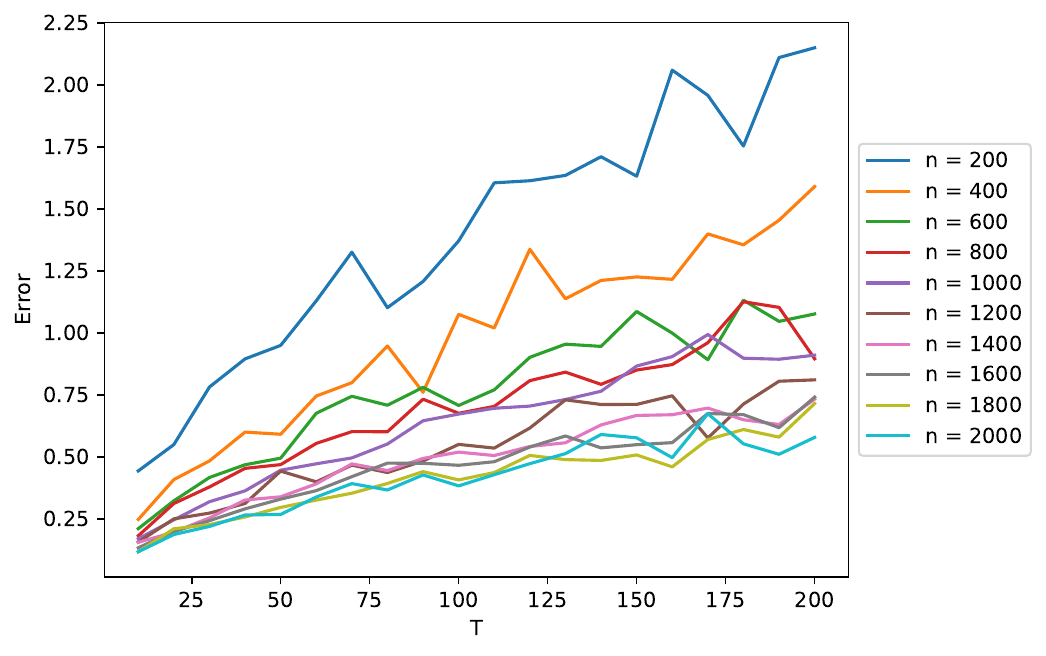}
        % \caption{$K$ is selected by LOOCV}
        % \label{fig:five over x}
    \end{subfigure}
    \hfill
    \begin{subfigure}[b]{0.43\textwidth}
        \centering
        \includegraphics[width=\textwidth]{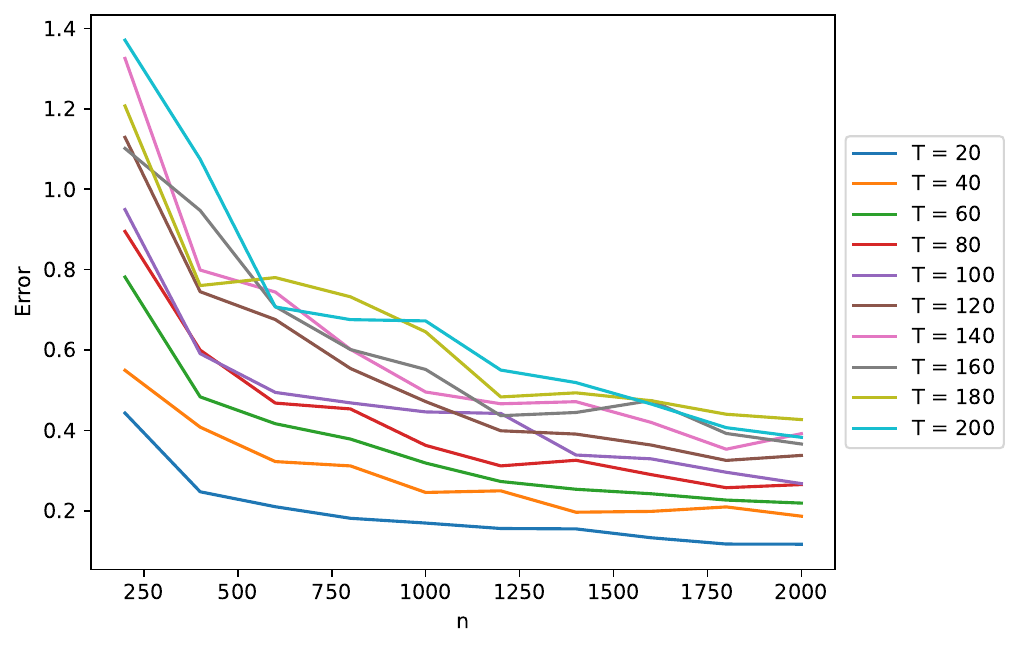}
        % \caption{$K$ is selected by LOOCV}
        % \label{fig:five over x}
    \end{subfigure}
       \caption{Simulation results for $ |\hat{\nu}(\pi) - \nu(\pi)| $ when the target policy $\pi$ is (c). See detailed description in Figure \ref{fig:policya}.}
       \label{fig:policyc}
\end{figure}

\begin{figure}[h]
    \centering
    \begin{subfigure}[b]{0.45\textwidth}
        \centering
        \includegraphics[width=\textwidth]{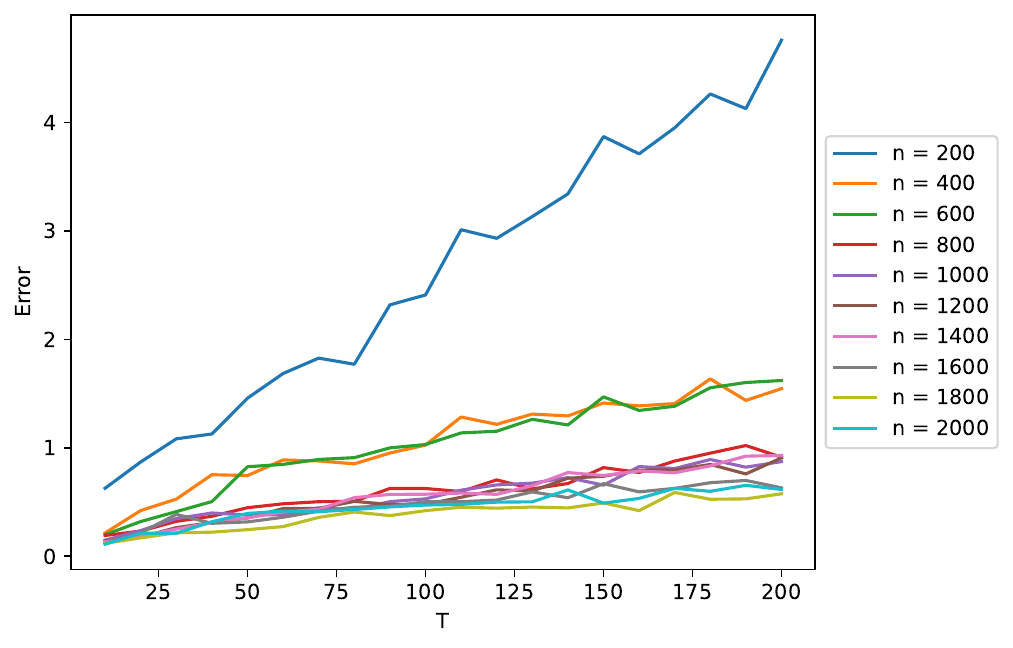}
        % \caption{$K = 3n^{1/5}$}
        % \label{fig:y equals x}
    \end{subfigure}
    \hfill
    \begin{subfigure}[b]{0.45\textwidth}
        \centering
        \includegraphics[width=\textwidth]{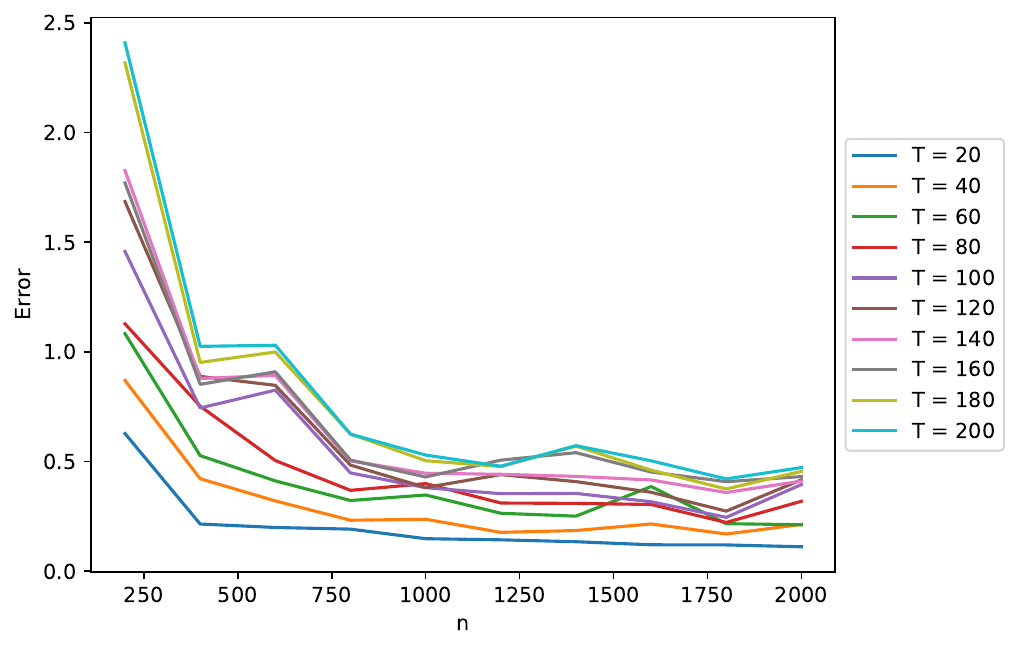}
        % \caption{$K = 3n^{1/5}$}
        % \label{fig:three sin x}
    \end{subfigure}
    % \hfill
    % \begin{subfigure}[b]{0.43\textwidth}
    %     \centering
    %     \includegraphics[width=\textwidth]{policyc_CV_n.pdf}
    %     % \caption{$K$ is selected by LOOCV}
    %     % \label{fig:five over x}
    % \end{subfigure}
    % \hfill
    % \begin{subfigure}[b]{0.43\textwidth}
    %     \centering
    %     \includegraphics[width=\textwidth]{policyc_CV_T.pdf}
    %     % \caption{$K$ is selected by LOOCV}
    %     % \label{fig:five over x}
    % \end{subfigure}
       \caption{Simulation results for $ |\hat{\nu}(\pi) - \nu(\pi)| $ when the target policy $\pi$ is (c). See detailed description in Figure \ref{fig:policya_more}.}
       \label{fig:policyc_more}
\end{figure}

\section{Detailed Theorem Statements}
\label{sec:thm_more}
\subsection{Parametric Setting}
\begin{thm}
    \label{thm:para_bound_more}
   Under Assumptions %\ref{ass: Markovian}-\ref{ass: DGP}, 
   \ref{ass:realizability}-\ref{ass:basis}, %\ray{need to fix for the whole paper, as now the assumptions are also indexed by sections}, 
    we have
   \begin{align}
    |\hat{\nu}(\pi) - \nu(\pi)| 
    = \bigO \left\{ \sqrt{\frac{ T^3 }{n} \kappa } + T \sum_{t=1}^T \left[ \left(  1 + \sqrt{\frac{{K \log n \log T} }{n}} \right)^t -1 \right]  \sqrt{\frac{{\log n \log T} }{n}}   \right\}, \text{ W.H.P.} \label{eqn: para_polynomial_more}
   \end{align}
%    \ray{use the notation $\kappa$ to simplify this term}
%    where $\|f\|_{\calL_2} = \sqrt{\EE f^2(S_t, A_t)}$ for some function $f$ defined on $\calS \times \calA$. 

   If we further assume that 
   \begin{align}
    % \label{eqn:para_T_condition}
    T = \smallO([n/(\log n \log T)]^{1/2}),   \nonumber
   \end{align}
   we have 
   \begin{align}
   |\hat{\nu}(\pi) - \nu(\pi)| 
    = &\bigO \left( \sqrt{\frac{ T^3 }{n}\kappa}  
 + T^{3} \frac{\log n \log T}{n} \right), \text{ W.H.P}.  \nonumber % \label{eqn: para_comp}.
   \end{align}
%    where \begin{align}
%     \label{eqn:kappa}
%     \kappa :=  \frac{1}{T}\sum_{t=1}^T \sup_{f \in \calQ^{(t)}} \frac{[\mathcal{E}^\pi_t f]^2}{\|f\|^2_{\calL_2}}.
%  \end{align}
\end{thm}

\subsection{Nonparametric Setting}

\begin{thm}
    \label{thm:slow_more}
   Under Assumptions %\ref{ass: Markovian}-\ref{ass: DGP}, 
   \ref{ass:realizability}-\ref{ass:basis}, \ref{ass:covering}-\ref{ass:bias}, if $K = \smallO(\sqrt{n/(\log n \log T)})$, then we have
   \begin{multline}
    |\hat{\nu}(\pi) - \nu(\pi)| 
    = \bigOp \left\{ T^2 K^{-\beta_Q} + \sqrt{\frac{T^{3}}{n} \kappa}  \right. \\
    \left.  + T \sum_{t=1}^T \left[ \left(  1 + K^{-\beta_Q} + \sqrt{\frac{{K \log n \log T} }{n}} \right)^t -1 \right] \left( K^{-\beta_Q} + \sqrt{\frac{{K \log n \log T} }{n}} \right)   \right\}. \nonumber % \label{eqn: slow1_more}
   \end{multline}
   If we further assume that 
   \begin{align}
    \label{eqn:T_condition}
    % T = \smallO(K^{\beta_Q} + [n/(\log n \log T)]^{1/2} K^{-1/2}), 
    K = \smallO\{n/(T^2\log n \log T)\}, \  T = \smallO(K^{\beta_Q}) 
   \end{align}
%    \ZL{it is a bit confusing when we have two conditions on $K$ and $T$}
   we have 
   \begin{align}
   |\hat{\nu}(\pi) - \nu(\pi)| 
    = &\bigOp \left( 
    T^2 K^{-\beta_Q} + \sqrt{\frac{T^{3}}{n}\kappa}  +  \frac{T^{3}K \log n \log T}{n} \right). \nonumber %\label{eqn:slow_more}
   \end{align}
%    where $\|f\|_{\calL_2} = \sqrt{\EE f^2(S_t, A_t)}$ for some function $f$ defined on $\calS \times \calA$. 
   %To ensure the existence of $K$, we require $$T \log T = \smallO\left(\left(\frac{n}{\log n}\right)^{\frac{\beta_Q}{1+2\beta_Q}}\right).$$
\end{thm}

\begin{thm}
    \label{thm:fast_more}
    Under Assumptions %\ref{ass: Markovian}, \ref{ass: reward}, \ref{ass: DGP}, 
    \ref{ass:realizability}, \ref{ass:basis}, \ref{ass:covering}, \ref{ass:bias} and \ref{ass:basis_w}, $K = \smallO(\sqrt{n/(\log n \log T)})$,  we have 
       \begin{align*}
   & |\hat{\nu}(\pi) - \nu(\pi)| 
    = \bigO \left\{ \frac{T}{\sqrt{n}}   +  K^{-\beta_w}\left[ T^2 K^{-\beta_Q} + \sqrt{\frac{T^{3}}{n}}  \right] \right.  %\label{eqn:comp2}
   \\
   & \left.  + T \sum_{t=1}^T \left[ \left(  1 + K^{-\beta_Q} + \sqrt{\frac{{K \log n \log T} }{n}} \right)^t -1 \right]   \left( K^{-\beta_Q} + \sqrt{\frac{{K \log n \log T} }{n}} \right)   \right\}, \text{ W.H.P}. %\label{eqn:comp3}
   \end{align*}
   If we further assume \eqref{eqn:T_condition},
%    \begin{align}
%     % \label{eqn:T_condition}
%     T = \smallO(K^{\beta_Q} + [n/(\log n \log T)]^{1/2} K^{-1/2}),  \nonumber
%    \end{align}
   we have 
        \begin{multline}
            |\hat{\nu}(\pi) - \nu(\pi)| 
            = \bigO \left( \frac{T}{\sqrt{n}}  +   T^2 K^{-\beta_Q-\beta_w} + T^3  K^{-2\beta_Q}  + T^3 K^{-\beta_Q} \sqrt{\frac{{K \log n \log T} }{n}}    + T^{3} \frac{K \log n \log T}{n}  \right), \text{ W.H.P}.  \nonumber %\label{eqn:fast} 
           \end{multline} 
\end{thm}

\section{Proof of Main Theorems}
\label{sec:proof}
In this section, we provide the proof for Theorem \ref{thm:slow} and \ref{thm:fast}. The proof for theoretical results in the parametric setting can be derived as a special case by taking $\beta_Q$ and $\beta_w$ to be infinity. 

First of all, we recall and introduce some notations. 
Let  
$\Sigma_t = \EE [\phi_{K}(S_t, A_t) \phi_{K}(S_t, A_t)^\tp]\in \mathbb{R}^{K}$, 
$\hat{\Sigma}_t = \frac{1}{n}\sum_{i=1}^n [\phi_{K}(S_{i,t}, A_{i,t}) \phi_{K}(S_{i,t}, A_{i,t})^\tp]\in \mathbb{R}^{K}$ and $\Sigma_{t,a} = \EE [\psi_{K}(S_t) \psi_{K}(S_t)^\tp\mid A_t = a]$. And we denote $\mathcal{D}_t$ as the collection of historical data up to time step $t$, i.e., $\mathcal{D}_t = \{S_1, A_1, R_1, \dots, S_{t-1}, A_{t-1}, R_{t-1}, S_t, A_t\}$. Write  $\langle \pi, Q \rangle  (\cdot)  = \sum_{a\in \calA}  \pi(a\mid \cdot) Q(\cdot, a)$.

Define $\mathcal{P}_t$ and  $\hat {\mathcal{P}}_t^\pi $ as the population and estimated conditional expectation operators respectively, such that 
% $$\hat{\Sigma}_{t,a} = \frac{1}{\sum_{i=1}^n \mathbbm{1}(A_{i,t}=a)}\sum_{i=1}^n [\psi_K(S_{i,t}) \psi^\tp_K(S_{i,t})\mathbbm{1}(A_{i,t}=a) ]$$
% And we take
\begin{align*}
    (\mathcal{P}_t^\pi  f)(s,a) = \EE\left\{ \sum_{a'} \pi_t(a'\mid S_{t+1}) f(S_{t+1}, a') \mid S_t = s, A_t = a\right\}, \\
    (\hat {\mathcal{P}}_t^\pi  f)(s,a) %= \hat{\EE}\left\{ \sum_{a'} \pi(a'\mid S_{h+1}) f(S_{h+1}, a') \mid S_h = s, A_h = a\right\}\\
    = \phi_{K}(s,a)^\tp (\hat{\Sigma}_t)^{-1} \left( \frac{1}{n}\sum_{i=1}^n \phi_{K}(S_{i,t}, A_{i,t}) \left[ \sum_{a'} \pi_t(a'\mid S_{i,h+1}) f(S_{i,t+1}, a') \right]\right),
\end{align*}
for $f \in \calQ^{(t+1)}$, $t=1,\dots, T$.  In addition, we define ${\Pi}_t$ and  $\hat {{\Pi}}_t^\pi $ as the population and estimated projection operators respectively, such that 
\begin{align*}
    \Pi_t g (s,a) =   \phi_{K}(s,a)^\tp ({\Sigma}_t)^{-1} \EE \left[ \phi_{K}(S_{i,t}, A_{i,t}) g(S_{i,t}, A_{i,t})\right], \\
    (\hat \Pi_t  g)(s,a) %= \hat{\EE}\left\{  f(S_{h}, A_h) \mid S_h = s, A_h = a\right\}\\
    = \phi_{K}(s,a)^\tp (\hat{\Sigma}_t)^{-1} \left( \frac{1}{n}\sum_{i=1}^n \phi_{K}(S_{i,t}, A_{i,t}) g(S_{i,t}, A_{i,t})\right), \\
     (\tilde \Pi_t  g)(s,a) 
    = \phi_{K}(s,a)^\tp ({\Sigma}_t)^{-1} \left( \frac{1}{n}\sum_{i=1}^n \phi_K(S_{i,h}, A_{i,h}) g(S_{i,h}, A_{i,h})\right)
\end{align*}
for $g \in \calQ^{(t)}$, $t=1,\dots, T$. 

Then we have the following decomposition
\begin{align*}
    \nu(\pi) - \hat{\nu}(\pi) =  & \mathcal{E}_1\{ Q_1^\pi - \hat{Q}_1^\pi\} \\
     = & \mathcal{E}_1\{  Q_1^\pi - (\hat{\Pi}_1 R_1 + \hat{\calP_1^\pi}Q_2^\pi) + \hat{\calP_1^\pi}(Q_2^\pi - \hat Q_2^\pi) \} \\
    = &  \mathcal{E}_1\{ [ Q_1^\pi - (\hat{\Pi}_1 R_1 + \hat{\calP_1^\pi}Q_2^\pi) ] + \hat{\calP_1^\pi}[Q_2^\pi - (\hat \Pi_tR_2 + \hat{\calP_2^\pi} Q_3^\pi)] + \hat{\calP_1^\pi}\hat{\calP_2^\pi}(Q_3^\pi - \hat Q_3^\pi)\}\\
    = & \cdots \\
    = &  \mathcal{E}_1\{ [ Q_1^\pi - (\hat{\Pi}_1 R_1 + \hat{\calP_1^\pi}Q_2^\pi) ] + \hat{\calP_1^\pi}[Q_2^\pi - (\hat \Pi_tR_2 + \hat{\calP_2^\pi} Q_3^\pi)]  + \cdots + \hat{\calP_1^\pi} \cdots \hat{\calP_{T-1}^\pi}[Q_T^\pi - \hat{\calP_T^\pi}R_T]\}\\
    % = & [ \hat{\Pi}_1Q_1  - \hat{\calP_1^\pi}(R_1 + Q_2^\pi) + ( Q_1 - \hat{\Pi}_1Q_1) ] \\
    % & + \hat{\calP_1^\pi}[\hat{\Pi}_2Q_2 - \hat{\calP_2^\pi}(R_2 + Q_3^\pi)+ ( Q_2 - \hat{\Pi}_2Q_2)]  + \cdots \\
    % & + \hat{\calP_1^\pi} \cdots \hat{\calP_{T-1}^\pi}[\hat{\Pi}_TQ_T - \hat{\calP_T^\pi}R_T + ( Q_T - \hat{\Pi}_TQ_T)] \\
    = & \mathcal{E}_1\{E_1\} +  \mathcal{E}_1\{E_2\} + \mathcal{E}_1\{E_3\},
\end{align*}
where 
\begin{align}
    E_1 &= \sum_{t=1}^T \left(\prod_{t'=0}^{t-1} \mathcal{P}_{t'}^\pi\right) \tilde{\Pi}_t \left[ Q^\pi_t - (R_t + \langle \pi_t, Q_{t+1}^\pi \rangle) \right], \label{eqn:def_E1} \\ 
     E_2 &= \sum_{t=1}^T \left( \left[ \prod_{t'=0}^{t-1} \hat{\mathcal{P}}_{t'}^\pi \right] \hat{\Pi}_t  - \left[ \prod_{t'=0}^{t-1} \mathcal{P}_{t'}^\pi \right] \tilde{\Pi}_t\right)  \left[ Q^\pi_t - (R_t + \langle \pi_t, Q_{t+1}^\pi \rangle) \right], \label{eqn:def_E2} \\ 
     E_3 &= \sum_{t=1}^T \left( \prod_{t'=0}^{t-1} \hat{\mathcal{P}}_{t'}^\pi\right) \left[ Q_t - \hat{\Pi}_t Q_t\right],\label{eqn:def_E3}
\end{align}
In the following, we focus on these three terms one by one. 

\subsection{Bounding $\mathcal{E}_1\{E_1\}$}
\label{sec:E1}
Note that 
\begin{align*}
    \EE  ( \mathcal{E}^\pi_1 [E_1]) & = \sum_{t=1}^T \EE \left(  \calE_1 \left\{ \left(\prod_{t'=0}^{t-1} \mathcal{P}_{t'}^\pi\right) \tilde{\Pi}_t \left[ Q^\pi_t - (R_t + \langle \pi_t, Q_{t+1}^\pi \rangle)  \right] \right\} \right) \\
   & =  \sum_{t=1}^T \EE \left[ \EE \left(  \calE_1 \left\{ \left(\prod_{t'=0}^{t-1} \mathcal{P}_{t'}^\pi\right) \tilde{\Pi}_t \left[ Q^\pi_t - (R_t + \langle \pi_t, Q_{t+1}^\pi \rangle)  \right] \right\}  \mid \calD_t \right) \right] = 0.
\end{align*}
Then it suffices to derive the bound for the variance of $\mathcal{E}_1\{E_1\}$.
\begin{align*}
    &\Var ( \mathcal{E}^\pi_1 [E_1(S_1, A_1)]) =  \Var \left( \mathcal{E}^\pi_1\left\{ \sum_{t=1}^T \left(\prod_{t'=0}^{t-1} \mathcal{P}_{t'}^\pi\right) \tilde{\Pi}_t \left[ Q^\pi_t - (R_t + \langle \pi_t, Q_{t+1}^\pi \rangle) \right]\right\} \right) \\
    &= \sum_{t=1}^T \EE \left\{ \Var \left(   \mathcal{E}^\pi_1 \left\{ \left(\prod_{t'=0}^{t-1} \mathcal{P}_{t'}^\pi\right) \tilde{\Pi}_t \left[ Q^\pi_t - (R_t + \langle \pi_t, Q_{t+1}^\pi \rangle) \right] \right\}  \mid \mathcal{D}_t\right) \right\} \\
   & = \sum_{t=1}^T \EE \left\{ \Var \left(   \mathcal{E}^\pi_t \left\{ \tilde{\Pi}_t \left[ Q^\pi_t - (R_t + \langle \pi_t, Q_{t+1}^\pi \rangle) \right] \right\}  \mid  \mathcal{D}_t\right) \right\} \\
    &=  \sum_{t=1}^T \EE \left\{ \Var \left(   [\mathcal{E}^\pi_t \phi_K]^\tp  { {\Sigma}_t}^{-1} \left( \frac{1}{n}\sum_{i=1}^n \phi_K(S_{i,t}, A_{i,t}) \left[ Q^\pi_{t}(S_{i,t}, A_{i,t}) - (R_{i,t} + \langle \pi_t, Q_{t+1}^\pi \rangle (S_{i,t+1})) \right]  \right) \mid  \mathcal{D}_t\right) \right\}  \cdots \cdots (i)
   \end{align*}
The first inequality is due to Lemma \ref{lem:decompoeision}.

Next we consider bounding $(i)$ under different conditions in Theorem \ref{thm:slow} and Theorem \ref{thm:fast}.

\begin{itemize}
    \item Under conditions in Theorem \ref{thm:slow}.
    \begin{align*}
        (i)  & =  \sum_{t=1}^T \EE \left\{ \Var \left(   [\mathcal{E}^\pi_t \phi_K]^\tp  { {\Sigma}_t}^{-1} \left( \frac{1}{n}\sum_{i=1}^n \phi_K(S_{i,t}, A_{i,t}) \left[ Q^\pi_{t}(S_{i,t}, A_{i,t}) - (R_{i,t} + \langle \pi_t, Q_{t+1}^\pi \rangle (S_{i,t+1})) \right]  \right) \mid  \mathcal{D}_t\right) \right\} \\
       & =  \sum_{t=1}^T \frac{1}{n} \EE \left\{   \Var \left(  [\mathcal{E}^\pi_t \phi_K]^\tp   { {\Sigma}_t}^{-1} \phi_K(S_{i,t}, A_{i,t})\left( \left[ Q^\pi_{t}(S_{i,t}, A_{i,t}) - (R_{i,t} + \langle \pi_t, Q_{t+1}^\pi \rangle (S_{i,t+1})) \right]  \right) \mid  \mathcal{D}_t\right) \right\} \\
       & \lesssim \sum_{t=1}^T \frac{1}{n} (T-t+1)^2 \EE \left\{  [\mathcal{E}^\pi_t \phi_K]^\tp   { {\Sigma}_t}^{-1} \phi_K(S_{i,t}, A_{i,t}) \phi^\tp_K(S_{i,t}, A_{i,t}) \Sigma_t^{-1} [\mathcal{E}^\pi_t \phi_K]
        \right\}\\
        & = \frac{1}{n} (T-t+1)^2  \sum_{t=1}^T  [\mathcal{E}^\pi_t \phi_K]^\tp  \Sigma_t^{-1} [\mathcal{E}^\pi_t \phi_K] \\
        & \leq \frac{T^3}{n} \kappa.
    \end{align*} 
    The second equality is due to the independence between different episodes and the last inequality is due to the definition of $\kappa$. 
    Therefore, we have 
    \begin{align*}
        \calE_1(E_1) = \bigOp\left( \sqrt{ \frac{T^3}{n} \kappa} \right).
    \end{align*}
    \item Under conditions in Theorem \ref{thm:fast}. 
    Take $\Delta_t(s,a,r, s') = Q_t^\pi(s,a) - [r + \langle \pi, Q_{t+1}^{\pi} \rangle (s')]. $

    \begin{align*}
        % \scriptsize{
      &   (i)  =  \sum_{t=1}^T \EE \left\{ \Var \left(   [\mathcal{E}^\pi_t \phi_K]^\tp  { {\Sigma}_t}^{-1} \left( \frac{1}{n}\sum_{i=1}^n \phi_K(S_{i,t}, A_{i,t}) \Delta_t(S_{i,t}, A_{i,t}, S_{i,t+1}, R_{i,t})  \right) \mid  \mathcal{D}_t\right) \right\} \\
           & =  \sum_{t=1}^T \frac{1}{n} \EE \left\{   \Var \left(   \left[\mathbb
   {E} \phi_K(S_t, A_t) \frac{\dist^\pi_t(S_t, A_t)}{\dist^b_t(S_t, A_t)}\right]^\tp  { {\Sigma}_t}^{-1} \phi_K(S_{i,t}, A_{i,t})\left( \Delta_t(S_{i,t}, A_{i,t}, S_{i,t+1}, R_{i,t})   \right) \mid  \mathcal{D}_t\right) \right\} \\
   &\leq  2\sum_{t=1}^T \frac{1}{n} \EE \left\{   \Var \left( \frac{\dist^\pi_t(S_{i,t}, A_{i,t})}{\dist^b_t(S_{i,t}, A_{i,t})}  \left( \Delta_t(S_{i,t}, A_{i,t}, S_{i,t+1}, R_{i,t})  \right) \mid S_{i,t}, A_{i,t} \right) \right\} \\
   &  +   \frac{2}{n} \sum_{t=1}^T \EE \left\{   \Var \left(\left[ \phi_K(S_{i,t}, A_{i,t})^\tp { {\Sigma}_t}^{-1} \mathbb
   {E} \phi_K(S_t, A_t) \frac{\dist^\pi_t(S_t, A_t)}{\dist^b_t(S_t, A_t)}- \frac{\dist^\pi_t(S_{i,t}, A_{i,t})}{\dist^b_t(S_{i,t}, A_{i,t})}\right]  \Delta_t(S_{i,t}, A_{i,t}, S_{i,t+1}, R_{i,t})  \mid  \mathcal{D}_t \right) \right\} \\
   % & \qquad 
   & \leq   2\sum_{t=1}^T \frac{1}{n}  \sup_{s,a} \frac{\dist^\pi_t(s,a)}{\dist^b_t(s,a)}  \EE \left\{ \frac{\dist^\pi_t(S_{i,t}, A_{i,t})}{\dist^b_t(S_{i,t}, A_{i,t})}   \Var \left( \left(   \Delta_t(S_{i,t}, A_{i,t}, S_{i,t+1}, R_{i,t}) \right) \mid S_{i,t}, A_{i,t} \right) \right\} \\
   &+   \frac{2}{n} \sum_{t=1}^T \EE \left\{   \Var \left(\left[ \phi_K(S_{i,t}, A_{i,t})^\tp { {\Sigma}_t}^{-1} \mathbb
   {E} \phi_K(S_t, A_t) \frac{\dist^\pi_t(S_t, A_t)}{\dist^b_t(S_t, A_t)}- \frac{\dist^\pi_t(S_{i,t}, A_{i,t})}{\dist^b_t(S_{i,t}, A_{i,t})}\right]  \Delta_t(S_{i,t}, A_{i,t}, S_{i,t+1}, R_{i,t})  \mid  \mathcal{D}_t \right) \right\} \\
   & \leq 2\sum_{t=1}^T \frac{1}{n}  \left[ \sup_{s,a} \frac{\dist^\pi_t(s,a)}{\dist^b_t(s,a)} \right] \EE^\pi \left\{  \Var \left( \left( \left[ Q^\pi_{t}(S_{i,t}, A_{i,t}) - (R_{i,t} + \langle \pi_t, Q_{t+1}^\pi \rangle (S_{i,t+1})) \right]  \right) \mid S_{i,t}, A_{i,t} \right) \right\} \cdots \cdots (ii) \\
   & +   \frac{2}{n} \sum_{t=1}^T \EE \left\{   \Var \left(\left[ \phi_K(S_{i,t}, A_{i,t})^\tp { {\Sigma}_t}^{-1} \mathbb
   {E} \phi_K(S_t, A_t) \frac{\dist^\pi_t(S_t, A_t)}{\dist^b_t(S_t, A_t)}- \frac{\dist^\pi_t(S_{i,t}, A_{i,t})}{\dist^b_t(S_{i,t}, A_{i,t})}\right]  \Delta_t(S_{i,t}, A_{i,t}, S_{i,t+1}, R_{i,t})  \mid  \mathcal{D}_t \right) \right\} \\
   & \qquad \qquad  \cdots\cdots (iii).
%    & \leq  2\left[ \sup_{s,a,t} \frac{d^\pi_t(s,a)}{d^b_t(s,a)} \right] \sum_{t=1}^T \frac{1}{n}  \EE^\pi \left\{  \Var \left( \left( \left[ Q^\pi_{t}(S_{i,t}, A_{i,t}) - (R_{i,t} + \langle \pi_t, Q_{t+1}^\pi \rangle (S_{i,t+1})) \right]  \right) \mid S_{i,t}, A_{i,t} \right) \right\} \\
%    & \qquad +   2\sum_{t=1}^T \frac{1}{n} \EE \left\{   \Var \left(\left[ \phi_K(S_{i,t}, A_{i,t})^\tp { {\Sigma}_t}^{-1} \mathbb
%    {E} \phi_K(S_t, A_t) \frac{d^\pi_t(S_t, A_t)}{d^b_t(S_t, A_t)}- \frac{d^\pi_t(S_{i,t}, A_{i,t})}{d^b_t(S_{i,t}, A_{i,t})}\right]   \left( \Delta_t(S_{i,t}, A_{i,t}, S_{i,t+1})  \right) \mid  \mathcal{D}_t \right) \right\} \\
%    & \lesssim 2 \left[ \sup_{s,a,t} \frac{d^\pi_t(s,a)}{d^b_t(s,a)} \right]  \frac{T^2}{n} \\
%     & \qquad  +   2\sum_{t=1}^T \frac{1}{n} \EE \left\{   \Var \left(\left[ \phi_K(S_{i,t}, A_{i,t})^\tp { {\Sigma}_t}^{-1} \mathbb
%    {E} \phi_K(S_t, A_t) \frac{d^\pi_t(S_t, A_t)}{d^b_t(S_t, A_t)}- \frac{d^\pi_t(S_{i,t}, A_{i,t})}{d^b_t(S_{i,t}, A_{i,t})}\right]   \left( \Delta_t(S_{i,t}, A_{i,t}, S_{i,t+1})  \right) \mid  \mathcal{D}_t \right) \right\}
% }
    \end{align*}
    
    The first equality is due to the independence among episodes and the fact that 
    $$\calE_t^\pi \phi_K = \mathbb
    {E} \phi_K(S_t, A_t) \frac{\dist^\pi_t(S_t, A_t)}{\dist^b_t(S_t, A_t)}.$$

By applying Lemma  3.4 in \citet{yin2020asymptotically}, %and Assumprion \ref{ass: reward}, 
we have 
\begin{align*}
    (ii) &  \leq \frac{2}{n} \left[ \sup_{s,a,t} \frac{\dist^\pi_t(s,a)}{\dist^b_t(s,a)} \right] \sum_{t=1}^T   \EE^\pi \left\{  \Var \left( \left( \left[ Q^\pi_{t}(S_{i,t}, A_{i,t}) - (R_{i,t} + \langle \pi_t, Q_{t+1}^\pi \rangle (S_{i,t+1})) \right]  \right) \mid S_{i,t}, A_{i,t} \right) \right\} \\
& \leq \frac{2}{n} \left[ \sup_{s,a,t} \frac{\dist^\pi_t(s,a)}{\dist^b_t(s,a)} \right]  \Var^\pi\left( \sum_{t=1}^T R_t \right) \leq \frac{2}{n} \left[ \sup_{s,a,t} \frac{\dist^\pi_t(s,a)}{\dist^b_t(s,a)} \right] T^2 \lesssim \frac{T^2}{n}. 
\end{align*}
As for $(iii)$, we have $|\Delta_t(S_{i,t}, A_{i,t}, S_{i,t+1}, R_{i,t}) |\lesssim (T-t+1)$ due to that $|R_t|\leq 1$, and we obtain 
\begin{align*}
   (iii) %& \sum_{t=1}^T \frac{1}{n} \EE \left\{   \Var \left(\left[ \phi_K(S_{i,t}, A_{i,t})^\tp { {\Sigma}_t}^{-1} \mathbb
%    {E} \phi_K(S_t, A_t) \frac{d^\pi_t(S_t, A_t)}{d^b_t(S_t, A_t)}- \frac{d^\pi_t(S_{i,t}, A_{i,t})}{d^b_t(S_{i,t}, A_{i,t})}\right]   \left( \Delta_t(S_{i,t}, A_{i,t}, S_{i,t+1}, R_{i,t})  \right) \mid  \mathcal{D}_t \right) \right\}\\
   % \leq  & \sum_{t=1}^T \frac{1}{n} \left\| \frac{d^\pi_t}{d^b_t} - \Pi \left( \frac{d^\pi_t}{d^b_t}\right)\right\|_{\infty}^2 \EE \left\{   \Var \left( \Delta_t(S_{i,t}, A_{i,t}, S_{i,t+1})   \mid  \mathcal{D}_t \right) \right\}\\ 
   \lesssim & \sum_{t=1}^T \frac{1}{n} (T-t +2)^2 \EE \left[ \frac{\dist^\pi_t}{\dist^b_t} (S_t, A_t)- \Pi \left( \frac{\dist^\pi_t}{\dist^b_t}\right)(S_t, A_t)\right]^2 \\
   \leq& \frac{T^2}{n} \sum_{t=1}^T \left\| \frac{\dist^\pi_t}{\dist^b_t} - \Pi \left( \frac{\dist^\pi_t}{\dist^b_t}\right)\right\|_{\calL_2}^2  \leq \frac{T^3}{n} K^{-\beta_w}.
   \end{align*}   
Overall, we have 
\begin{align*}
    \calE_1(E_1) = \bigOp\left( \frac{T}{\sqrt{n}} + \sqrt{\frac{T^3 K^{-\beta_w}}{n}} \right).
\end{align*}
\end{itemize}

\subsection{Bounding $\calE_1(E_2)$}
\label{sec:E2}
\begin{align*}
    \mathcal{E}_1^\pi (E_2) &=   \sum_{t=1}^T  \mathcal{E}_1^\pi \left\{ \left( \left[ \prod_{t'=0}^{t-1} \hat{\mathcal{P}}_{t'}^\pi \right] \hat{\Pi}_t  - \left[ \prod_{t'=0}^{t-1} \mathcal{P}_{t'}^\pi \right] \tilde{\Pi}_t\right)  \left[ Q^\pi_t - (R_t + \langle \pi_t, Q_{t+1}^\pi \rangle) \right]  \right\} \\
   &  =   \sum_{t=1}^T  \mathcal{E}_1^\pi \left\{ \left(  \prod_{t'=0}^{t-1} \hat{\mathcal{P}}_{t'}^\pi    - \prod_{t'=0}^{t-1} \mathcal{P}_{t'}^\pi \right)  \hat{\Pi}_t\left[ Q^\pi_t - (R_t + \langle \pi_t, Q_{t+1}^\pi \rangle) \right]  \right\} \cdots \cdots (i) \\
   & \quad + \sum_{t=1}^T  \mathcal{E}_1^\pi \left\{ \prod_{t'=0}^{t-1} \mathcal{P}_{t'}^\pi \left( \hat{\Pi}_t  - \tilde{\Pi}_t \right)\left[ Q^\pi_t - (R_t + \langle \pi_t, Q_{t+1}^\pi \rangle) \right]  \right\} \cdots \cdots (ii)
  \end{align*}

  Let's first deal with term $(ii)$. For every $t$, we have
  \begin{align*}
      & \EE \left[  \mathcal{E}_1^\pi \left\{ \prod_{t'=0}^{t-1} \mathcal{P}_{t'}^\pi \left( \hat{\Pi}_t  - \tilde{\Pi}_t \right)\left[ Q^\pi_t - (R_t + \langle \pi_t, Q_{t+1}^\pi \rangle) \right]   \right\} \mid \mathcal{D}_t\right] = 0\\
      & \EE \left[  \mathcal{E}_1^\pi \left\{ \prod_{t'=0}^{t-1} \mathcal{P}_{t'}^\pi \left( \hat{\Pi}_t  - \tilde{\Pi}_t \right)\left[ Q^\pi_t - (R_t + \langle \pi_t, Q_{t+1}^\pi \rangle) \right]   \right\} \right] = 0
  \end{align*}
  Therefore, we consider the variance of 
  \begin{align*}
      \sum_{t=1}^T  \mathcal{E}_1^\pi \left\{ \prod_{t'=0}^{t-1} \mathcal{P}_{t'}^\pi \left( \hat{\Pi}_t  - \tilde{\Pi}_t \right)\left[ Q^\pi_t - (R_t + \langle \pi_t, Q_{t+1}^\pi \rangle) \right]  \right\}.
  \end{align*}
  Using a similar argument as in Lemma \ref{lem:decompoeision}, we can decompose the variance as 
  \begin{align*}
     & \Var\left[ \sum_{t=1}^T  \mathcal{E}_1^\pi \left\{ \prod_{t'=0}^{t-1} \mathcal{P}_{t'}^\pi \left( \hat{\Pi}_t  - \tilde{\Pi}_t \right)\left[ Q^\pi_t - (R_t + \langle \pi_t, Q_{t+1}^\pi \rangle) \right]  \right\}\right] \\
      = & \sum_{t=1}^T \EE \left\{ \Var\left[  \mathcal{E}_1^\pi \left\{ \prod_{t'=0}^{t-1} \mathcal{P}_{t'}^\pi \left( \hat{\Pi}_t  - \tilde{\Pi}_t \right)\left[ Q^\pi_t - (R_t + \langle \pi_t, Q_{t+1}^\pi \rangle) \right]  \right\}\mid \mathcal{D}_t \right] \right\} \\
      = & \sum_{t=1}^T \EE \left\{ \Var\left[  \mathcal{E}_t^\pi \left\{ \left( \hat{\Pi}_t  - \tilde{\Pi}_t \right)\left[ Q^\pi_t - (R_t + \langle \pi_t, Q_{t+1}^\pi \rangle) \right]  \right\}\mid \mathcal{D}_t \right] \right\} \\
      = & \sum_{t=1}^T \frac{1}{n} \EE \left\{ \left[ \EE \left(\phi_K(S_t, A_t) \frac{\dist^\pi_t(S_t, A_t)}{\dist^b_t(S_t, A_t)}  \right)^\tp\Sigma_t^{-1/2}\left(\Sigma_t^{1/2}\hat{\Sigma}_t^{-1}\Sigma_t^{1/2} - I_K  \right)  \Sigma_t^{-1/2} \phi_K(S_t, A_t)  \right]^2  \right. \\
      & \qquad \left. \Var\left(  Q^\pi_t(S_{t}, A_{t}) - (R_t + \langle \pi_t, Q_{t+1}^\pi(S_{t+1}) \rangle)  \mid \mathcal{D}_t \right) \right\} \\ 
      \leq & \sum_{t=1}^T \frac{(T-t+2)^2}{n} \EE \left\{ \left\| \EE \left(\phi_K(S_t, A_t) \frac{\dist^\pi_t(S_t, A_t)}{\dist^b_t(S_t, A_t)}  \right)\Sigma_t^{-1/2}\right\|^2_2  \left\| \Sigma_t^{1/2}\hat{\Sigma}_t^{-1}\Sigma_t^{1/2} - I_K\right\|^2_2 \left\|   \Sigma_t^{-1/2} \phi_K(S_t, A_t) \right\|^2_2\right\} \\
      \leq  & \sum_{t=1}^T \frac{(T-t+2)^2}{n} [\mathcal{E}_t^\pi \phi_K]^\tp \Sigma_t^{-1}  [\mathcal{E}_t^\pi \phi_K] \left(\frac{\zeta_K^2 \log n}{n} \right) \zeta_K^2 % = \bigOp \left( \frac{T^3 \zeta_K^4 \log n}{n^2}\right).
  \end{align*}
  The last inequality is by applying Lemma \ref{lem:matrix} and the definition of $\zeta_K$.
  Therefore, we obtain 
  \begin{align*}
    (ii) = \bigO \left( \sqrt{\frac{T^3 \zeta_K^4 \log n}{n^2} } \right), \text{ W.H.P}.
  \end{align*}

Next, we derive the bound for term $(i)$. 

First of all, notice that %$$\|\mathcal{P}^\pi_t \|_{\pi} \leq 1, \forall t.$$
$$\|\mathcal{P}^\pi_t f \|_{\infty} = \sup_{s,a} \EE [f^\pi(S') \mid S=s, A= a] \leq \sup_{s,a} \EE [ |f^\pi(S')| \mid S=s, A= a] \leq \sup_{s'} f^\pi(s') \leq \sup_{(s',a')}f(s',a'). $$
Therefore
$$\|\mathcal{P}^\pi_t \|_{\infty} \leq 1. $$

\begin{align*}
    & \mathcal{E}_1^\pi \left\{ \left( \prod_{t'=0}^{t-1} \hat{\mathcal{P}}_{t'}^\pi - \prod_{t'=0}^{t-1} \mathcal{P}_{t'}^\pi\right)\hat{\Pi}_t \left[ Q^\pi_t - (R_t + \langle \pi_t, Q_{t+1}^\pi \rangle) \right] \right\} \\
     = & \mathcal{E}_1^\pi \left\{ \left( \prod_{t'=0}^{t-1} (\mathcal{P}_{t'}^\pi +  \hat{\mathcal{P}}_{t'}^\pi - \mathcal{P}_{t'}^\pi )- \prod_{t'=0}^{t-1} \mathcal{P}_{t'}^\pi\right) \hat{\Pi}_t \left[ Q^\pi_t - (R_t + \langle \pi_t, Q_{t+1}^\pi \rangle) \right] \right\} \\
      = & \mathcal{E}_1^\pi \left\{ \left( \sum_{(\delta_{t,0}, \dots, \delta_{t,t-1}) \in \{0,1\}^t\setminus\{0\}^t }  (\mathcal{P}_{0}^\pi)^{1-\delta_{t,0}} (\hat{\mathcal{P}}_{0}^\pi - \mathcal{P}_{0'}^\pi )^{\delta_{t,0}} \cdots  (\mathcal{P}_{t-1}^\pi)^{1-\delta_{t,t-1}} (\hat{\mathcal{P}}_{t-1}^\pi - \mathcal{P}_{t-1}^\pi )^{\delta_{t,t-1}} \right) \right. \\
      & \qquad \qquad \left. \hat{\Pi_t}  \left[ Q^\pi_t - (R_t + \langle \pi_t, Q_{t+1}^\pi \rangle) \right] \right\} \\
      \leq  &  \left\| \left( \sum_{(\delta_{t,0}, \dots, \delta_{t,t-1}) \in \{0,1\}^t\setminus\{0\}^t }  (\mathcal{P}_{0}^\pi)^{1-\delta_{t,0}} (\hat{\mathcal{P}}_{0}^\pi - \mathcal{P}_{0'}^\pi )^{\delta_{t,0}} \cdots  (\mathcal{P}_{t-1}^\pi)^{1-\delta_{t,t-1}} (\hat{\mathcal{P}}_{t-1}^\pi - \mathcal{P}_{t-1}^\pi )^{\delta_{t,t-1}} \right) \right. \\
      & \qquad \qquad \left. \hat{\Pi}_t \left[ Q^\pi_t - (R_t + \langle \pi_t, Q_{t+1}^\pi \rangle) \right] \right\|_{\infty} \\
      \leq & \sum_{(\delta_{t,0}, \dots, \delta_{t,t-1}) \in \{0,1\}^t\setminus\{0\}^t } \|\mathcal{P}_{0}^\pi\|_{\infty}^{1-\delta_{t,0}} \|\hat{\mathcal{P}}_{0}^\pi - \mathcal{P}_{0}^\pi \|_{\infty}^{\delta_{t,0}} \cdots  \|\mathcal{P}_{t-1}^\pi\|_{\infty}^{1-\delta_{t,t-1}} \|\hat{\mathcal{P}}_{t-1}^\pi - \mathcal{P}_{t-1}^\pi \|_{\infty}^{\delta_{t,t-1}} \\
     & \qquad \qquad \| \hat{\Pi}_t  \left[ Q^\pi_t - (R_t + \langle \pi_t, Q_{t+1}^\pi \rangle) \right] \|_{\infty}\\
     \leq & \left\{  \left(\|\mathcal{P}_{0}^\pi\|_{\infty} +  \|\hat{\mathcal{P}}_{0}^\pi - \mathcal{P}_{0}^\pi \|_{\infty} \right) \cdots \left(\|\mathcal{P}_{t-1}^\pi\|_{\infty} +  \|\hat{\mathcal{P}}_{t-1}^\pi - \mathcal{P}_{t-1}^\pi \|_{\infty} \right) - \|\mathcal{P}_{0}^\pi\|_{\infty}\cdots \|\mathcal{P}_{t-1}^\pi\|_{\infty}\right\}\\
      & \qquad \qquad \|  \hat{\Pi}_t \left[ Q^\pi_t - (R_t + \langle \pi_t, Q_{t+1}^\pi \rangle) \right] \|_{\infty}\\
       \leq & \left\{  \left(1 +  \|\hat{\mathcal{P}}_{0}^\pi - \mathcal{P}_{0}^\pi \|_{\infty} \right) \cdots \left(1 +  \|\hat{\mathcal{P}}_{t-1}^\pi - \mathcal{P}_{t-1}^\pi \|_{\infty} \right) - 1\right\}\|  \hat{\Pi}_t  \left[ Q^\pi_t - (R_t + \langle \pi_t, Q_{t+1}^\pi \rangle) \right] \|_{\infty}
 \end{align*}

 From the argument in Section \ref{sec:operator}, with probability at least $1- 4 T [(nK)^{-2} - c_1\exp\left\{\beta \log n - c_2\zeta_K^{-1} \sqrt{n} \right\} ]$, we have that for any $t=1,\dots, T$
 \begin{multline*}
     \|\hat{\mathcal{P}}_{t}^\pi - \mathcal{P}_{t}^\pi \|_{\infty} \\
     \lesssim \left( 1 + \frac{\zeta^2_K \sqrt{\log n \log K}}{\sqrt{n}}\right)  \sup_{h \in \mathcal{Q}^{(t)}(1)}  \| h - \Pi h\|_{\infty} + \frac{\zeta_K}{\sqrt{n}} +   K^{\frac{-1}{2}} \frac{\zeta^2_K \sqrt{\log n \log K}}{\sqrt{n}}+ \frac{\zeta_K^3 \log n \log K}{n}.
 \end{multline*}
 In addition, we have  
 \begin{multline*}
    \|  \hat{\Pi}_t \left[ Q^\pi_t - (R_t + \langle \pi_t, Q_{t+1}^\pi \rangle) \right] \|_{\infty} \\
    =\bigO\left(  \|Q^\pi_t - (R_t + \langle \pi_t, Q_{t+1}^\pi \rangle)\|_{\infty} \left[\frac{\zeta_K}{\sqrt{n}} +   K^{\frac{-1}{2}} \frac{\zeta^2_K \sqrt{\log n \log K}}{\sqrt{n}}+ \frac{\zeta_K^3 \log n \log K}{n} \right] \right), \text{ W.H.P}.
 \end{multline*}
 Denote 
 \begin{align*}
     \xi_{1,n,K} & = \left( 1 + \frac{\zeta^2_K \sqrt{\log n \log K}}{\sqrt{n}}\right)  \sup_{h \in \mathcal{Q}^{(t)}(1)}  \| h - \Pi h\|_{\infty} \\
     \xi_{2,n,K} & =  \frac{\zeta_K}{\sqrt{n}} +   K^{\frac{-1}{2}} \frac{\zeta^2_K \sqrt{\log n \log K}}{\sqrt{n}}+ \frac{\zeta_K^3 \log n \log K}{n}
 \end{align*}
\begin{itemize}
    \item Without further condition on $T$, we bound $\calE_1(E_2)$ by 
    \begin{align*}
        \calE_1(E_2) = \bigO\left( \sum_{t=1}^T  \left\{\left[ (\xi_{1,n,K} + \xi_{2,n,K}) + 1 \right]^{t} -1 \right\} \xi_{2,n,K} \right), \text{ W.H.P}.
    \end{align*}
    \item 
    Under the condition \eqref{eqn:T_requirement},  we have  
    \begin{align*}
        \xi_{1, n,K} + \xi_{2,n,K} < 1/T,
    \end{align*}
and therefore 
   \begin{align*}
        \left\{  \left(1 +  \|\hat{\mathcal{P}}_{0}^\pi - \mathcal{P}_{0}^\pi \|_{\infty} \right) \cdots \left(1 +  \|\hat{\mathcal{P}}_{t-1}^\pi - \mathcal{P}_{t-1}^\pi \|_{\infty} \right) - 1\right\} \\
       \lesssim t (\xi_{1,n,K} + \xi_{2,n,K}).
   \end{align*}
   Then the term $E_2$ is bounded by 
   \begin{multline*}
       \mathcal{E}_1^\pi(E_2) \lesssim \sum_{t=1}^T \left\{ t (\xi_{1,n,K} + \xi_{2,n, K}) \| Q^\pi_t - (R_t + \langle \pi_t, Q_{t+1}^\pi \rangle)\|_{\infty}\xi_{2,n,K}  \right\} \\
       + \sum_{t=1}^T  \sup_{s,a} \frac{\dist^\pi_t(s,a)}{\dist^b_t(s,a)}\frac{\zeta_K^2 \log K \log n}{n} \left\|  Q^\pi_t - (R_t + \langle \pi_t, Q_{t+1}^\pi\rangle)  \right\|_{\infty}\\
       \lesssim T^2 (\xi_{1,n,K} + \xi_{2,n, K}) \xi_{2,n,K}\left[ \sup_{t= 1,\dots,T}\| Q^\pi_t - (R_t + \langle \pi_t, Q_{t+1}^\pi \rangle)\|_{\infty}\right]  ,
   \end{multline*}
   with probability at least $1- 4 T [(nK)^{-2} - c_1\exp\left\{\beta \log n - c_2\zeta_K^{-1} \sqrt{n} \right\} ]$.  
   And we have 
   \begin{align*}
    \mathcal{E}_1^\pi(E_2)  = \bigO\left( T^3 (\xi_{1,n,K} + \xi_{2,n, K}) \xi_{2,n,K} \right), \text{ W.H.P}.
   \end{align*}.
\end{itemize}

\subsection{Bounding $\calE_1(E_3)$}
\label{sec:E3}
\begin{align}
   \calE_1( E_3) = & \calE_1 \left\{\sum_{t=1}^T \left( \prod_{t'=0}^{t-1} \hat{\mathcal{P}}_{t'}^\pi\right) \left[ Q_t - \hat{\Pi}_t Q_t\right]   \right\} \nonumber \\
    = & \calE_1 \left\{\sum_{t=1}^T \left( \prod_{t'=0}^{t-1} {\mathcal{P}}_{t'}^\pi\right) \left[ Q_t - \hat{\Pi}_t Q_t\right]\right\} + \calE_1 \left\{\sum_{t=1}^T \left( \prod_{t'=0}^{t-1} \hat{\mathcal{P}}_{t'}^\pi - \prod_{t'=0}^{t-1} {\mathcal{P}}_{t'}^\pi \right) \left[ Q_t - \hat{\Pi}_t Q_t\right] \right\} \label{eqn:bias_decomp}
    \end{align}
    For the first component in \eqref{eqn:bias_decomp}, 
    \begin{align*}
       \mathcal{E}_1^\pi \left\{ \sum_{t=1}^T \left( \prod_{t'=0}^{t-1} {\mathcal{P}}_{t'}^\pi\right) \left[ Q_t - \hat{\Pi}_t Q_t\right] \right\} = \sum_{t=1}^T  \mathcal{E}_t^\pi \left[ Q_t - \hat{\Pi}_t Q_t\right] 
    \end{align*}
    
    \begin{align*}
     \left| \sum_{t=1}^T  \mathcal{E}_t^\pi \left[ Q_t - \hat{\Pi}_t Q_t\right] \right| 
        =  & \left| \sum_{t=1}^T  \mathcal{E}_t^\pi \left[ \Pi Q_t - \hat{\Pi}_t Q_t + Q_t - \Pi Q_t\right] \right| \\
        \leq & \left| \sum_{t=1}^T  \mathcal{E}_t^\pi \left[ \Pi Q_t - \hat{\Pi}_t Q_t\right] \right|  + \left| \sum_{t=1}^T  \mathcal{E}_t^\pi \left[ Q_t - \Pi Q_t\right] \right|  \\
        \leq & (i) + (ii)
    %     \leq & \left| \sum_{t=1}^T  \mathcal{E}_t^\pi \phi_K(\cdot, \cdot)^\tp \right|  + \left| \sum_{t=1}^T  \mathcal{E}_t^\pi \left[ Q_t - \Pi Q_t\right] \right|  \\
    %    = & \bigOp\left( \sum_{t=1}^T \left\{  \frac{(\phi_K^\pi)^\tp \Sigma^{-1} \phi_K^\pi}{\sqrt{n}} \| Q_t - \Pi Q_t\|_{\infty}  +  \left| \sum_{t=1}^T  \mathcal{E}_t^\pi \left[ Q_t - \Pi Q_t\right] \right| \right\}\right)
    \end{align*}
    % where $\phi_K^\pi = \EE^\pi[\phi_K(S_t, A_t)]$. The last equality is derived by using a similar argument in Lemma \ref{lem:point}.

    For term $(i)$, note that 
    \begin{align*}
        \Pi_t Q_t (s,a) = \phi_K(s,a)^\tp C_{t,K}
    \end{align*}
   where  $C_{t, K} = (\Sigma_t)^{-1} \EE \phi_K(S_t, A_t) Q^\pi_t(S_t, A_t) $, and 
   \begin{align*}
    &\EE \phi_K(S_{i,t}, A_{i,t}) \left[  Q_{t}^\pi(S_{i,t}, A_{i,t}) - \phi_K^\tp(S_{i,t}, A_{i,t})^\tp C_{t,K}  \right] \\
    = & \EE \phi_K(S_{i,t}, A_{i,t}) \left[  Q_{t}^\pi(S_{i,t}, A_{i,t}) - \phi_K^\tp(S_{i,t}, A_{i,t})^\tp  (\Sigma_t)^{-1} \EE \phi_K(S_t, A_t) Q^\pi_t(S_t, A_t) \right]\\
    = & \EE \phi_K(S_{i,t}, A_{i,t}) \left[  Q_{t}^\pi(S_{i,t}, A_{i,t})\right] - \phi_K(S_{t}, A_{t}) \left[  Q_{t}^\pi(S_{t}, A_{t}) \right] = 0.
   \end{align*}
   we have
    % \begin{align*}
    %     \sum_{t=1}^T \left[ \sup_{s,a} \frac{d^\pi_t(s,a)}{d^b_t(s,a)}\right] \frac{1}{\sqrt{n}} \| Q_t - \Pi Q_t\|_{\infty} 
    % \end{align*}
\begin{align*}
    & \left| \sum_{t=1}^T  \mathcal{E}_t^\pi \left[ \Pi Q_t - \hat{\Pi}_t Q_t\right] \right| \\
    \leq & \left| \sum_{t=1}^T \mathcal{E}_t^\pi \left[ \phi_K^\tp(\cdot, \cdot) (\hat \Sigma_t)^{-1} \left\{ \frac{1}{n} \sum_{i=1}^n \phi_K(S_{i,t}, A_{i,t})Q_t^\pi(S_{i,t}, A_{i,t})  \right\}  - \Pi_t Q_t^\pi \right] \right| \\
    % \leq & \sum_{t=1}^T \left\|  \phi_K^\tp(\cdot, \cdot) (\hat \Sigma_t)^{-1} \left\{ \frac{1}{n} \sum_{i=1}^n \phi_K(S_{i,t}, A_{i,t})Q_t^\pi(S_{i,t}, A_{i,t})  \right\}  - \Pi_t Q_t^\pi  \right\|
    =  & \left| \sum_{t=1}^T \mathcal{E}_t^\pi \left[ \phi_K^\tp(\cdot, \cdot) (\hat \Sigma_t)^{-1} \left\{ \frac{1}{n} \sum_{i=1}^n \phi_K(S_{i,t}, A_{i,t}) [Q_t^\pi(S_{i,t}, A_{i,t})- \phi_K^\tp (S_{i,t}, A_{i,t}) C_{t,K}]   \right. \right. \right. \\
    &\qquad \left. \left. \left.  - \EE \phi_K(S_{i,t}, A_{i,t}) \left[  Q_{t}^\pi(S_{i,t}, A_{i,t}) - \phi_K^\tp(S_{i,t}, A_{i,t})^\tp C_{t,K}  \right]  \right\} \right] \right| \\
    =  & \left| \sum_{t=1}^T \mathcal{E}_t^\pi \left[ \phi_K^\tp(\cdot, \cdot) (\hat \Sigma_t)^{-1} \left\{ \frac{1}{n} \sum_{i=1}^n \phi_K(S_{i,t}, A_{i,t}) [Q_t^\pi(S_{i,t}, A_{i,t})- \Pi_tQ_t^\pi(S_{i,t}, A_{i,t})]   \right. \right. \right. \\
    &\qquad \left. \left. \left.  - \EE \phi_K(S_{i,t}, A_{i,t}) \left[  Q_{t}^\pi(S_{i,t}, A_{i,t}) - \Pi_tQ_t^\pi(S_{i,t}, A_{i,t})  \right]  \right\} \right] \right|. 
\end{align*}
Using similiar argument in deriving the bound \eqref{eqn:bias_bound1}, replacing $\calP_t^\pi f$ with $Q_t^\pi - \Pi_t Q_t^\pi / \|Q_t^\pi - \Pi_t Q_t^\pi\|_{\infty}$, we have 
\begin{align*}
    & \left| \sum_{t=1}^T  \mathcal{E}_t^\pi \left[ \Pi Q_t - \hat{\Pi}_t Q_t\right] \right| \\
    \leq & \bigO\left( \sum_{t=1}^T \calE_t^\pi \left\| \calE_1\{\phi_K^\tp\}\Sigma_t^{-1/2} \right\|_2 \zeta_K \sqrt{\frac{\log n \log T}{n}} {\left\| Q_t^\pi - \Pi_t Q_{t}^\pi\right\|_{\infty}} \right) \\
    % = & \bigOp \left( \sum_{t=1}^T  \frac{ \calE_t^\pi \| \phi_K(\cdot, \cdot)^\tp \Sigma_t^{-1/2}\|_2\zeta_K \sqrt{\log n \log T}}{\sqrt{n}}   {\left\| Q_t^\pi - \Pi_t Q_{t}^\pi\right\|_{\infty}}  \right) \\
    = & \bigO\left( T^2 \kappa  \frac{\zeta_K \sqrt{\log n \log T}}{\sqrt{n}} K^{-\beta_Q}  \right), \text{ W.H.P}.
\end{align*}

    For term $(ii)$, we derive the bound under different conditions. 
    \begin{itemize}
        \item Under conditions in Theorem \ref{thm:slow}.
        \begin{align*}
            (ii) \leq \sum_{t=1}^T \|Q_t^\pi - \Pi_t Q_t^\pi\|_{\infty} \leq T^2 K^{-\beta_Q}.
        \end{align*}
        \item Under conditions in Theorem \ref{thm:fast}.
        
        First of all, note that 
        \begin{align*}
            \EE \left\{\phi_K(S_t, A_t)\left[ Q_t(S_t, A_t) - \Pi Q_t (S_t, A_t)\right]   \right\} = \bm 0
        \end{align*}
        due to the definition of $\Pi_t$. Then we have 
        \begin{align*}
           & \sum_{t=1}^T \mathcal{E}_t^\pi  \left[ Q^\pi_t - \Pi Q^\pi_t\right] \\
            = & \sum_{t=1}^T \EE \left\{ \frac{\dist^\pi_t(S_t,A_t)}{\dist^b_t(S_t,A_t)}  \left[ Q_t(S_t, A_t) - \Pi Q_t (S_t, A_t)\right]  \right\} \\
          = & \sum_{t=1}^T \EE \left\{ \Pi_t\left\{\frac{\dist^\pi_t}{\dist^b_t}  \right\}(S_t, A_t)  \left[ Q_t(S_t, A_t) - \Pi Q_t (S_t, A_t)\right]  \right\} \\
         &  + \sum_{t=1}^T \EE \left\{ \left( \frac{\dist^\pi_t(S_t,A_t)}{\dist^b_t(S_t,A_t)}-  \Pi_t\left\{\frac{\dist^\pi_t}{\dist^b_t}  \right\}(S_t, A_t) \right) \left[ Q_t(S_t, A_t) - \Pi Q_t (S_t, A_t)\right]  \right\} \\ 
         = &  0 + \sum_{t=1}^T \EE \left\{ \left( \frac{\dist^\pi_t(S_t,A_t)}{\dist^b_t(S_t,A_t)}-  \Pi_t\left\{\frac{\dist^\pi_t}{\dist^b_t}  \right\}(S_t, A_t) \right) \left[ Q_t(S_t, A_t) - \Pi Q_t (S_t, A_t)\right]  \right\} \\
         \leq & \sum_{t=1}^T \left\| \frac{\dist^\pi_t}{\dist^b_t} - \Pi_t \left[ \frac{\dist^\pi_t}{\dist^b_t}\right]\right\|_{\calL_2} \left\| Q_t - \Pi_t Q_t\right\|_{\calL_2}\\
         \leq & T^2 K^{-\beta_w} K^{-\beta_Q}.
        \end{align*}
        The last equality is due to the fact that $\Pi_t \{w_t\} (s,a) = \phi_K(s,a)^\tp \omega $ for some $\omega \in \mathbb{R}^K$.
    \end{itemize}

    For the second component in \eqref{eqn:bias_decomp}, using a similar idea as in bounding $E_2$, we have 
    \begin{align*}
       \sum_{t=1}^T \left( \prod_{t'=0}^{t-1} \hat{\mathcal{P}}_{t'}^\pi - \prod_{t'=0}^{t-1} {\mathcal{P}}_{t'}^\pi \right) \left[ Q_t - \hat{\calP_t} Q_t\right] \leq  \sum_{t=1}^T \bigO \left[ t(\xi_{1,n,K} + \xi_{2,n,K}) \right] \left\| Q_t - \hat{\calP_t} Q_t \right\|_{\infty} \\
       \leq \sum_{t=1}^T  \bigO\left( t (\xi_{1,n,K} + \xi_{2,n, K}) \left( 1 + \frac{\zeta^2_K \sqrt{\log n \log K}}{\sqrt{n}}\right)\| Q_t - \Pi Q_t\|_{\infty}\right) \\
        \leq \bigO\left\{ (\xi_{1,n,K} + \xi_{2,n, K}) \left( 1 + \frac{\zeta^2_K \sqrt{\log n \log K}}{\sqrt{n}}\right) \left( \sum_{t=1}^T t \| Q_t - \Pi Q_t\|_{\infty}\right) \right\} \\
        \leq \bigO\left( T^3 (\xi_{1,n,K} + \xi_{2,n, K}) K^{-\beta_Q} \right), \text{ W.H.P}.
    \end{align*}
    The last inequality is due to the condition of $\zeta_K$.

By combining results from Section \ref{sec:E1}, \ref{sec:E2} and \ref{sec:E3}, we obtain the bounds in Theorem \ref{thm:slow} and \ref{thm:fast}.

\subsection{Bounding $\|\hat{\mathcal{P}}_{t}^\pi - \mathcal{P}_{t}^\pi \|_{\infty}$}
\label{sec:operator}

% Take $\Pi h_0 = \argmin_{h \in \Phi_K} \| h - h_0 \|_{2}$.

\begin{align*}
   & [\hat{\mathcal{P}}_{t}^\pi] f (s,a) - [\mathcal{P}_t^\pi] f (s,a)\\
    = & \phi_K(s,a)^\tp (\hat{\Sigma})^{-1}\left\{ \frac{1}{n}\sum_{i=1}^n \phi_K(S_i, A_i)  \mathcal{P}_{t}^\pi f(S_i, A_i)\right\} - \Pi(\mathcal{P}_{t}^\pi f) (s,a) 
     + \Pi(\mathcal{P}_{t}^\pi f) (s,a) - (\mathcal{P}_{t}^\pi f) (s,a) \\
    & + \phi_K(s,a)^\tp (\hat{\Sigma})^{-1}\left\{ \frac{1}{n}\sum_{i=1}^n \phi_K(S_i, A_i)  \left[ f^\pi(S_i') -  \mathcal{P}_{t}^\pi f(S_i, A_i) \right]\right\} \\
    = & I + II
\end{align*}
 where $I$ indicates the bias term and $II$ indicates the variance term. For the following, we constrain $f \in \mathcal{Q}^{(t+1)}$ with $\|f\|_{\infty} \leq 1$. Here, we omit te subscirpt $t$ for $\Pi_t$ and $\Sigma_t$. 
\subsubsection{Bounding the bias term}
 We first look at the bias term.  From the definition of $\Pi h_0$, we know that 
 \begin{align*}
     \Pi h_0(s,a) = \phi_K(s,a)^\tp ({\Sigma})^{-1}\EE \left\{ \phi_K(S,A) h_0(S,A) \right\} 
 \end{align*}
 Take $C_{K,f} $ as 
\begin{align*}
   C_{K,f} =  ({\Sigma})^{-1}\EE \left\{ \phi_K(S,A) \mathcal{P}^\pi_tf(S,A) \right\} 
\end{align*} and Define
\begin{align*}
    \zeta_K = \sup_{s,a} \| \Sigma^{-1/2} \phi_K(s,a)\|_2.
\end{align*}
 \begin{align*}
     &\phi_K(s,a)^\tp (\hat{\Sigma})^{-1}\left\{ \frac{1}{n}\sum_{i=1}^n \phi_K(S_i, A_i)  \mathcal{P}_{t}^\pi f(S_i, A_i)\right\} - \Pi(\mathcal{P}_{t}^\pi f) (s,a) \\
     = & \phi_K(s,a)^\tp ({\Sigma})^{-1}\left\{ \frac{1}{n}\sum_{i=1}^n \phi_K(S_i, A_i)  \left( \mathcal{P}_{t}^\pi f(S_i, A_i) - \phi_k(S_i, A_i)^\tp C_{K,f} \right) \right. \\
     & \qquad \left. - \EE \left[ \phi_K(S, A)  \left( \mathcal{P}_{t}^\pi f(S, A) - \phi_k(S, A)^\tp C_{K,f}\right)\right]\right\}\\
     & + \phi_K(s,a)^\tp \Sigma^{-1/2} (\Sigma^{1/2} \hat{\Sigma}^{-1} \Sigma^{1/2} - I_K) \Sigma^{-1/2} \left\{ \frac{1}{n}\sum_{i=1}^n \phi_K(S_i, A_i)  \left( \mathcal{P}_{t}^\pi f(S_i, A_i) - \phi_k(S_i, A_i)^\tp C_{K,f} \right) \right. \\
     & \qquad \left. - \EE \left[ \phi_K(S, A)  \left( \mathcal{P}_{t}^\pi f(S, A) - \phi_k(S, A)^\tp C_{K,f}\right)\right]\right\}
 \end{align*}
Take $\mathcal{Q}^{(t+1)}(1) = \{ f\in \mathcal{Q}^{(t+1)}: \|f\|_{\infty} \leq 1 \}$, we have 
 \begin{align*}    
    &  \sup_{f \in \mathcal{Q}^{(t+1)}(1)}  \left\| \phi_K(\cdot, \cdot)^\tp (\hat{\Sigma})^{-1}\left\{ \frac{1}{n}\sum_{i=1}^n \phi_K(S_i, A_i)  \mathcal{P}_{t}^\pi f(S_i, A_i)\right\} - \Pi(\mathcal{P}_{t}^\pi f) \right\|_{\infty} \\
    \leq & \left( \sup_{s,a}\left\| \Sigma^{-1/2} \phi_K(s,a) \right\|_2  + \sup_{s,a}\left\| \Sigma^{-1/2} \phi_K(s,a) \right\|_2  \| \Sigma^{1/2} \hat{\Sigma}^{-1} \Sigma^{1/2} - I_K\| \right)\\
    &
   \sup_{f \in \mathcal{Q}^{(t+1)}(1)}   \left\| \Sigma^{\frac{-1}{2}}\left\{ \frac{1}{n}\sum_{i=1}^n \phi_K(S_i, A_i)  \left( \mathcal{P}_{t}^\pi f(S_i, A_i)  -  \phi_k(S_i, A_i)^\tp C_{K,f} \right) \right. \right. \\
  & \qquad  \qquad  \left. \left. - \EE \left[ \phi_K(S, A)  \left( \mathcal{P}_{t}^\pi f(S, A) - \phi_k(S, A)^\tp C_{K,f}\right)\right]\right\} \right\|_2  \\
    \leq & \zeta_K(1 + \| \Sigma^{1/2} \hat{\Sigma}^{-1} \Sigma^{1/2} - I_K\| )\\
    &
\sup_{ h \in \mathcal{Q}^{(t)}: \|h \|_{\infty} \leq  \sup_{h \in \mathcal{Q}^{(t)}}  \| h - \Pi h\|_{\infty}}   \left\| \Sigma^{\frac{-1}{2}}\left\{ \frac{1}{n}\sum_{i=1}^n \phi_K(S_i, A_i) h(S_i, A_i) - \EE \left[ \phi_K(S, A)  h (S,A)\right]\right\} \right\|_2 
 \end{align*}
 And it remains to bound $\| \Sigma^{1/2} \hat{\Sigma}^{-1} \Sigma^{1/2} - I_K\|$ and 
 \begin{align*}
   \sup_{ h \in \mathcal{Q}{(1)}}   \left\| \Sigma^{\frac{-1}{2}}\left\{ \frac{1}{n}\sum_{i=1}^n \phi_K(S_i, A_i) h(S_i, A_i)   - \EE \left[ \phi_K(S, A)  h (S,A)\right]\right\} \right\|_2.   
 \end{align*}
 Note that under Assumption \ref{ass:covering}, by Theorem 2.7.3 in \citet{van1996weak}, there exists a constant $B>0$ such that  
 $\calN(\mathcal{Q}^{(t)}(1), \|\cdot\|_{\infty}, \epsilon) \leq \exp(B \epsilon^{-d/(p)})$ for $t=1,\dots, T$.
 
 Therefore, by Lemma \ref{lem:l2bound} and \ref{lem:matrix}, with probability at least $1-2(nK)^{-2}$, we have 
\begin{align}
    &  \sup_{f \in \mathcal{Q}^{(t+1)}(1)}  \left\| \phi_K(\cdot, \cdot)^\tp (\hat{\Sigma})^{-1}\left\{ \frac{1}{n}\sum_{i=1}^n \phi_K(S_i, A_i)  \mathcal{P}_{t}^\pi f(S_i, A_i)\right\} - \Pi(\mathcal{P}_{t}^\pi f) \right\|_{\infty} \nonumber \\
    \lesssim &   \frac{\zeta^2_K \sqrt{\log n \log K}}{\sqrt{n}} \sup_{h \in \mathcal{Q}^{(t)}}  \| h - \Pi h\|_{\infty}. \label{eqn:bias_bound1}
\end{align}
And therefore with probability at least $1- 2(nK)^{-2}$, 
\begin{align*}
\sup_{f \in \mathcal{Q}^{(t+1)}}\sup_{s,a}\left|  \phi_K(s,a)^\tp (\hat{\Sigma})^{-1}\left\{ \frac{1}{n}\sum_{i=1}^n \phi_K(S_i, A_i)  \mathcal{P}_{t}^\pi f(S_i, A_i)\right\} - \Pi(\mathcal{P}_{t}^\pi f) (s,a) 
     + \Pi(\mathcal{P}_{t}^\pi f) (s,a) - (\mathcal{P}_{t}^\pi f) (s,a) \right| \\
     \lesssim \left( 1 + \frac{\zeta^2_K \sqrt{\log n \log K}}{\sqrt{n}}\right)  \sup_{h \in \mathcal{Q}^{(t)}(1)}  \| h - \Pi h\|_{\infty}.
\end{align*}
\subsubsection{Bounding the variance term}

\begin{align}
    II(s,a)  =  &\phi_K(s,a)^\tp (\hat{\Sigma})^{-1}\left\{ \frac{1}{n}\sum_{i=1}^n \phi_K(S_i, A_i)  \left[ f^\pi(S_i') -  \mathcal{P}_{t}^\pi f(S_i, A_i) \right]\right\} \nonumber \\
    = & \phi_K(s,a)^\tp {\Sigma}^{-1}\left\{ \frac{1}{n}\sum_{i=1}^n \phi_K(S_i, A_i)  \left[ f^\pi(S_i') -  \mathcal{P}_{t}^\pi f(S_i, A_i) \right]\right\}  \nonumber \\
    & \qquad + \phi_K(s,a)^\tp \Sigma^{-1/2} \left( \Sigma^{1/2} {\Sigma}^{-1} \Sigma^{1/2} - I_K \right) \Sigma^{-1/2} \left\{ \frac{1}{n}\sum_{i=1}^n \phi_K(S_i, A_i)  \left[ f^\pi(S_i') -  \mathcal{P}_{t}^\pi f(S_i, A_i) \right]\right\}.  \label{eqn:var_decomp}
\end{align}
By using the same argument in Lemma \ref{lem:l2bound}, we have with probability at least $1- (nK)^{-2}$, 
\begin{align*}
   \sup_{f \in \mathcal{Q}^{(t+1)}(1)} \left\| \Sigma^{-1/2} \left\{ \frac{1}{n}\sum_{i=1}^n \phi_K(S_i, A_i)  \left[ f^\pi(S_i') -  \mathcal{P}_{t}^\pi f(S_i, A_i) \right]\right\} \right\| \lesssim \frac{\zeta_K \sqrt{\log n \log K}}{\sqrt{n}}.
\end{align*}
Then the second term in \eqref{eqn:var_decomp} can be bounded by 
\begin{align*}
    &\sup_{s,a} \sup_{f \in \mathcal{Q}^{(t+1)}}\left| \phi_K(s,a)^\tp \Sigma^{-1/2} \left( \Sigma^{1/2} {\Sigma}^{-1} \Sigma^{1/2} - I_K \right) \Sigma^{-1/2} \left\{ \frac{1}{n}\sum_{i=1}^n \phi_K(S_i, A_i)  \left[ f^\pi(S_i') -  \mathcal{P}_{t}^\pi f(S_i, A_i) \right]\right\}\right| \\
    \leq &\sup_{s,a}\|\phi^\tp_K(s,a) \Sigma^{-1/2}\|_2 \left\|  \Sigma^{1/2} {\Sigma}^{-1} \Sigma^{1/2} - I_K\right\|  \sup_{f \in \mathcal{Q}^{(t+1)}(1)} \left\| \Sigma^{-1/2} \left\{ \frac{1}{n}\sum_{i=1}^n \phi_K(S_i, A_i)  \left[ f^\pi(S_i') -  \mathcal{P}_{t}^\pi f(S_i, A_i) \right]\right\} \right\| \\
    \leq &  \zeta_K \frac{\zeta_K \sqrt{\log n \log K}}{\sqrt{n}}\frac{\zeta_K \sqrt{\log n \log K}}{\sqrt{n}}
\end{align*}
with probability at least $1-2(nK)^{-2}$. For the following, we focus on bounding the first term in \eqref{eqn:var_decomp}.

Let $\mathcal{X}_n \subset \mathcal{S} \times \mathcal{A}$ be a grid of finitely many points such that for each $(s,a) \in \mathcal{S} \times \mathcal{A}$ there exists a $\overline{(s,a)}_n(s,a) \in \mathcal{X}_n$ such that $\| (s,a) - \overline{(s,a)}_n(s,a)\| \lesssim (\zeta_K K^{-(\omega + 1/2)})^{1/\omega'}$, where $\omega$ and $\omega'$ are the constants defined in Assumption \ref{ass:basis}. By compactness and convexity of the support of $(S,A)$, we may choose $\mathcal{X}_n$ to have cardinality $\#(\mathcal{X}_n) \lesssim n^{\beta}$ for some constant $0 < \beta < \infty$.
\begin{align*}
    &\sup_{s,a}\left| \phi_K(s,a)^\tp {\Sigma}^{-1}\left\{ \frac{1}{n}\sum_{i=1}^n \phi_K(S_i, A_i)  \left[ f^\pi(S_i') -  \mathcal{P}_{t}^\pi f(S_i, A_i) \right]\right\} 
 \right| \\
 \leq & \max_{(s_n, a_n) \in \mathcal{X}_n}\left| \phi_K(s_n,a_n)^\tp {\Sigma}^{-1}\left\{ \frac{1}{n}\sum_{i=1}^n \phi_K(S_i, A_i)  \left[ f^\pi(S_i') -  \mathcal{P}_{t}^\pi f(S_i, A_i) \right]\right\}  \right| \\
 & \qquad +  \sup_{s,a} \left| [\phi_K(s,a) - \phi_K(s_n,a_n)]^\tp {\Sigma}^{-1}\left\{ \frac{1}{n}\sum_{i=1}^n \phi_K(S_i, A_i)  \left[ f^\pi(S_i') -  \mathcal{P}_{t}^\pi f(S_i, A_i) \right]\right\}  \right|  \\
 \leq & \max_{(s_n, a_n) \in \mathcal{X}_n}\left| \phi_K(s_n,a_n)^\tp {\Sigma}^{-1}\left\{ \frac{1}{n}\sum_{i=1}^n \phi_K(S_i, A_i)  \left[ f^\pi(S_i') -  \mathcal{P}_{t}^\pi f(S_i, A_i) \right]\right\}  \right| \\
 & \qquad +  C_{\omega} K^{\omega} (\zeta_K K^{-(\omega + 1/2)}) \left\|\Sigma^{-1/2}  \frac{1}{n}\sum_{i=1}^n \phi_K(S_i, A_i)  \left[ f^\pi(S_i') -  \mathcal{P}_{t}^\pi f(S_i, A_i) \right] \right\|_2.
\end{align*}
The second term can be bounded by using the same argument as that in Lemma \ref{lem:l2bound}. We focus on the first term.

For any fixed $(s_n, a_n)$, from Lemma \ref{lem:point}, we have 
\begin{align*}
        \sup_{f \in \mathcal{Q}^{(t+1)}(1)}\left| \phi_K(s,a)^\tp \Sigma^{-1}\left\{\frac{1}{n}\sum_{i=1}^n \phi_K(S_i, A_i)  \left[ f^\pi(S_i') -  \mathcal{P}_{t}^\pi f(S_i, A_i) \right]\right\} \right|\lesssim \frac{\zeta_K }{\sqrt{n}}
    \end{align*}
with probability at least $1-c\exp\{-c^{-1}\zeta_K^{-1}\sqrt{n}\}$ for some universal constant $c>0$.
It follows by the union bound that 
\begin{align*}
    & \Pr \left( \max_{(s_n, a_n) \in \mathcal{X}_n} \left| \phi_K(s_n,a_n)^\tp \Sigma^{-1}\left\{\frac{1}{n}\sum_{i=1}^n \phi_K(S_i, A_i)  \left[ f^\pi(S_i') -  \mathcal{P}_{t}^\pi f(S_i, A_i) \right]\right\} \right| > \frac{\zeta_K }{\sqrt{n}} \right) \\
    \leq & \mathrm{card}(\mathcal{X}_n) \max_{(s_n, a_n) \in \mathcal{X}_n} \Pr \left(  \left| \phi_K(s_n,a_n)^\tp \Sigma^{-1}\left\{\frac{1}{n}\sum_{i=1}^n \phi_K(S_i, A_i)  \left[ f^\pi(S_i') -  \mathcal{P}_{t}^\pi f(S_i, A_i) \right]\right\} \right| > \frac{\zeta_K }{\sqrt{n}} \right)  \\
    \lesssim & c_1\exp\left\{\beta \log n - c_2\zeta_K^{-1} \sqrt{n} \right\},
\end{align*}
for some universal constants $c_1, c_2 >0$.
Therefore, 
\begin{multline*}
    \sup_{f \in \mathcal{Q}^{(t+1)}} \sup_{s,a} \left| \phi_K(s,a)^\tp (\hat{\Sigma})^{-1}\left\{ \frac{1}{n}\sum_{i=1}^n \phi_K(S_i, A_i)  \left[ f^\pi(S_i') -  \mathcal{P}_{t}^\pi f(S_i, A_i) \right]\right\}\right| \\
  \lesssim \frac{\zeta_K}{\sqrt{n}} +   K^{\frac{-1}{2}} \frac{\zeta^2_K \sqrt{\log n \log K}}{\sqrt{n}}+ \frac{\zeta_K^3 \log n \log K}{n} 
\end{multline*}
with probability at least $1 - 2(nK)^{-1} - c_1\exp\left\{\beta \log n - c_2\zeta_K^{-1} \sqrt{n} \right\}$.

\section{Additional Lemmas}
\begin{lemma}
\label{lem:decompoeision}
The variance of $\calE_1(E_1)$ can be decomposed as 
\begin{align*}
    \Var \left( \mathcal{E}^\pi_1\left\{ \sum_{t=1}^T \left(\prod_{t'=0}^{t-1} \mathcal{P}_{t'}^\pi\right) \tilde{\Pi}_t \left[ Q^\pi_t - (R_t + \langle \pi_t, Q_{t+1}^\pi \rangle) \right]\right\} \right) \\
= \sum_{t=1}^T \EE \left\{ \Var \left(   \mathcal{E}^\pi_1 \left\{ \left(\prod_{t'=0}^{t-1} \mathcal{P}_{t'}^\pi\right) \tilde{\Pi}_t \left[ Q^\pi_t - (R_t + \langle \pi_t, Q_{t+1}^\pi \rangle) \right] \right\}  \mid \mathcal{D}_t\right) \right\}
\end{align*}
\end{lemma}
\begin{proof}
    By  iteratively applying law of total variance, we have 
    \begin{align*}
       & \Var \left( \mathcal{E}^\pi_1\left\{ \sum_{t=1}^T \left(\prod_{t'=0}^{t-1} \mathcal{P}_{t'}^\pi\right) \tilde{\Pi}_t \left[ Q^\pi_t - (R_t + \langle \pi_t, Q_{t+1}^\pi \rangle) \right]\right\} \right) \\
        = & \EE \left( \Var \left(  \mathcal{E}^\pi_1\left\{ \sum_{t=1}^T \left(\prod_{t'=0}^{t-1} \mathcal{P}_{t'}^\pi\right) \tilde{\Pi}_t \left[ Q^\pi_t - (R_t + \langle \pi_t, Q_{t+1}^\pi \rangle) \right]\right\} \mid \calD_T\right) \right) \\
        & \qquad + \Var \left[ \EE \left(  \mathcal{E}^\pi_1\left\{ \sum_{t=1}^T \left(\prod_{t'=0}^{t-1} \mathcal{P}_{t'}^\pi\right) \tilde{\Pi}_t \left[ Q^\pi_t - (R_t + \langle \pi_t, Q_{t+1}^\pi \rangle) \right]\right\} \mid \calD_T \right) \right] \\
         = &  \EE \left( \Var \left(  \mathcal{E}^\pi_1\left\{ \left(\prod_{t'=0}^{T-1} \mathcal{P}_{t'}^\pi\right) \tilde{\Pi}_T \left[ Q^\pi_T - (R_T + \langle \pi_t, Q_{T+1}^\pi \rangle) \right]\right\} \mid \calD_T\right) \right) \\
         & \qquad + \Var \left[  \mathcal{E}^\pi_1\left\{ \sum_{t=1}^{T-1} \left(\prod_{t'=0}^{t-1} \mathcal{P}_{t'}^\pi\right) \tilde{\Pi}_t \left[ Q^\pi_t - (R_t + \langle \pi_t, Q_{t+1}^\pi \rangle) \right]\right\} \right] \\
         = &  \EE \left( \Var \left(  \mathcal{E}^\pi_1\left\{ \left(\prod_{t'=0}^{T-1} \mathcal{P}_{t'}^\pi\right) \tilde{\Pi}_T \left[ Q^\pi_T - (R_T + \langle \pi_t, Q_{T+1}^\pi \rangle) \right]\right\} \mid \calD_T\right) \right) \\
         & \qquad + \EE \left( \Var \left(  \mathcal{E}^\pi_1\left\{ \sum_{t=1}^{T-1} \left(\prod_{t'=0}^{t-1} \mathcal{P}_{t'}^\pi\right) \tilde{\Pi}_t \left[ Q^\pi_t - (R_t + \langle \pi_t, Q_{t+1}^\pi \rangle) \right]\right\} \mid \calD_{T-1}\right) \right) \\
         & \qquad + \Var \left[ \EE \left(  \mathcal{E}^\pi_1\left\{ \sum_{t=1}^{T-1} \left(\prod_{t'=0}^{t-1} \mathcal{P}_{t'}^\pi\right) \tilde{\Pi}_t \left[ Q^\pi_t - (R_t + \langle \pi_t, Q_{t+1}^\pi \rangle) \right]\right\} \mid \calD_{T-1} \right) \right] \\
         = & \EE \left( \Var \left(  \mathcal{E}^\pi_1\left\{ \left(\prod_{t'=0}^{T-1} \mathcal{P}_{t'}^\pi\right) \tilde{\Pi}_T \left[ Q^\pi_T - (R_T + \langle \pi_t, Q_{T+1}^\pi \rangle) \right]\right\} \mid \calD_T\right) \right) \\
        &  \qquad +\EE \left( \Var \left(  \mathcal{E}^\pi_1\left\{ \left(\prod_{t'=0}^{T-2} \mathcal{P}_{t'}^\pi\right) \tilde{\Pi}_{T-1} \left[ Q^\pi_{T-1} - (R_{T-1} + \langle \pi_t, Q_{T}^\pi \rangle) \right]\right\} \mid \calD_{T-1}\right) \right) \\
        & \qquad +  \Var \left[ \EE \left(  \mathcal{E}^\pi_1\left\{ \sum_{t=1}^{T-2} \left(\prod_{t'=0}^{t-1} \mathcal{P}_{t'}^\pi\right) \tilde{\Pi}_t \left[ Q^\pi_t - (R_t + \langle \pi_t, Q_{t+1}^\pi \rangle) \right]\right\} \mid \calD_{T-2} \right) \right] \\
        = & \cdots \\
        = &\sum_{t=1}^T \EE \left\{ \Var \left(   \mathcal{E}^\pi_1 \left\{ \left(\prod_{t'=0}^{t-1} \mathcal{P}_{t'}^\pi\right) \tilde{\Pi}_t \left[ Q^\pi_t - (R_t + \langle \pi_t, Q_{t+1}^\pi \rangle) \right] \right\}  \mid \mathcal{D}_t\right) \right\}.
    \end{align*}
\end{proof}

\begin{lemma}
    \label{lem:matrix}
    With probability at least $1- (nK)^{-2}$, we have
    \begin{align}
    \label{eqn:mat_con1}
        \left\|  \Sigma^{-1/2}  \hat{\Sigma}  \Sigma^{-1/2} - I_K \right\| \lesssim \frac{\zeta_K \sqrt{\log n \log K \log \zeta_K }}{\sqrt{n}} + \frac{\zeta_K^2\log n \log K \log \zeta_K }{n}.
    \end{align}
    If we further assume that $\zeta^2_K \log n \log K \log \zeta_K = \smallO(n) $,
        \begin{align}
        \label{eqn:mat_con2}
            \left\|  \Sigma^{1/2}  \hat{\Sigma}^{-1}  \Sigma^{1/2} - I_K \right\| \lesssim \frac{\zeta_K \sqrt{\log n \log K \log \zeta_K }}{\sqrt{n}} 
        \end{align}
    \end{lemma}
    \begin{proof}
        By applying Lemma B.5 in \cite{duan2020minimax}, we obtain \eqref{eqn:mat_con1}. Next, we condition on the event that $\eqref{eqn:mat_con1}$ holds, under the condition that $\zeta^2_K \log n \log K \log \zeta_K = \smallO(n) $, we have 
        \begin{align*}
              \left\|  \Sigma^{-1/2}  \hat{\Sigma}  \Sigma^{-1/2} - I_K \right\| \leq \frac{1}{2}  = \frac{1}{2} \lambda_{\min}(I_K).
        \end{align*}
        Then we apply Lemma F.4 in \cite{chen2018optimal} and obtain 
        \begin{align*}
               \left\|  \Sigma^{1/2}  \hat{\Sigma}^{-1}  \Sigma^{1/2} - I_K \right\| \leq 2(1+\sqrt{5}) [\lambda_{\min}(I_K)]^{-2}  \left\|  \Sigma^{1/2}  \hat{\Sigma}^{-1}  \Sigma^{1/2} - I_K \right\|  \lesssim \frac{\zeta_K \sqrt{\log n \log K \log \zeta_K }}{\sqrt{n}} .
        \end{align*}
    \end{proof}

    \begin{lemma}
        \label{lem:l2bound}
        Suppose the covering number of space $\mathcal{Q}(1)$ satisfies that  $N(\mathcal{Q}(1), \|\cdot\|_{\infty}, \epsilon) \leq \exp(A \epsilon^{-\alpha}) $ for some constant $A>0$ and $\alpha < 2$. Then with  probability at least $1-(nK)^{-2}$, we have 
        \begin{align*}
            \sup_{h \in \mathcal{Q}(1)}\left\| \Sigma^{\frac{-1}{2}}\left\{ \frac{1}{n}\sum_{i=1}^n \phi_K(S_i, A_i) h(S_i, A_i)   - \EE \left[ \phi_K(S, A)  h (S,A)\right]\right\} \right\|_2 \leq  C\frac{\zeta_K \sqrt{\log N \log K}}{\sqrt{n}},
        \end{align*}
        where the constant $C$ depends on $A$ and $\alpha$.
    \end{lemma}
    \begin{proof}
        Take $r_i$, $i = 1,\dots, n$ as independent Rademacher random variables. Then by symmetrization inequality, we have
        \begin{multline*}
           \EE \sup_{ h \in \mathcal{Q}{(1)}}   \left\| \Sigma^{\frac{-1}{2}}\left\{ \frac{1}{n}\sum_{i=1}^n \phi_K(S_i, A_i) h(S_i, A_i)   - \EE \left[ \phi_K(S, A)  h (S,A)\right]\right\} \right\|^2_2 \\
           \lesssim \EE \sup_{ h \in \mathcal{Q}{(1)}}   \left\| \Sigma^{\frac{-1}{2}}\left\{ \frac{1}{n}\sum_{i=1}^n \phi_K(S_i, A_i) h(S_i, A_i)r_i \right\} \right\|^2_2.
        \end{multline*} 
        \begin{align*}
     &  \EE \sup_{ h \in \mathcal{Q}{(1)}}   \left\| \Sigma^{\frac{-1}{2}}\left\{ \frac{1}{n}\sum_{i=1}^n \phi_K(S_i, A_i) h(S_i, A_i)r_i \right\} \right\|^2_2    \\
       \leq & \frac{1}{n^2}  \sum_{i=1}^n  \EE \sup_{ h \in \mathcal{Q}^{(1)}} \|\Sigma^{-1/2} \phi_K(S_i, A_i) h(S_i, A_i) r_i \|^2_{2} \\
       & + \frac{1}{n^2} \sum_{i\neq j}\EE \sup_{ h \in \mathcal{Q}{(1)}} r_i r_jh(S_i, A_i) h(S_j, A_j) \phi_K(S_i, A_i)^\tp \Sigma^{-1} \phi_K(S_j, A_j) \\
       \leq & \frac{1}{n}  \zeta_K^2 + \frac{1}{n^2} \sum_{i\neq j}\EE \sup_{ h \in \mathcal{Q}^{(1)}} r_i r_jh(S_i, A_i) h(S_j, A_j) \phi_K(S_i, A_i)^\tp \Sigma^{-1} \phi_K(S_j, A_j) 
        \end{align*}
    
        Next, we focus on bounding  $$\frac{1}{n^2} \sum_{i\neq j}\EE \sup_{ h \in \mathcal{Q}^{(1)}} r_i r_jh(S_i, A_i) h(S_j, A_j) \phi_K(S_i, A_i)^\tp \Sigma^{-1} \phi_K(S_j, A_j) .$$
    As $r_i$, $i=1,\dots, n$ are Rademacher random variables, then there exists a constant $\sigma^2$ such that $\EE \exp(\lambda r_ir_j) \leq \exp(\lambda^2 \sigma^2/2)$ for any $\lambda>0$.
    Conditioned on $(S_i, A_i)$, $i = 1,\dots, n$, then 
    $$r_ir_j[ h_1(S_i, A_i) h_1(S_j, A_j) -h_2(S_i, A_i) h_2(S_j, A_j)] \phi_K(S_i, A_i)^\tp \Sigma^{-1} \phi_K(S_j, A_j)$$ is a subgaussian random variable with parameter $$[ h_1(S_i, A_i) h_1(S_j, A_j) -h_2(S_i, A_i) h_2(S_j, A_j)]^2 \sigma^2.$$
    Define 
    $$ \mathcal{S}(h) =   \frac{1}{n} \sum_{i\neq j}  r_i r_jh(S_i, A_i) h(S_j, A_j) \phi_K(S_i, A_i)^\tp \Sigma^{-1} \phi_K(S_j, A_j) .$$
    We know that $r_ir_j$ is independent of $r_{i'} r_{j'}$ as long as either $i\neq i'$ or $j \neq j'$, then conditioned on $(S_i, A_i)$, $i = 1,\dots, n$,  $S(h_1) - S(h_2)$ is a subgaussian with parameter
    \begin{align*}
        d^2(h_1, h_2) = \frac{1}{n^2} \sum_{i\neq j} [ h_1(S_i, A_i) h_1(S_j, A_j) -h_2(S_i, A_i) h_2(S_j, A_j)]^2 \sigma^2.
    \end{align*}
    
    Next, we derive $H^{1/2}(\mathcal{Q}(1), d, \epsilon)$. We know that the covering number $N(\mathcal{Q}, \|\cdot\|_{\infty}, \epsilon) \leq \exp(A \epsilon^{-\alpha})$. Consider $\mathcal{N} \subset \mathcal{Q}(1)$  as the $\epsilon$-net of $\mathcal{Q}(1)$
    with respect to $\|\cdot\|_\infty$. 
     By definition, for any $h \in \mathcal{Q}(1)$,  there exists a $u_0\in \mathcal{N}$, such that
        \begin{align}
            \sup_{(s,a)} |h(s,a)- h_0(s,a)| \leq \epsilon.
        \end{align}
     Then 
     \begin{align*}
         d^2(h, h_0) = & \frac{1}{n^2} \sum_{i\neq j} [ h(S_i, A_i) h(S_j, A_j) -h_0(S_i, A_i) h_0(S_j, A_j) \phi_K(S_i, A_i)^\tp \Sigma^{-1} \phi_K(S_j, A_j)]^2 \sigma^2  \\
         \leq & \frac{1}{n^2}\sigma^2  \sup_{s,a} \| \Sigma^{-1/2} \phi_K(s,a)\|^4_2  \sum_{i\neq j} [ h(S_i, A_i) h(S_j, A_j) -h_0(S_i, A_i) h_0(S_j, A_j)]^2  \\
         \leq & \frac{1}{n^2}\sigma^2 \zeta_K^4 \sum_{i\neq j} \{h(S_i, A_i)(h(S_j, A_j) - h_0(S_j, A_j)) + h_0(S_j, A_j)(h(S_i, A_i) - h_0(S_i, A_i) \}^2 \\
         \leq & \frac{4}{n^2}\sigma^2 \zeta_K^4 \sum_{i\neq j} (2\epsilon)^2 \leq 4\sigma^2 \zeta_K^4 \epsilon^2.
     \end{align*}
     Therefore we have
     \begin{align*}
         N(\mathcal{Q}(1), d, 2\sigma \zeta_K^2 \epsilon) \leq N(\mathcal{Q}(1), \|\cdot\|_{\infty}, \epsilon) \\
         N(\mathcal{Q}(1), d, \epsilon) \leq N(\mathcal{Q}(1), \|\cdot\|_{\infty}, \epsilon/(2\sigma \zeta_K^2)) \leq \exp\left\{ A\left(\frac{\epsilon}{2\sigma \zeta_K^2}\right)^{-\alpha} \right\}.
     \end{align*}
    
    Following Dudley’s entropy bounds, we have
    \begin{align*}
        \EE \sup_{h \in \mathcal{Q}(1)} \mathcal{S}(h) &\lesssim \EE \int_{0}^D H^{1/2}(\mathcal{Q}(1), d, \epsilon) d \epsilon\\
       & \leq \int_{0}^{\zeta_K^2} \left\{ A\left(\frac{\epsilon}{2\sigma \zeta_K^2}\right)^{-\alpha} \right\}^{1/2} d\epsilon  \leq C \zeta_K^2,
    \end{align*}
     where $D = \sup_{h \in \mathcal{Q}^{(1)}} \sqrt{\frac{1}{n^2} \sum_{i\neq j} [h(S_i, A_i) h(S_j, A_j) \phi_K(S_i, A_i)^\tp \Sigma^{-1} \phi_K(S_j, A_j)]^2 } \leq \sup_{s,a} \| \Sigma^{-1/2} \phi_K(s,a)\|^2_2 \leq \zeta_K^2$, $C$ is a constant depending on $A$ and $\alpha$.
    
    Therefore we have 
    \begin{align*}
     & \frac{1}{n^2} \sum_{i\neq j}\EE \sup_{ h \in \mathcal{Q}^{(1)}} r_i r_jh(S_i, A_i) h(S_j, A_j) \phi_K(S_i, A_i)^\tp \Sigma^{-1} \phi_K(S_j, A_j) \lesssim \frac{1}{n} \zeta_K^2 \\
      & \EE \sup_{ h \in \mathcal{Q}{(1)}}   \left\| \Sigma^{\frac{-1}{2}}\left\{ \frac{1}{n}\sum_{i=1}^n \phi_K(S_i, A_i) h(S_i, A_i)   - \EE \left[ \phi_K(S, A)  h (S,A)\right]\right\} \right\|_2 \\
    \leq & \left[\EE \sup_{ h \in \mathcal{Q}{(1)}}   \left\| \Sigma^{\frac{-1}{2}}\left\{ \frac{1}{n}\sum_{i=1}^n \phi_K(S_i, A_i) h(S_i, A_i)   - \EE \left[ \phi_K(S, A)  h (S,A)\right]\right\} \right\|^2_2 \right]^{1/2} \\
      \lesssim &  \left[\EE \sup_{ h \in \mathcal{Q}{(1)}}   \left\| \Sigma^{\frac{-1}{2}}\left\{ \frac{1}{n}\sum_{i=1}^n \phi_K(S_i, A_i) h(S_i, A_i)r_i \right\} \right\|^2_2 \right]^{1/2} \lesssim \frac{1}{\sqrt{n}} \zeta_K.
    \end{align*}
    
    Next, take $X_i = (S_i, A_i)$ and  define 
    \begin{align*}
        g(X_1,\dots, X_n) =  \sup_{h \in \mathcal{Q}(1)}\left\| \Sigma^{\frac{-1}{2}}\left\{ \frac{1}{n}\sum_{i=1}^n \phi_K(S_i, A_i) h(S_i, A_i)   - \EE \left[ \phi_K(S, A)  h (S,A)\right]\right\} \right\|_2
    \end{align*}
    We have
    {\scriptsize
    \begin{align*}
       &  |g(X_1,\dots, X_j, \dots, X_n) - g(X_1,\dots, X'_{j}, \dots, X_n)| \\
         \leq &  \sup_{h \in \mathcal{Q}(1)} \left| \left\| \Sigma^{\frac{-1}{2}}\left\{ \frac{1}{n}\sum_{i\neq j} \phi_K(S_i, A_i) h(S_i, A_i)   - \EE \left[ \phi_K(S, A)  h (S,A)\right] + \frac{1}{n} \phi_K(S_j, A_j) h(S_j, A_j)   - \EE \left[ \phi_K(S, A)  h (S,A)\right] \right\} \right\|_2 \right.\\
         & \qquad \left. -\left\| \Sigma^{\frac{-1}{2}}\left\{ \frac{1}{n}\sum_{i\neq j} \phi_K(S_i, A_i) h(S_i, A_i)   - \EE \left[ \phi_K(S, A)  h (S,A)\right]+ \frac{1}{n} \phi_K(S'_j, A'_j) h(S'_j, A'_j)   - \EE \left[ \phi_K(S, A)  h (S,A)\right]   \right\} \right\|_2 \right| \\
         \leq &  \sup_{h \in \mathcal{Q}(1)} \left\| \Sigma^{\frac{-1}{2}} \frac{1}{n}  \left[ \phi_K(S_j, A_j) h(S_j, A_j)   -\phi_K(S'_j, A'_j) h(S'_j, A'_j)   \right] \right\|_2 \\
         \leq & \frac{1}{n} \sup_{s,a}\|\Sigma^{-1/2} \phi_K(s,a)\|_2 \leq \frac{1}{n} \zeta_K,
    \end{align*}}
    where the first and the second inequality is due to the triangle inequality for $\sup$ and $\|\cdot\|_2$.
    Then we are able to use the bounded difference inequality and we have
    \begin{align*}
      &  \Pr \left( \sup_{h \in \mathcal{Q}(1)}\left\| \Sigma^{\frac{-1}{2}}\left\{ \frac{1}{n}\sum_{i=1}^n \phi_K(S_i, A_i) h(S_i, A_i)   - \EE \left[ \phi_K(S, A)  h (S,A)\right]\right\} \right\|_2 \right. \\
      & \qquad  \left. - \EE \sup_{h \in \mathcal{Q}(1)}\left\| \Sigma^{\frac{-1}{2}}\left\{ \frac{1}{n}\sum_{i=1}^n \phi_K(S_i, A_i) h(S_i, A_i)   - \EE \left[ \phi_K(S, A)  h (S,A)\right]\right\} \right\|_2 \ge t \right) \\
      \leq &  \exp \left\{ -\frac{2t^2}{n (\zeta_K/n)^2}\right\} =  \exp \left\{ -\frac{2nt^2}{\zeta^2_K}\right\}.
    \end{align*}
     By taking $t = \frac{  \zeta_K \sqrt{\log n \log K}}{\sqrt{n}}$, we obtain the result. 
    \end{proof}

    \begin{lemma}
        \label{lem:point}
        Suppose the covering number of space $\mathcal{Q}^{(t+1)}(1)$ satisfies that  $N(\mathcal{Q}^{(t+1)}(1), \|\cdot\|_{\infty}, \epsilon) \leq \exp(A \epsilon^{-\alpha}) $ for some constant $A>0$ and $\alpha \leq 2$. And we assume that $\zeta_K = \bigO(\sqrt{n})$, then for any $(s,a)$, %with  probability at least $1-c \exp\{-2 c^{-1} \sqrt{n} \zeta_K^{-1}\}$ with some universal constant $C>0$, we have 
        we have 
        % \begin{align*}
            % \sup_{f \in \mathcal{Q}^{(t+1)}(1)}\left| \phi_K(s,a)^\tp \Sigma^{-1}\left\{\frac{1}{n}\sum_{i=1}^n \phi_K(S_i, A_i)  \left[ f^\pi(S_i') -  \mathcal{P}_{t}^\pi f(S_i, A_i) \right]\right\} \right| \leq  C\frac{\sqrt{\|\Sigma^{-1/2}\phi_K(s,a)\|_2 \zeta_K} }{\sqrt{n}},
        % \end{align*}
        \begin{align*}
            \left|  \sup_{f \in \mathcal{Q}^{(t+1)}(1)} \phi_K(s,a)^\tp \Sigma^{-1}\left\{\frac{1}{n}\sum_{i=1}^n \phi_K(S_i, A_i)  \left[ f^\pi(S_i') -  \mathcal{P}_{t}^\pi f(S_i, A_i) \right]\right\} \right| = \bigO\left( \frac{\|\Sigma^{-1/2}\phi_K(s,a)\|_2 \sqrt{\log n \log T}}{\sqrt{n}} \right), \text{ W.H.P}.
          \end{align*}
          uniformly holds for all $t=1,\dots, T$.
        % where the constant $C$ depends on $A$ and $\alpha$.
    \end{lemma}
    
    \begin{proof}
     Take $r_i$, $i = 1,\dots, n$ as independent Rademacher random variables. Then by symmetrization inequality, we have
        \begin{align*}
          &  \EE \sup_{f \in \mathcal{Q}^{(t+1)}(1)}\left| \phi_K(s,a)^\tp \Sigma^{-1}\left\{\frac{1}{n}\sum_{i=1}^n \phi_K(S_i, A_i)  \left[ f^\pi(S_i') -  \mathcal{P}_{t}^\pi f(S_i, A_i) \right]\right\} \right|  \\
           \lesssim & \EE \sup_{f \in \mathcal{Q}^{(t+1)}(1)}\left| \phi_K(s,a)^\tp \Sigma^{-1}\left\{\frac{1}{n}\sum_{i=1}^n \phi_K(S_i, A_i)  r_i f^\pi(S_i')\right\} \right|.
        \end{align*}   
    Define 
    $$ \mathcal{S}(f) =    \sqrt{n} \phi_K(s,a)^\tp \Sigma^{-1}\left\{\frac{1}{n}\sum_{i=1}^n \phi_K(S_i, A_i)  r_i f^\pi(S_i') \right\}.$$
    Then conditioned on $(S_i, A_i, S'_i)$, $i = 1,\dots, n$,  $S(f_1 - f_2)$ is a subgaussian process with parameter
    \begin{align*}
        d^2(f_1, f_2) = \frac{1}{n} \sum_{i = 1}^n  [\phi_K(s,a)^\tp \Sigma^{-1}\phi_K(S_i, A_i)]^2 [f_1^\pi(S'_i) - f_2^\pi(S'_i)]^2 .
    \end{align*}
    
    Next, we derive $H^{1/2}(\mathcal{Q}^{(t+1)}(1), d, \epsilon)$. We know that the covering number $N(\mathcal{Q}^{(t+1)}(1), \|\cdot\|_{\infty}, \epsilon) \leq \exp(A \epsilon^{-\alpha})$. Consider $\mathcal{N} \subset \mathcal{Q}^{(t+1)}(1)$  as the $\epsilon$-net of $\mathcal{Q}^{(t+1)}(1)$
    with respect to $\|\cdot\|_\infty$. 
     By definition, for any $f \in \mathcal{Q}^{(t+1)}(1)$,  there exists a $f_0\in \mathcal{N}$, such that
        \begin{align}
            \sup_{(s,a)} |f(s,a)- f_0(s,a)| \leq \epsilon.
        \end{align}
     Then 
     \begin{align*}
        & d^2(f, f_0) = \frac{1}{n} \sum_{i = 1}^n  [\phi_K(s,a)^\tp \Sigma^{-1}\phi_K(S_i, A_i)]^2 [f^\pi(S'_i) - f^\pi_0(S'_i)]^2  \\
         \leq & \frac{1}{n} \sum_{i = 1}^n  [\phi_K(s,a)^\tp \Sigma^{-1}\phi_K(S_i, A_i)]^2 \epsilon^2 \\
         = & [\phi_K(s,a)^\tp \Sigma^{-1} \hat{\Sigma} \Sigma^{-1} \phi_K(s,a)] \epsilon^2.
     \end{align*}
     Therefore we have
     \begin{align*}
         & N(\mathcal{Q}^{(t+1)}(1), d, \sqrt{\phi_K(s,a)^\tp \Sigma^{-1} \hat{\Sigma} \Sigma^{-1} \phi_K(s,a)}  \epsilon)   \leq N(\mathcal{Q}^{(t+1)}(1), \|\cdot\|_{\infty}, \epsilon) \\
        & N(\mathcal{Q}(1), d, \epsilon)\\
        & \leq N(\mathcal{Q}(1), \|\cdot\|_{\infty}, \epsilon/\sqrt{\phi_K(s,a)^\tp \Sigma^{-1} \hat{\Sigma} \Sigma^{-1} \phi_K(s,a)} ) \leq \exp\left\{ A\left(\frac{\epsilon}{\sqrt{\phi_K(s,a)^\tp \Sigma^{-1} \hat{\Sigma} \Sigma^{-1} \phi_K(s,a)} }\right)^{-\alpha} \right\}.
     \end{align*}
    Following Dudley’s entropy bounds, we have
    \begin{align*}
        \EE \sup_{f \in \mathcal{Q}^{(t+1)}(1)}  \mathcal{S}(f) &\lesssim \EE \int_{0}^D H^{1/2}(\mathcal{Q}^{(t+1)}(1), d, \epsilon) d \epsilon\\
       & \leq \EE \int_{0}^{\sqrt{\phi_K(s,a)^\tp \Sigma^{-1} \hat{\Sigma} \Sigma^{-1} \phi_K(s,a)}} \left\{ A\left(\frac{\epsilon}{\sqrt{\phi_K(s,a)^\tp \Sigma^{-1} \hat{\Sigma} \Sigma^{-1} \phi_K(s,a)}}\right)^{-\alpha} \right\}^{1/2} d\epsilon  \\
       & \leq C \EE \sqrt{\phi_K(s,a)^\tp \Sigma^{-1} \hat{\Sigma} \Sigma^{-1} \phi_K(s,a)} \\
       & \leq C \sqrt{\EE [\phi_K(s,a)^\tp \Sigma^{-1} \hat{\Sigma} \Sigma^{-1} \phi_K(s,a)]} \\
       & = C \|\Sigma^{-1/2}\phi_K(s,a)\|_2. % \leq C \zeta_K.
    \end{align*}
     where  $C$ is a constant depending on $A$ and $\alpha$.
     
    Next, we apply the Talagrand concentration inequality.
    \begin{align*}
     \sup_{f \in \mathcal{Q}^{(t+1)}(1)} \left| \phi_K(s,a)^\tp \Sigma^{-1}\left\{\phi_K(S_i, A_i)  \left[ f^\pi(S_i') -  \mathcal{P}_{t}^\pi f(S_i, A_i) \right]\right\} \right|& \leq  2 \|\Sigma^{-1/2}\phi_K(s,a)\|_2 \zeta_K. % \sup_{s,a} \|\phi_K(s,a)^\tp \Sigma^{-1/2}\|_2^2 2 %\leq  2 \zeta_K^2.\\
     \\
     \sup_{f \in \mathcal{Q}^{(t+1)}(1)} \EE \left| \phi_K(s,a)^\tp \Sigma^{-1}\left\{\phi_K(S_i, A_i)  \left[ f^\pi(S_i') -  \mathcal{P}_{t}^\pi f(S_i, A_i) \right]\right\} \right|^2 &\leq 2 \EE \phi_K(s,a)^\tp \Sigma^{-1} \hat{\Sigma} \Sigma^{-1} \phi_K(s,a)  \\
     \leq 2 \|\Sigma^{-1/2}\phi_K(s,a)\|_2^2%\leq 2\zeta_K^2.
    \end{align*}
    Then we take 
    \begin{align*}
        U & = 2 \|\Sigma^{-1/2}\phi_K(s,a)\|_2 \zeta_K\\
        % V & = n (2\zeta_K^2) + 8 U \sqrt{n} \zeta_K.
        V & = n (2\|\Sigma^{-1/2}\phi_K(s,a)\|_2^2) + 8 U \sqrt{n} \|\Sigma^{-1/2}\phi_K(s,a)\|_2.
    \end{align*}
       There exists a universal constant $c>0$ such that for every $t>0$,
       \begin{align*}
          & \Pr \left( \left|  \left|  \sup_{f \in \mathcal{Q}^{(t+1)}(1)} \phi_K(s,a)^\tp \Sigma^{-1}\left\{\frac{1}{n}\sum_{i=1}^n \phi_K(S_i, A_i)  \left[ f^\pi(S_i') -  \mathcal{P}_{t}^\pi f(S_i, A_i) \right]\right\} \right| \right. \right. \\
         & \qquad  \left. \left. - \EE  \left| \sup_{f \in \mathcal{Q}^{(t+1)}(1)} \phi_K(s,a)^\tp \Sigma^{-1}\left\{\frac{1}{n}\sum_{i=1}^n \phi_K(S_i, A_i)  \left[ f^\pi(S_i') -  \mathcal{P}_{t}^\pi f(S_i, A_i) \right]\right\} \right| \right| > t \right) \\
         \leq & c \exp\left\{ -\frac{1}{c} \frac{nt}{2 \|\Sigma^{-1/2}\phi_K(s,a)\|_2 \zeta_K} \log\left(1 + \frac{nt2 \|\Sigma^{-1/2}\phi_K(s,a)\|_2 \zeta_K}{n (2\|\Sigma^{-1/2}\phi_K(s,a)\|_2^2) + 8 U \sqrt{n} \|\Sigma^{-1/2}\phi_K(s,a)\|_2}\right)\right\} \\
        %  = &  c \exp\left\{ -\frac{1}{c} \frac{nt}{2 \zeta_K^2} \log\left(1 + \frac{2nt\zeta_K^2}{2n\zeta_K^2 + 16\zeta_K^3 \sqrt{n}}\right)\right\}  \\
         \lesssim & c \exp\left\{ -\frac{1}{c} \frac{nt}{2 \|\Sigma^{-1/2}\phi_K(s,a)\|_2 \zeta_K} \log\left(1 + \frac{2nt\|\Sigma^{-1/2}\phi_K(s,a)\|_2 \zeta_K}{4n \|\Sigma^{-1/2}\phi_K(s,a)\|_2^2 }\right)\right\} \\
         \lesssim & c \exp\left\{ -\frac{1}{c} \frac{nt}{2  \|\Sigma^{-1/2}\phi_K(s,a)\|_2 \zeta_K} \log\left(1 + t/2\right)\right\}, 
       \end{align*}
      where the second inequality is due to the condition that $\zeta_K = \smallO(\sqrt{n})$. By taking $t = \frac{{\|\Sigma^{-1/2}\phi_K(s,a)\|_2 \sqrt{\log n \log T} }}{\sqrt{n}}$, we obtain 
      \begin{align*}
        & \Pr \left( \left|  \left|  \sup_{f \in \mathcal{Q}^{(t+1)}(1)} \phi_K(s,a)^\tp \Sigma^{-1}\left\{\frac{1}{n}\sum_{i=1}^n \phi_K(S_i, A_i)  \left[ f^\pi(S_i') -  \mathcal{P}_{t}^\pi f(S_i, A_i) \right]\right\} \right| \right. \right. \\
         & \qquad  \left. \left. - \EE  \left| \sup_{f \in \mathcal{Q}^{(t+1)}(1)} \phi_K(s,a)^\tp \Sigma^{-1}\left\{\frac{1}{n}\sum_{i=1}^n \phi_K(S_i, A_i)  \left[ f^\pi(S_i') -  \mathcal{P}_{t}^\pi f(S_i, A_i) \right]\right\} \right| \right| >   \frac{\|\Sigma^{-1/2}\phi_K(s,a)\|_2 \sqrt{\log n \log T}}{\sqrt{n}}  \right)\\
         & \lesssim c\exp \left\{ -\frac{1}{c} \frac{\sqrt{n} \sqrt{\log n \log T} }{\zeta_K} \right\}. \\
      \end{align*}
      Take a union bound over $T$, with the condition that $\zeta_K = \smallO(n)$, we have 
      \begin{align*}
        \left|  \sup_{f \in \mathcal{Q}^{(t+1)}(1)} \phi_K(s,a)^\tp \Sigma^{-1}\left\{\frac{1}{n}\sum_{i=1}^n \phi_K(S_i, A_i)  \left[ f^\pi(S_i') -  \mathcal{P}_{t}^\pi f(S_i, A_i) \right]\right\} \right| \\
        = \bigO\left( \frac{\|\Sigma^{-1/2}\phi_K(s,a)\|_2 \sqrt{\log n \log T}}{\sqrt{n}} \right), \text{ W.H.P}.
      \end{align*}
      uniformly holds for all $t=1,\dots, T$.
    \end{proof}

% You can have as much text here as you want. The main body must be at most $8$ pages long.
% For the final version, one more page can be added.
% If you want, you can use an appendix like this one.  

% The $\mathtt{\backslash onecolumn}$ command above can be kept in place if you prefer a one-column appendix, or can be removed if you prefer a two-column appendix.  Apart from this possible change, the style (font size, spacing, margins, page numbering, etc.) should be kept the same as the main body.
%%%%%%%%%%%%%%%%%%%%%%%%%%%%%%%%%%%%%%%%%%%%%%%%%%%%%%%%%%%%%%%%%%%%%%%%%%%%%%%
%%%%%%%%%%%%%%%%%%%%%%%%%%%%%%%%%%%%%%%%%%%%%%%%%%%%%%%%%%%%%%%%%%%%%%%%%%%%%%%

\end{document}